\documentclass{amsart}

\RequirePackage{amsthm,amsmath,amsfonts,amssymb}
\RequirePackage[colorlinks,allcolors=blue]{hyperref}%% uncomment this for coloring bibliography citations and linked URLs
\RequirePackage{graphicx}%% uncomment this for including figures
\RequirePackage[numbers]{natbib}

\usepackage[dvipsnames]{xcolor}
\usepackage{enumitem}
\usepackage{subcaption}
\usepackage{bbm}
\usepackage{breqn}
\usepackage{hyperref}
\usepackage[foot]{amsaddr}
\usepackage{booktabs}
\usepackage{graphicx}
\usepackage{array}
\usepackage{tabularx}

\usepackage[a4paper, total={7in, 10in}]{geometry}
\calclayout

\newcolumntype{N}{>{\centering\arraybackslash}m{.4in}}
\newcolumntype{G}{>{\centering\arraybackslash}m{2in}}

\theoremstyle{plain}
\newtheorem{lemma}{Lemma}[section]
\newtheorem{theorem}{Theorem}[section]

\theoremstyle{remark}

\newtheorem{definition}{Definition}[section]
\newtheorem{remark}{Remark}[section]
\newcommand{\innerthmname}{}% initialize

\newtheorem{manualtheoreminner}{Assumption}
\newenvironment{manualtheorem}[1]{%
  \renewcommand\themanualtheoreminner{#1}%
  \manualtheoreminner
}{\endmanualtheoreminner}

\newcommand{\peq}{\stackrel{\mathcal{L}}{=}}
\newcommand{\tr}{\text{tr}}

\DeclareMathOperator*{\argmax}{arg\,max}

\title[Uniform VAR Inference]{Uniform Inference For Cointegrated Vector Autoregressive Processes}
\author{Christian Holberg \and Susanne Ditlvesen}
\address{Department of Mathematical Sciences, University of Copenhagen, Universitetsparken 5, 2100
Copenhagen Ø, Denmark}
\email{c.holberg@math.ku.dk}
\begin{document}
\maketitle

%\begin{frontmatter}
\begin{abstract}
Uniformly valid inference for cointegrated vector autoregressive processes has so far proven difficult due to certain discontinuities arising in the asymptotic distribution of the least squares estimator. We extend asymptotic results from the univariate case to multiple dimensions and show how inference can be based on these results. Furthermore, we show that lag augmentation and a recent instrumental variable procedure can also yield uniformly valid tests and confidence regions. We verify the theoretical findings and investigate finite sample properties in simulation experiments for two specific examples.

\end{abstract}

%\end{frontmatter}
%%%%%%%%%%%%%%%%%%%%%%%%%%%%%%%%%%%%%%%%%%%%%%
%% Please use \tableofcontents for articles %%
%% with 50 pages and more                   %%
%%%%%%%%%%%%%%%%%%%%%%%%%%%%%%%%%%%%%%%%%%%%%%
%\tableofcontents

%%%%%%%%%%%%%%%%%%%%%%%%%%%%%%%%%%%%%%%%%%%%%%
%%%% Main text entry area:

\section{Introduction}

%\subsection{Motivation}

Persistence, i.e., long term sensitivity to small shocks, appears to be a commonly occurring characteristic of many stochastic systems encountered in practice. Such processes are often modeled with cointegration where the persistence can be attributed to a number of shared stochastic trends (random walks). Due to their relative simplicity, cointegration models are widely applied. Cointegration in vector-valued autoregressive processes arises when the characteristic polynomial possesses at least one unit root. The asymptotic theory and, hence, inference is heavily reliant upon the fact that a fixed number of roots can be assumed exactly one and the rest stay sufficiently far away from one. This is the case even for simple regression methods such as ordinary least squares. In practice, such assumptions are overly restrictive and more flexible models are often needed for a better description of the empirical data. Slight deviations from the unit root assumption can severely deteriorate the results of the statistical analysis \citep{elliott1998robustness}. Thus, the need arises for inference methods that are uniformly valid over a range of stationary and non-stationary behaviours.

%\subsection{Related work}

So far, methods of inference with proven uniform guarantees only work in the univariate case. Bootstrap inference algorithms are presented in \cite{andrews1993exactly, hansen1999grid} with uniform guarantees given in \cite{mikusheva2007uniform}. Furthermore, \cite{mikusheva2007uniform} provides a uniform asymptotic framework for one-dimensional autoregressive processes with potential unit roots, which served as an inspiration for much of the work in this paper. 

The problem is well-understood in one dimension, while less progress has been made in multiple dimensions. The only methods with uniform guarantees known to the authors employ lag augmentation \citep{dolado1996making, toda1995statistical}. The idea is simple and easy to apply, but lacks efficiency since it essentially involves overfitting the model. Other methods, seeking to avoid this problem, impose restrictive assumptions on the process such that the inference is no longer uniform, but only holds for specific configurations of parameters. The main approach is to assume that the roots are all of a similar proximity to one. In particular, the autoregressive matrix is modeled as a sequence of matrices that approach the identity matrix at some given rate \(k_n\). This is the setup in the instrumental variable methodology (IVX) developed in \cite{kostakis2015robust, magdalinos2020econometric, phillips2009econometric}. This offers some flexibility in terms of how close the roots can be to one. However, the framework does not allow for processes with simultaneously different degrees of persistence.

Generic results on uniform convergence are provided in \cite{andrews2020generic}. However, they hold only insofar as one can establish the right asymptotic distributions under appropriate drifting sequences. While progress in this direction has been made \citep{phillips2013predictive, phillips2015limit}, there has thus far been no general answer to this problem. For example, these papers do not allow the process to have parts that remain stationary for all sample sizes, and only diagonal regression matrices are allowed. Furthermore, the theory in \cite{andrews2020generic} does not provide a way to perform inference once the appropriate limits are established. The method they propose for the one-dimensional autoregressive process is not feasible for multiple dimensions.

A lot of focus has been given to predictive regression problems in which the predictive power of the past of one process on the future of another is assessed. Efficient tests are developed in \cite{campbell2006efficient, jansson2006optimal} and, based on the ideas of uniform inference for univariate autoregressive processes, \cite{phillips2014confidence} leverage these methods for uniformly valid inference. Unfortunately, most of the work only covers the bivariate case in which the regressor is a univariate autoregressive process so that the theory for one dimension directly applies. Another branch of research concerns inference on cointegrating relations robust to deviations from the unit root assumption \citep{duffy2023cointegration, franchi2017improved}, but, again, specific parameterizations of the deviations limit the generality of these results.

Extending the theory in \cite{mikusheva2007uniform} to multiple dimensions runs into several difficulties. Firstly, it is not clear what assumptions to put on the autoregressive matrix to ensure that the asymptotic results hold uniformly while still covering all relevant cases. In one dimension, the autoregressive parameter is a scalar and it is sufficient that it is real, bounded in norm by 1, and bounded away from -1 by some small \(\delta\). We need to extend this idea to matrices. Secondly, while many of the results on the asymptotics of the sample covariance matrices generalize nicely to multiple dimensions, the proofs are more involved. For example, \cite{mikusheva2007uniform} uses Skorohod's embedding, which famously only works in one dimension, to prove that the errors can be assumed to be Gaussian. The third and perhaps most profound difficulty is that the multivariate setting allows for cointegrated systems (or almost cointegrated systems in the case where the roots are only close to unity). This gives rise to certain asymptotic discontinuities. In particular, it necessitates a proper normalization of the sample covariances. These problems extend to the inferential side where, additionally, computational challenges arise. Naively adapting, for example, the grid bootstrap approach of \cite{hansen1999grid} would cause the computational complexity to explode.

%\subsection{Our contributions}

The present work provides methods of inference for multivariate vector autoregressive processes that are proven to be valid uniformly over a set of parameters given in Assumption \ref{ass: U}, including processes that are cointegrated and with roots arbitrarily close to the unit circle. Assumptions \ref{ass: eig} and \ref{ass: jord} are probably slightly stronger than necessary, but they are easy to work with and provide a clear interpretation while still being much more general than anything we have been able to find in the literature. Most of the work dealing with similar problems assume either that \(\Gamma\) is a drifting sequence of diagonal matrices (see, e.g., \cite{phillips2013predictive, phillips2015limit, magdalinos2020econometric}) or impose structural assumptions such that the problem can be reduced to studying only diagonal \(\Gamma\) (see, e.g., \cite{duffy2023cointegration} where they assume that \(\Gamma\) is normal).

The main contributions are the following. First, we extend the asymptotic results of \cite{mikusheva2007uniform} to vector autoregressive processes and where uniformity also holds over a \emph{family} of martingale difference error processes. The main result is Theorem \ref{thm: unif_app}, stating that the asymptotic distributions of the relevant sample covariances can be approximated by stochastic integrals of Ornstein-Uhlenbeck processes (a direct analog to the univariate case). Both the result and the proof are interesting in their own rights. As part of the proof we show that one can approximate the finite sample distribution of crucial statistics by replacing the general error terms with Gaussian errors. We can sample from this approximation at a comparably low computational cost which facilitates inference greatly. 

Second, we provide several ways to construct confidence regions for the autoregressive parameter and show that these are uniformly valid. Confidence regions constructed with IVX and lag augmentation are uniformly valid (but IVX only for the entire autoregressive matrix). 

Third, we show how these confidence regions can be used to answer more general inference questions. The two main applications are confidence intervals for a single coordinate and predictive regression testing with a multivariate regressor. To the best of our knowledge, there have thus far been no attempt in the literature to deal with these applications in a uniform fashion (apart from lag augmentation). We run Monte Carlo experiments to verify the theoretical results and compare the finite sample properties of the different methods. 

Our last contribution is the development of efficient algorithms to solve these inferential tasks. We show how the \emph{Evaluation-Approximation-Maximization} (EAM) algorithm from \cite{kaido2019confidence} circumvents the exploding computational cost inherent in algorithms relying on grid-like methods. Combined with the Gaussian approximation results, this is what makes inference possible in our two main applications.

%\subsection{Organization}

The paper is structured as follows. Section 2 introduces notation and relevant concepts. In particular, it explains the concept of uniform convergence of random variables and presents vector autoregressive processes. Section 3 is devoted to presenting and proving the main asymptotic results. Section 4 deals with inference and shows how the results of Section 3 can be applied to obtain uniformly valid confidence regions. Furthermore, it contains a section on predictive regression, lag augmentation, and IVX. Section 5 contains the results of our Monte Carlo experiments. Finally, Section 6 concludes. The Appendix contains proofs and further technical details on martingale limit results, the Gaussian approximation, the simulation experiments as well as the EAM algorithm, and details on lag augmentation and IVX.

\section{Preliminaries}

This paper is concerned with vector autoregressive processes of order 1 (VAR(1) processes) fulfilling certain assumptions. Specifically, processes that may be integrated of order 1 and cointegrated, that is, processes for which the first difference is stationary and there exists some linear combinations of the coordinate processes that are stationary. 

\emph{Notation}: For a matrix \(A\in\mathbb{C}^{d\times d}\), \(A^T\) denotes its conjugate transpose and its trace is \(\tr(A)\). We write \(\sigma_{max}(A)\) (\(\sigma_{min}(A)\)) for the largest (smallest) singular value of \(A\), and \(\lambda_{max}(A)\) (\(\lambda_{min}(A)\)) for the eigenvalue of \(A\) with the largest (smallest) magnitude. \(||A||\) is the Frobenius norm and \(||A||_2\) is the spectral norm, i.e., \(||A||=\sqrt{\tr(A^T A)}\) and \(||A||_2 = \sqrt{\sigma_{max}(A)}\). For vectors, \(||\cdot||\) is the usual Euclidean norm. Define \(S_{d}\) to be the set of \(d\times d\) positive semidefinite matrices. We employ the usual big-\(O\) and little-\(o\) notation and use \(o_p\) to denote convergence in probability and \(O_p\) to denote boundedness in probability.

\subsection{Uniform convergence of random variables} \label{sec: prelim}
The definitions of uniform convergence in probability and in distribution are essentially the same as in \cite{kasy2019uniformity, lundborg2021conditional}. Assume some background probability space, \((\Omega, \mathcal{F}, \mathbb{P})\), on which all future random variables are defined. For two random vectors, \(X\) and \(Y\), taking values in \((\mathbb{C}^d, \mathcal{B}(\mathbb{C}^d))\), we denote by \(P_{X}\) and \(P_{Y}\) the law of \(X\) and \(Y\) and write \(X \peq Y\) if they are equal in law. Let \(BL_1\) be the space of functions \(f:\mathbb{C}^d\rightarrow [-1, 1]\) that are Lipschitz continuous with constant at most 1. Let \(\mathcal{P}(\mathbb{C}^d, \mathcal{B}(\mathbb{C}^d))\) be the set of probability measures on \((\mathbb{C}^d, \mathcal{B}(\mathbb{C}^d))\). The bounded Lipschitz metric on \(\mathcal{P}(\mathbb{C}^d, \mathcal{B}(\mathbb{C}^d))\) is given by
\[
d_{BL}(\mu, \nu):=\sup_{f\in BL_1}\left|\int_{\mathbb{C}^d} f d\mu - \int_{\mathbb{C}^d} f d\nu\right|, \quad \mu, \nu\in \mathcal{P}(\mathbb{C}^d, \mathcal{B}(\mathbb{C}^d)).
\]
We use the shorthand \(d_{BL}(X, Y) = d_{BL}(P_X, P_Y)\) to denote the bounded Lipschitz metric between the laws of two random variables, \(X\) and \(Y\). It is well known that \(d_{BL}\) metrizes weak convergence which motivates the following definition of uniform convergence.

\begin{definition}[Uniform convergence] \label{def: unif_conv}
    Let \((X_{n,\theta})_{n\in\mathbb{N}, \theta\in\Theta}\) and \((Y_{n,\theta})_{n\in\mathbb{N}, \theta\in\Theta}\) be two sequences of families of random \(d\)-dimensional vectors defined on \((\Omega, \mathcal{F}, \mathbb{P})\) and indexed by some set \(\Theta\) (of possibly infinite dimension).
    \begin{enumerate}
        \item We say that \emph{\(X_{n, \theta}\) converges uniformly to \(Y_{n,\theta}\) over \(\Theta\) in distribution} (or, for short, \(X_{n,\theta}\rightarrow_w Y_{n, \theta}\) uniformly over \(\Theta\)) if 
        \[
        \lim_{n\rightarrow\infty}\sup_{\theta\in\Theta}d_{BL}\left(X_{n, \theta}, Y_{n, \theta}\right) = 0.
        \]
        \item We say that \emph{\(X_{n, \theta}\) converges uniformly to \(Y_{n,\theta}\) over \(\Theta\) in probability} (or, for short, \(X_{n,\theta}\rightarrow_p Y_{n, \theta}\) uniformly over \(\Theta\)) if, for every \(\epsilon > 0\),
        \[
        \lim_{n\rightarrow\infty}\sup_{\theta\in\Theta}\mathbb{P}\left(||X_{n, \theta} - Y_{n, \theta}|| > \epsilon\right) = 0.
        \]
    \end{enumerate}
\end{definition}

Uniform convergence could also be stated as convergence along all sub-sequences \(\theta_n\subset\Theta\) (see Definition 2 and Lemma 1 in \cite{kasy2019uniformity}). Additionally, we allow the limiting distribution to be a sequence, since the results below are stated in terms of an approximating sequence of random variables. We obtain the conventional notion by letting \(Y_{n, \theta} = Y_\theta\).

\subsection{Model}

Consider some \(\Theta \subset \mathbb{R}^{d\times d}\times S_d\times\mathbb{R}_+\). For any \(\theta\in\Theta\) there exist \(\Gamma_\theta, \Sigma_\theta\), and \(c_\theta\) such that \(\theta = (\Gamma_\theta, \Sigma_\theta, c_\theta)\). Let \(N_\theta\in\{1,...,d\}\) denote the number of distinct eigenvalues of \(\Gamma_\theta\) and \(\lambda_\theta\in\mathbb{C}^{N_\theta}\) the corresponding vector of ordered eigenvalues, that is, \(|\lambda_{\theta, 1}|\ge |\lambda_{\theta, 2}| \ge ... \ge |\lambda_{\theta, N_\theta}|\) with multiplicities \(m_{\theta, 1},..., m_{\theta, N}\in\{1,..., d\}\). Where this does not cause confusion, we omit the subscript \(\theta\).  Let \((X_{t, \theta})_{t\in\mathbb{N}, \theta\in\Theta}\) and \((\epsilon_{t, \theta})_{t\in\mathbb{N}, \theta\in\Theta}\) be two families of \(\mathbb{R}^d\)-valued stochastic processes and \((\mathcal{F}_{t,\theta})_{t\in\mathbb{N}}\) the filtration generated by \((\epsilon_{t, \theta})_{t\in\mathbb{N}}\).

\begin{manualtheorem}{M}\label{ass: M}
\(X_{t, \theta}\) and \(\epsilon_{t, \theta}\) satisfy the following:

\begin{enumerate}[label={M.\arabic*.}, ref={M.\arabic*}, align=left]
\item \label{ass: mart}%
    \(\epsilon_{t, \theta}\) is a stationary martingale difference sequence wrt. \(\mathcal{F}_{t, \theta}\), that is,
    \[
    \sup_{t\in\mathbb{N}, \theta\in\Theta} \mathbb{E}||\epsilon_{t, \theta}|| < \infty
    \]
    and \(\mathbb{E}(\epsilon_{t, \theta}|\mathcal{F}_{t-1, \theta})=\mathbb{E}\epsilon_{0, \theta}=0\) for all \(t\ge 1, \theta\in\Theta\).
\item \label{ass: cov}%
    For all \(\theta\in\Theta\), the conditional covariance matrix of \(\epsilon_{t, \theta}\) exists and is given by \(\mathbb{E}(\epsilon_{t, \theta}\epsilon_{t, \theta}^T|\mathcal{F}_{t-1, \theta})=\mathbb{E}\epsilon_{0, \theta}\epsilon_{0, \theta}^T=\Sigma_\theta\) a.s. for all \(t\ge 1\).
\item \label{ass: mom}%
    There exists some small \(\delta > 0\) such that \(\mathbb{E}||\epsilon_{t, \theta}||^{2+\delta} \le c_\theta\) a.s. for all \(t\in\mathbb{N}, \theta\in{\Theta}\).
\item \label{ass: var}%
    \(X_{t, \theta}\) is a VAR(1) process, that is, for all \(\theta\in\Theta\),
    \[
    X_{t, \theta} = \Gamma_\theta X_{t-1, \theta} + \epsilon_{t, \theta}
    \]
    for \(t\ge 1\) and \(X_{0, \theta}=0\).
\end{enumerate}
\end{manualtheorem}

These assumptions ensure that \(X_{t, \theta}\) is a VAR(1) process started at \(0\) with model parameters given by the index \(\theta\). Assumptions \ref{ass: mart} and \ref{ass: var} imply that \(X_{t, \theta}\) is adapted to \(\mathcal{F}_{t, \theta}\).

For future reference let us define the following set of \(d\times d\) matrices. For a given \(\delta > 0\), let \(\mathcal{J}_d(\delta)\subset\mathbb{C}^{d\times d}\) be the set of upper triangular matrices such that every \(J\in\mathcal{J}_d(\delta)\) can be decomposed as \(J=D+N\) with \(D\) diagonal such that \(|D_{11}|\ge |D_{22}|\ge\dots\ge |D_{dd}|\) and \(N\) equal to 0 everywhere except on the super-diagonal where it satisfies \(N_{i, i+1}\in\{0, 1\}\) if \(|D_{ii}|\le \delta\) and 0 otherwise for \(i=1,\dots, d-1\). In other words, every \(J\in\mathcal{J}_d(\delta)\) can be written as a block diagonal matrix where the upper left block is diagonal and contains all eigenvalues greater than \(\delta\) while the lower right block can have ones on the super-diagonal and has eigenvalues less than \(\delta\). We call the matrices in \(\mathcal{J}_d(\delta)\) \emph{Jordan-like}.

\begin{remark} \label{rem: jord_norm_form}
For any \(\theta\in\Theta\), there exist matrices \(F_\theta\in\mathbb{C}^{d\times d}\) and \(J_\theta\) such that \(J_\theta\) is a Jordan matrix and \(\Gamma_\theta = F_\theta J_\theta F_\theta^{-1}\). Up to reordering of the eigenvalues, the matrix \(J_\theta\) is unique and satisfies \(J_\theta\in \mathcal{J}_d(|\lambda_{\theta, 1}|)\). It is called the \emph{Jordan canonical form} of \(\Gamma_\theta\).
\end{remark}

\section{Asymptotic Properties} \label{sec: asym}

The key building blocks for inference are the two covariance matrices
\begin{equation}
\label{eq: cov}
S_{XX} = \frac{1}{n}\sum_{t=1}^n X_{t-1, \theta}X_{t-1, \theta}^T, \quad S_{X\epsilon} = \frac{1}{n}\sum_{t=1}^n X_{t-1, \theta}\epsilon_{t, \theta}^T.
\end{equation}
Obviously, \(S_{XX}\) and \(S_{X\epsilon}\) are families of stochastic processes depending on \(n\) and \(\theta\), which we suppress to avoid cluttering up the notation, but the dependence should be kept in mind. We first need to determine what happens to \(S_{XX}\) and \(S_{X\epsilon}\) when \(n\) goes to infinity and for varying \(\theta\). We cannot hope to say anything uniformly without further assumptions on \(\Theta\). The following assumptions are sufficiently general to cover a wide range of behaviours while still allowing for uniform asymptotic results.

\begin{manualtheorem}{U}\label{ass: U}
\(\Theta\) satisfies the following:
\begin{enumerate}[label={U.\arabic*.}, ref={U.\arabic*}, align=left]
\item \label{ass: c} %
    \(\sup_{\theta\in\Theta} c_\theta < \infty\).
\item \label{ass: sig}%
    \(\sup_{\theta\in\Theta}\{\sigma_{max}(\Sigma_\theta) + \sigma_{min}(\Sigma_\theta)^{-1}\} < \infty\).
\item \label{ass: eig}%
    There exists \(\alpha\in(0, 1)\) small so that with \(r_\alpha=(1-\alpha)(2-\alpha)/\alpha\) (see Fig. \ref{fig: eig})
    \[
    \sup_{\theta\in\Theta} \left\{\max_{1\le i\le N_{\theta}}\left[\max\left( \frac{||\lambda_{\theta, i}|(1-\lambda_{\theta, i})|}{r_\alpha(1-|\lambda_{\theta, i}|)}, \;|\lambda_{\theta, i}|\right)\right]\right\}\le 1.
    \]
\item \label{ass: jord}%
    There exists \(F_\theta\in\mathbb{C}^{d\times d}\) and \(J_{\theta}\in \mathcal{J}_d(1-\alpha)\) such that \(F_\theta^{-1}\Gamma_\theta F_\theta = J_\theta\) and 
    \[
    \sup_{\theta\in\Theta}\left\{\sigma_{max}\left(F_\theta\right) + \sigma_{min}\left(F_\theta\right)^{-1} + \sigma_{max}\left(J_\theta\right)\right\}<\infty.
    \]
\end{enumerate}

\end{manualtheorem}

Assumption \ref{ass: c} is a moment condition on the error process which is needed for some of the triangular array martingale difference limit results. It is implied by the other conditions if the errors are i.i.d. Gaussian. Assumption \ref{ass: sig} states that \(||\Sigma_\theta||\) and \(||\Sigma_\theta^{-1}||\) are uniformly bounded for any matrix norm. In particular, \(\Sigma_\theta\) is of full rank. This is a natural condition when considering uniform convergence. The important assumptions are \ref{ass: eig} and \ref{ass: jord}. Both assumptions have clear interpretations and are sufficient if we want to limit our attention to processes that are at most integrated of order 1 and without seasonal cointegration. In particular, to avoid higher orders of integration, we must restrict all eigenvalues to have magnitude less than 1 (ensured by Assumption \ref{ass: eig}, see Fig. \ref{fig: eig}). For eigenvalues with magnitude 1, the corresponding Jordan block must be scalar \citep{archontakis1998alternative} (ensured by Assumption \ref{ass: jord}). Note that the matrices \(J_\theta\) in Assumption \ref{ass: jord} are not required to be Jordan matrices so that the assumption allows, for example, for matrices of the form 
\[
\Gamma = \begin{pmatrix}
    \lambda & 1 \\
    0 & \lambda'
\end{pmatrix}
\]
for \(\lambda, \lambda' \in \mathbb{R}\) arbitrarily close together as long as \(|\lambda|, |\lambda'|\le 1 - \alpha\).

To avoid seasonal cointegration, eigenvalues with magnitude 1 are restricted to be exactly equal to 1. Specifically, as the magnitude of an eigenvalue approaches one, the eigenvalue itself approaches 1. See Fig. \ref{fig: eig} for the region of the complex plane satisfying Assumption \ref{ass: eig} for some small \(\alpha \in (0, 1)\). Since $\alpha$ can be chosen freely, the parameter space can be made arbitrarily close to the unit circle.

\begin{figure}
    \centering
    \includegraphics[width=200pt]{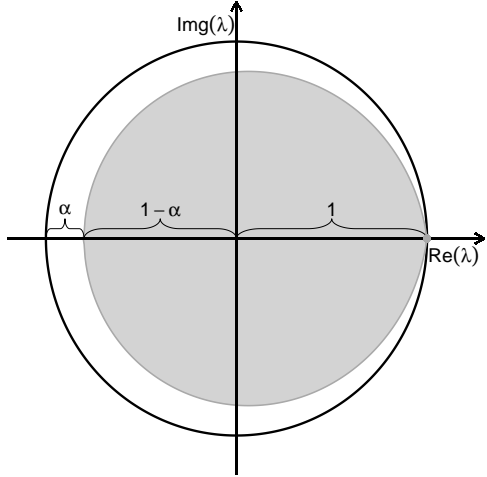}
    \caption{\small The region of the complex plane of allowed eigenvalues given by Assumption \ref{ass: eig}. The gray area includes all the eigenvalues allowed in the current setting. The black circle enclosing the grey area is the unit circle. The smallest allowed real eigenvalue is \(\alpha-1\) and the only eigenvalue with magnitude 1 is real and equal to 1.}
    \label{fig: eig}
\end{figure}

To state our main result, define \(M_i = \sum_{j\le i} m_j\) and write \(i_k = \min\{i \ge 1 | M_i \ge k\}\). The limiting behaviour of the \(k\)'th coordinate of \(X_{t, \theta}\) depends on how close the eigenvalue, \(\lambda_{i_k}\), is to 1. As in the univariate case \citep{mikusheva2007uniform}, we unify the range of asymptotics with an Ornstein-Uhlenbeck process. For any \(\theta\), we let \(C_n(\theta)\) be the \(d\times d\) diagonal matrix whose \(i\)'th diagonal block is \(n\log(|\lambda_i|)I_{m_i}\) with the convention that \(\log(0)=-\infty\). For convenience, we sometimes suppress the dependence on \(\theta\) and \(n\) and just write \(C\). In what follows, uniform convergence of random matrices means uniform convergence of the vectorization of these matrices so that Definition \ref{def: unif_conv} applies directly.

\begin{theorem}[Uniform convergence of covariance matrices] \label{thm: unif_app}
    Under Assumptions \ref{ass: M} and \ref{ass: U} and after possibly enlarging \((\Omega, \mathcal{F}, \mathbb{P})\), there exists a standard \(d\)-dimensional Brownian motion, \((W_t)_{t\in[0,1]}\), and a family of processes, \((J_{t, C})_{t\in[0,1], n\in\mathbb{N}, \theta\in\Theta}\), with %\(C=C(n, \theta)\),
    \begin{equation} \label{eq: JtC}
    J_{t, C} = \int_0^t e^{(t-s)C}F_{\theta}\Sigma^{\frac{1}{2}}dW_s, \quad J_{0, C} = 0,
    \end{equation}
    such that the following approximations hold for \(n\rightarrow\infty\)
    \begin{equation} \label{eq: unif_app_xx}
        H^{-\frac{1}{2}}F_{\theta}S_{XX}F_{\theta}^T H^{-\frac{1}{2}} \rightarrow_w G^{-\frac{1}{2}}\int_0^1 J_{t, C}J_{t, C}^T dt G^{-\frac{1}{2}},
    \end{equation}
    \begin{equation} \label{eq: unif_app_xe}
        \sqrt{n}H^{-\frac{1}{2}}F_{\theta} S_{X\epsilon} \rightarrow_w G^{-\frac{1}{2}}\int_0^1 J_{t, C} dW_t^T \Sigma^{\frac{1}{2}},
    \end{equation}
    uniformly over \(\Theta\) where the covariance matrices are defined in \eqref{eq: cov} and the normalizing matrices are given by
    \begin{equation} \label{eq: HandG}
    H = F_{\theta}\mathbb{E}\left(\frac{1}{n}\sum_{t=1}^n X_{t-1, \theta}X_{t-1, \theta}^T\right)F_{\theta}^T, \quad G = F_{\theta}\mathbb{E}\left(\int_0^1 J_{t, C} J_{t, C}^T dt\right)F_{\theta}^T.
    \end{equation}
\end{theorem}

We prove the uniform results for the special case of $F_{\theta}$ being the identity, i.e., under the assumption that \(\Gamma\in \mathcal{J}_d(1 - \delta)\). Technically, this allows for complex-valued \(\Gamma\), thus, the proofs are more general than the real-valued case. Also, it is not hard to generalize to any $F_{\theta}$ fulfilling Assumption \ref{ass: jord}. Indeed, for \(X_{t, \theta}\) generated by \(\Gamma\in \mathbb{R}^{d\times d}\), there exist \(F\in \mathbb{C}^{d\times d}\) and \(J\in J_d(1 - \delta)\) such that \(F\) is invertible with \(F^{-1}JF = \Gamma\). The transformed process \(\tilde{X}_{t, \theta} = FX_{t, \theta}\) is then of the required form with parameters \(\tilde{\theta}=(J, F\Sigma F^T, ||F||^{2+\delta}c)\). Assuming that \(F\) is uniformly invertible and bounded in norm then ensures that \(\tilde{\theta}\) satisfies Assumption \ref{ass: U}.

We prove Theorem \ref{thm: unif_app} in several steps. The main idea is to split \(\Theta\) into overlapping regions and prove that Theorem \ref{thm: unif_app} holds in each region. Consider 
\[
R_{n, 0} := \left\{\theta\in\Theta : |\lambda_1| \le 1 - \frac{\log n}{n}\right\}, \quad R_{n, d}:= \left\{\theta\in\Theta : |\lambda_N| \ge 1 - n^{-\eta}\right\},
\]
where \(\eta\in(0, 1)\) is to be specified later. The two regions correspond to the stationary and local-to-unity (non-stationary) regimes, respectively, in the univariate case. Different asymptotics arise depending on the region and, in particular, on how fast the eigenvalues converge to unity. Throughout the rest of this section we assume that Assumptions \ref{ass: M} and \ref{ass: U} hold with \(F_\theta = I\).

\subsection{Non-stationary asymptotics} \label{sec: ltu}

In this section we consider sequences of parameters in the non-stationary region \(R_{n, d}\). For simplicity we assume throughout this subsection that \(\epsilon_{t, \theta}\) is i.i.d. Gaussian with mean 0 and covariance \(\Sigma_\theta\). In Appendix \ref{app: approx} it is argued why this is not a restriction. Indeed, all the relevant sample moments can be approximated by Gaussian counterparts.

Let \((W_{t})_{t\in[0,1]}\) be a standard \(d\)-dimensional Brownian motion. Since \(\Sigma^{\frac{1}{2}}(W_{t/n} - W_{(t-1)/n})\peq \epsilon_{t, \theta}/\sqrt{n}\) for all \(n\in\mathbb{N}\), \(0\le t \le n\) and \(\theta\in\Theta\), we get
\begin{equation} \label{eq: yr_int}
    H^{-\frac{1}{2}}S_{X\epsilon} \peq \int_0^{1}\int_0^t f(t, s, n, \theta) dW_s dW_t^T \Sigma^{\frac{1}{2}},
\end{equation}
\begin{equation} \label{eq: yy_int}
H^{-\frac{1}{2}}S_{XX} H^{-\frac{1}{2}} \peq  \int_0^1\left(\int_0^t f(t, s, n, \theta) dW_s\right)\left(\int_0^t f(t, s, n, \theta) dW_s\right)^T dt,
\end{equation}
where $H$ is defined in \eqref{eq: HandG} and \(f(t, s, n, \theta) =  \sqrt{n}H^{-\frac{1}{2}}\Gamma^{\lfloor nt \rfloor - \lfloor ns \rfloor - 1} \Sigma^{\frac{1}{2}}\mathbf{1}\{s \le \lfloor nt\rfloor / n\}\). We then see that the following Lemma is a direct consequence of Lemma \ref{lem: l2_app} in Appendix \ref{app: proofs}.

\begin{lemma} \label{lem: non_stat_app}
    For the covariance matrices \eqref{eq: cov} and $H, G$ and $J_{C,t}$ given in \eqref{eq: HandG} and \eqref{eq: JtC}, the following hold
    \begin{equation} \label{eq: non_stat_app_xx}
        \lim_{n\rightarrow\infty}\sup_{\theta\in R_{n, d}}d_{BL}\left(H^{-\frac{1}{2}}S_{XX}H^{-\frac{1}{2}}, G^{-\frac{1}{2}}\int_0^1 J_{C, t}J_{C, t}^T dt G^{-\frac{1}{2}}\right) = 0,
    \end{equation}
    \begin{equation} \label{eq: non_stat_app_xe}
        \lim_{n\rightarrow\infty}\sup_{\theta\in R_{n, d}}d_{BL}\left(\sqrt{n}H^{-\frac{1}{2}}S_{X\epsilon}, G^{-\frac{1}{2}}\int_0^1 J_{C, t} dW_t^T \Sigma^{\frac{1}{2}}\right) = 0.
    \end{equation}
\end{lemma}

\subsection{Stationary asymptotics} \label{sec: stat}

In this section we consider sequences of parameters in the stationary region \(R_{n, 0}\). We first show that the classical asymptotic theory for stationary VAR(1) processes applies. Since \(R_{n, 0}\) and \(R_{n, d}\) overlap, we then prove that the right hand sides of \eqref{eq: unif_app_xx} and \eqref{eq: unif_app_xe} converge to the standard stationary limiting distributions for the diagonal entries \(C_{ii}\) going to \(-\infty\).

The standard stationary theory in multiple dimensions mimics the univariate case. We follow the same strategy as in \cite{phillips2007limit}, but allowing for multiple dimensions and a family of error processes \(\epsilon_{t, \theta}\). In this regime, we find that, when properly normalized, \(S_{XX}\) converges in probability to the identity matrix and \(\text{vec}(S_{Xe})\) converges in distribution to a \(d^2\)-dimensional standard Gaussian. 

\begin{theorem} \label{thm: stat_asym}
    Let \(V\sim \mathcal{N}(0, I_{d^2})\). For all \(\epsilon>0\) and \(s\in [0, 1]\),
    \begin{equation} \label{eq: stat_asym_xx}
        \lim_{n\rightarrow\infty}\sup_{\theta\in R_{n, 0}} \mathbb{P}\left(\left|\left|\frac{1}{n}H^{-\frac{1}{2}}\left(\sum_{t=1}^{\lfloor ns\rfloor}X_{t-1, \theta}X_{t-1, \theta}^T\right)H^{-\frac{1}{2}} - sI\right|\right| > \epsilon\right) = 0
    \end{equation}
    and
    \begin{equation} \label{eq: stat_asym_xe}
        \lim_{n\rightarrow\infty}\sup_{\theta\in R_{n, 0}} d_{BL}\left(\textnormal{vec}\left(\sqrt{n}H^{-\frac{1}{2}}S_{X\epsilon}\Sigma^{-\frac{1}{2}}\right), V\right) = 0.
    \end{equation}
\end{theorem}

For the special case \(s=1\), equation \eqref{eq: stat_asym_xx} shows that \(S_{XX}\) converges in probability to the identity matrix. By Theorem \ref{thm: stat_asym} and Proposition 8 in the supplementary material for \cite{lundborg2021conditional}, the proof of Theorem \ref{thm: unif_app} in the stationary regime is complete if we can show that, for any \((\theta_n\in R_{n, 0})_{n\in\mathbb{N}}\),
\begin{equation} \label{eq: jc_stat_xx}
    G^{-\frac{1}{2}}\int_0^1 J_{t, C}J_{t, C}^T dt G^{-\frac{1}{2}} \rightarrow_w I,
\end{equation}
\begin{equation} \label{eq: jc_stat_xe}
    G^{-\frac{1}{2}}\int_0^1 J_{t, C} dW_t^T \rightarrow_w N.
\end{equation}
We emphasize that \(G\) and \(C\) in eqs. \eqref{eq: jc_stat_xx}-\eqref{eq: jc_stat_xe} are functions of \(\theta_n\) and therefore they are sequences of matrices. In particular, \(C_{ii}\le n\log(1-\log(n)/n)\rightarrow -\infty\) for \(n\rightarrow \infty\) and \(1\le i \le d\). Eqs. \eqref{eq: jc_stat_xx}-\eqref{eq: jc_stat_xe} are therefore a consequence of the following.

\begin{lemma} \label{lem: jc_stat}
    Let \((C_n)_{n\in\mathbb{N}}, (\Omega_n)_{n\in\mathbb{N}}\subset \mathbb{R}^{d\times d}\) be sequences of matrices such that \(C_n\) is diagonal, \((C_n)_{ii}\rightarrow - \infty\) for \(n\rightarrow \infty\) and \(1\le i \le d\), and \(\Omega_n\) is positive definite with singular values bounded from below and above uniformly over \(n\). Let \((W_t)_{t\in [0, 1]}\) be a standard \(d\)-dimensional Brownian motion and define the family of \(d\)-dimensional Ornstein-Uhlenbeck processes, \((J_{t, n})_{t\in[0,1], n\in \mathbb{N}}\), given by
    \[
    J_{t, n} = \int_0^t e^{(t-s)C_n}\Omega_n^{\frac{1}{2}}dW_s, \quad J_{0, n} = 0,
    \]
    along with the normalizing matrices \(G_n = \mathbb{E}\left(\int_0^1 J_{t, n} J_{t, n}^T dt\right)\). Then, for \(n\rightarrow \infty\),
    \[
    G_n^{-\frac{1}{2}}\int_0^1 J_{t, n}J_{t, n}^T dt G_n^{-\frac{1}{2}} \rightarrow_p I,
    \]
    \[
    \text{vec}\left(G_n^{-\frac{1}{2}}\int_0^1 J_{t, n} dW_t^T\right) \rightarrow_w V,
    \]
    where \(V\sim \mathcal{N}(0, I_{d^2})\).
\end{lemma}

\subsection{Mixed asymptotics} \label{sec: mixed}

So far, all the eigenvalues were in the same regime. We have yet to explore what happens when there are eigenvalues in both regimes. We call this case the mixed regime. Define for \(1\le k \le d-1\) and some fixed \(\gamma\in (0, 1-\eta)\) the sets
\[
R_{n, k} = \left\{\theta\in\Theta : M_{i_k} = k, |\lambda_{i_k}| \ge 1 - n^{-\eta - \gamma}, |\lambda_{i_{k+1}}|\le 1 - n^{-\eta - \gamma}\right\}.
\]
Since \(1-n^{-\eta} \le 1 - n^{-\eta - \gamma} \le 1 - \log(n)/n\), then for \(\theta\in R_{n, k}\) there are at least \(k\) coordinates in the non-stationary regime and \(d-k\) coordinates in the stationary regime (and some might be in both). Furthermore, for any \(n\in \mathbb{N}\), \(\Theta = \bigcup_{0\le k\le d} R_{n, k}.\) Thus, showing that \eqref{eq: unif_app_xx} and \eqref{eq: unif_app_xe} hold uniformly over \(R_{n, k}\) for any fixed \(1\le k\le d-1\) completes the proof of Theorem \ref{thm: unif_app}. This is the content of the following lemma proved in Appendix \ref{app: mix}.

\begin{lemma} \label{lem: mixed_app}
    Let \(1\le k\le d-1\). We have
    \begin{equation} \label{eq: mixed_app_xx}
        \lim_{n\rightarrow\infty}\sup_{\theta\in R_{n, k}}d_{BL}\left(H^{-\frac{1}{2}}S_{XX}H^{-\frac{1}{2}}, G^{-\frac{1}{2}}\int_0^1 J_{C, t}J_{C, t}^T dt G^{-\frac{1}{2}}\right) = 0,
    \end{equation}
    \begin{equation} \label{eq: mixed_app_xe}
        \lim_{n\rightarrow\infty}\sup_{\theta\in R_{n, k}}d_{BL}\left(\sqrt{n}H^{-\frac{1}{2}}S_{X\epsilon}, G^{-\frac{1}{2}}\int_0^1 J_{C, t} dW_t^T \Sigma^{\frac{1}{2}}\right) = 0.
    \end{equation}
\end{lemma}

\section{Uniform Inference} \label{sec: inf}

Having established the asymptotic properties of \(S_{XX}\) and \(S_{X\epsilon}\), we now seek to develop uniformly valid methods of inference. We focus on two important cases in which uniformly valid inference has so far proven challenging: predictive regression testing and coordinate confidence intervals. Inference in these settings can be hard even from a point-wise perspective since the presence of exact unit roots makes it problematic to construct test statistics with standard asymptotic distributions.

It is not trivial to conduct inference on \(\Gamma\) even in lieu of Theorem \ref{thm: unif_app}. The main problem is the presence of the nuisance parameter \(C_n(\theta)\) in \eqref{eq: unif_app_xx}-\eqref{eq: unif_app_xe}, which cannot be uniformly consistently estimated. Indeed, for a sequence \(\Gamma_n=I - C/n\) where the real part of the eigenvalues of \(C\in \mathbb{R}^{d\times d}\) are all strictly negative, the problem is essentially equivalent to estimating \(C = n(I - \Gamma)\). But it is well known that, in this setting, \(\Gamma\) can only be estimated at rate \(O(n^{-1})\). One way to solve this is by the use of test inversion or so-called grid bootstrap methods, which have been widely applied in the unitary case (see \cite{hansen1999grid, mikusheva2007uniform} for grid bootstrap and \cite{campbell2006efficient, phillips2014confidence} for an application to predictive regression). While this is fairly easy in one dimension, adapting these methods to vector autoregressive processes is prohibitive since the computational complexity quickly explodes. We now present an approach, which keeps the computational burden to a minimum. We omit the dependence on \(\theta\) in the subscript of all random variables.

Consider the least squares estimator, \(\hat{\Gamma}\), given by
\[
\hat{\Gamma} = \frac{1}{n}\sum_{t=1}^n X_{t}X_{t-1}^T\left(\frac{1}{n}\sum_{t=1}^n X_{t-1}X_{t-1}^T\right)^{-1} = \Gamma + S_{X \epsilon}^T S_{XX}^{-1}.
\]
It follows from Theorem \ref{thm: unif_app} that \(\hat{\Gamma}\) is a uniformly consistent estimator of \(\Gamma\) with a rate of convergence depending on the proximity of the eigenvalues of \(\Gamma\) to one. Indeed, since \(\sqrt{n}(\hat{\Gamma} - \Gamma)=(\sqrt{n} S_{X\epsilon})^T S_{XX}^{-1}\), we find that \(n(\hat{\Gamma} - \Gamma)S_{XX}(\hat{\Gamma}-\Gamma)^T = O_p(1)\) which implies that \(\sqrt{n}(\hat{\Gamma} - \Gamma) = O_p(1)\). We define a uniformly consistent estimator of \(\Sigma\) by averaging the squared residuals, i.e., with \(\hat{\epsilon}_t = X_t - \hat{\Gamma}X_{t-1}\) we define \(\hat{\Sigma} = S_{\hat{\epsilon}\hat{\epsilon}}\) analogously to \(S_{XX}\) with \(X_{t-1}\) replaced by \(\hat{\epsilon}_t\). Another consequence of Theorem \ref{thm: unif_app} is a uniform approximation of the \(t^2\)-statistic. In particular, 
\begin{multline} \label{eq: t_gam}
\hat{t}_\Gamma^2 = \tr\left(n\hat{\Sigma}^{-\frac{1}{2}}\left(\hat{\Gamma} - \Gamma\right)S_{XX}\left(\hat{\Gamma} - \Gamma\right)^T\hat{\Sigma}^{-\frac{1}{2}}\right) \\
\rightarrow_w \tr\left(\left(\int_0^1 \hat{J}_{t, C} dW_t^T\right)^T\left(\int_0^1 \hat{J}_{t, C}\hat{J}_{t, C}^T dt\right)^{-1}\int_0^1 \hat{J}_{t, C} dW_t^T \right) := t^2_\Gamma
\end{multline}
uniformly over \(\Theta\) where \(C = C_n(\theta)\) is given in Section \ref{sec: asym} and \(\hat{J}_{t, C}\) is defined analogously to \(J_{t, C}\) but with \(\Sigma\) replaced by the consistent estimator \(\hat{\Sigma}\). Thus, for a fixed significance level, \(\alpha\in (0, 1)\), a uniformly valid \(100(1-\alpha)\%\) confidence region for \(\Gamma\) can be constructed by test-inversion. Let \(q_{n, \Gamma}(1-\alpha)\) denote the \(100(1-\alpha)\%\) quantile\footnote{Note that the quantile depends on \(\Gamma\) (which is also made explicit in the notation) since the asymptotic distribution depends on \(C\) which is a function of \(\Gamma\) and \(n\).} of \(t^2_\Gamma\) and define the confidence region
\begin{equation} \label{eq: cr_a}
    CR_a(\alpha) = \left\{\Gamma : \hat{t}^2_\Gamma \le q_{n, \Gamma}(1-\alpha)\right\}.
\end{equation}
However, the distribution of \(t^2_\Gamma\) is non-standard and, therefore, computing the quantiles \(q_{n, \Gamma}\) requires extensive simulations and can be quite expensive. 

Another approach relies on the Gaussian approximations (Appendix \ref{app: approx}). It is similar to \emph{Andrew's Method} (see \cite{mikusheva2007uniform}), and similar in spirit to grid bootstrap. It was originally suggested by \cite{andrews1993exactly} but has so far only been applied in the univariate case. For a given \(\Gamma\), define the \(VAR(1)\) process, \((Y_t)_{t\in \mathbb{N}}\), by
\[
Y_t = \Gamma Y_{t-1} + e_t, \quad Y_0 = 0
\]
where \(e_t\sim \mathcal{N}(0, \hat{\Sigma})\) i.i.d. Let
\[
\tilde{t}^2_\Gamma = \tr\left(n\hat{\Sigma}^{-\frac{1}{2}}S_{eY} S_{YY}^{-1} S_{eY}^T\hat{\Sigma}^{-\frac{1}{2}}\right)
\]
and denote by \(\tilde{q}_{n, \Gamma}(1-\alpha)\) the \(100(1-\alpha)\%\) quantile of \(\tilde{t}^2_\Gamma\). A confidence region for \(\Gamma\) is then obtained by
\begin{equation} \label{eq: cr_b}
CR_b(\alpha) = \left\{\Gamma : \hat{t}^2_\Gamma \le \tilde{q}_{n, \Gamma}(1-\alpha)\right\}. 
\end{equation}
The distribution of \(\tilde{t}^2_\Gamma\) is still non-standard, but the quantiles \(\tilde{q}_{n, \Gamma}\) are much easier to compute by simulation. The following Theorem states that both confidence regions are uniformly asymptotically valid over the parameter space \(\Theta\). A proof can be found in Appendix \ref{app: cr}.

\begin{theorem} \label{thm: valid_cr}
    Fix \(\alpha\in (0, 1)\) and let \(CR_a(\alpha)\) and \(CR_b(\alpha)\) be as given in \eqref{eq: cr_a} and \eqref{eq: cr_b}. Under Assumptions \ref{ass: M} and \ref{ass: U}, both are asymptotically uniformly valid over \(\Theta\) in the sense that
    %\[
    %\liminf_{n\rightarrow\infty} \inf_{\theta\in \theta} %\mathbb{P}\left(\Gamma\in CR_a(\alpha)\right) \ge 1 - \alpha
    %\]
    %and
    %\[
    %\liminf_{n\rightarrow\infty} \inf_{\theta\in \theta} %\mathbb{P}\left(\Gamma\in CR_b(\alpha)\right) \ge 1 - \alpha.
    %\]
    \[
    \liminf_{n\rightarrow\infty} \inf_{\theta\in \theta} \mathbb{P}\left(\Gamma\in CR_i(\alpha)\right) \ge 1 - \alpha
    \]
    for $i = a,b$.
\end{theorem}

\subsection{Predictive Regression} \label{sec: pr}

One application is robust inference in the predictive regression model. For a deeper discussion of why uniformly valid inference methods are important in this setting see \cite{campbell2006efficient, elliott1994inference}, which cover the case of a univariate regressor, but the same considerations hold more generally. To fix ideas, consider 
\[
\Theta_P = \left\{\theta \in \Theta : \Gamma_{j1}= 0 \; \forall j=1,..., d\right\}.
\]
If \(X_t\) is a VAR(1)-process satisfying Assumption \ref{ass: M} parameterized by \(\theta\in \Theta_P\), then we can split \(X_t=(Y_t, \tilde{X}_t^T)^T\) and \(\epsilon_t = (\rho_t, \tilde{\epsilon}_t^T)^T\) into their first coordinate and their last \(d-1\) coordinates such that 
\begin{align*}
Y_t = \gamma^T &\tilde{X}_{t-1} + \rho_t, \\
\tilde{X}_t = \tilde{\Gamma}&\tilde{X}_{t-1} + \tilde{\epsilon}_t,
\end{align*}
where \(\gamma^T = (\Gamma_{1j})_{2\le j\le d}\) and \(\tilde{\Gamma} = (\Gamma_{ij})_{2\le i,j\le d}\). The parameter of interest is \(\gamma\). The standard approach is to compute the least squares estimator \(\hat{\gamma}\) and base inference on the \(t^2\)-statistic. Unfortunately, we encounter the same issues as described above. To see this, let \(\Sigma_Y = \Sigma_{11}\), \(\Sigma_X = (\Sigma_{i, j})_{2\le i,j\le d}\), \(\Sigma_{YX} = (\Sigma_{1, j})_{2\le j\le d}\), and \(\Sigma_{XY}=\Sigma_{YX}^T\) and define \(\delta = \Sigma_{X}^{-1}\Sigma_{XY}\). Then, adopting previous notation, Theorem \ref{thm: unif_app} yields
\[
\hat{t}_\gamma^2 \rightarrow_w  \left(\int_0^1 \tilde{J}_{t,C}dB_{1,t}\right)^T\left(\int_0^1 \tilde{J}_{t, C}\tilde{J}_{t, C}^T dt\right)^{-1}\int_0^1 \tilde{J}_{t, C} dB_{1,t} =: t_\gamma^2
\]
uniformly over \(\Theta_P\) where \(\tilde{J}_{t, C}\) consists of the last \(d-1\) coordinates of \(\hat{J}_{t, C}\) and \(B_{1, t}\) is the first coordinate and \(B_{2, t}\) the last \(d-1\) coordinates of \(W_t \hat{\Sigma}^{\frac{1}{2}}\). 
Since \(B'_{1,t} = (\Sigma_Y - \delta^T \Sigma_X\delta)^{-\frac{1}{2}}(B_{1, t} - \delta B_{2, t})\) is a standard \((d-1)\)-dimensional Brownian motion independent of \(B_{2, t}\) satisfying \(B_{1, t} = \delta B_{2, t} + (\Sigma_Y - \delta^T \Sigma_X\delta)^{\frac{1}{2}}B'_{1, t}\), we find that 
\begin{equation} \label{eq: pred_asym}
    t_\gamma^2 \peq \left|\left|\left(\Sigma_Y - \delta^T \Sigma_{X}\delta\right)^{\frac{1}{2}}Z + Z_{\tilde{\Gamma}}\delta\right|\right|^2,
\end{equation}
where \(Z_{\tilde{\Gamma}}=(\int \tilde{J}_{t, C}\tilde{J}_{t, C}^Tdt)^{-\frac{1}{2}}\int \tilde{J}_{t, C}dB_{2, t}\), and \(Z\) is a \((d-1)\)-dimensional standard normal vector independent of \(Z_{\tilde{\Gamma}}\). The nuisance parameter, \(C=C_n(\theta)\), is therefore also present in the distribution of \(t^2_\gamma\) via \(Z_{\tilde{\Gamma}}\) necessitating alternative methods of inference. Using the results of Theorem \ref{thm: valid_cr}, we adopt the univariate Bonferroni approach of \cite{campbell2006efficient} to obtain uniformly asymptotically valid confidence intervals. Say we want to find a confidence region for \(\gamma\) with significance level \(\alpha\in (0, 1)\). For \(\alpha_1, \alpha_2\in (0, 1)\) with \(\alpha_1 + \alpha_2 = \alpha\), the construction proceeds as follows: First construct a \(100(1-\alpha_1)\%\) confidence region for \(\tilde{\Gamma}\) using, e.g.,  either \(CR(\alpha_1) = CR_a(\alpha_1)\) or \(CR(\alpha_1) = CR_b(\alpha_1)\) (suitably modified for \(d-1\) dimensions). Then, for each \(\tilde{\Gamma}\in CR(\alpha_1)\), let \(CR_{\gamma|\tilde{\Gamma}}(\alpha_2)\) be a \(100(1-\alpha_2)\%\) confidence region for \(\gamma\) given \(\tilde{\Gamma}\). A confidence region not depending on \(\tilde{\Gamma}\) and with coverage of at least \(100(1-\alpha)\%\) is then obtained via a Bonferroni correction:
\begin{equation} \label{eq: ci_pr}
CR_\gamma(\alpha_1, \alpha_2) = \bigcup_{\tilde{\Gamma}\in CR(\alpha_1)}CR_{\gamma|\tilde{\Gamma}}(\alpha_2).    
\end{equation}
Let \(\hat{\gamma}_{\tilde{\Gamma}}\) be the estimator obtained by regressing \(Y_{\tilde{\Gamma}, t} = Y_t - \hat{\Sigma}_{YX}\hat{\Sigma}_X^{-1}(\tilde{X}_t - \tilde{\Gamma}\tilde{X}_{t-1})\) on \(\tilde{X}_{t-1}\) with standard error \(\hat{\sigma}^2_Y = \hat{\Sigma}_Y - \hat{\Sigma}_{YX}\hat{\Sigma}_X^{-1}\hat{\Sigma}_{XY}\). A choice for \(CR_{\gamma|\tilde{\Gamma}}(\alpha_2)\) is then given by
\begin{equation} \label{eq: gam_cond_cr}
    CR_{\gamma|\tilde{\Gamma}}(\alpha_2) = \left\{\gamma : \hat{\sigma}_Y^{-2}\hat{t}_{\gamma|\tilde{\Gamma}}^2 \le q_{d-1, 1-\alpha_2}\right\},
\end{equation}
with \(q_{d-1, 1-\alpha_2}\) denoting the \(1-\alpha_2\) quantile of the \(\chi^2_{d-1}\) distribution and \(\hat{t}^2_{\gamma|\tilde{\Gamma}}\) the usual \(t^2\)-statistic for the estimator \(\hat{\gamma}_{\tilde{\Gamma}}\) evaluated at \(\gamma\). A proof of the following is given in Appendix \ref{app: cr}.

\begin{lemma} \label{lem: ci_pr}
    For fixed significance levels \(\alpha_1, \alpha_2\in (0, 1)\) with \(\alpha_1 + \alpha_2\in (0, 1)\), let \(CR_\gamma(\alpha_1, \alpha_2)\) be the confidence interval given in \eqref{eq: ci_pr} with \(CR(\alpha_1)\) having uniform asymptotic level and \(CR_{\gamma|\tilde{\Gamma}}(\alpha_2)\) as given in \eqref{eq: gam_cond_cr}. Then, under Assumptions \ref{ass: M} and \ref{ass: U},
    \[
    \liminf_{n\rightarrow\infty}\inf_{\theta\in \Theta_P}\mathbb{P}\left(\gamma\in CR_\gamma(\alpha_1, \alpha_2)\right) \ge 1 - \alpha_1 - \alpha_2.
    \]
\end{lemma}

\begin{remark} \label{rem: pr_bonf}
    Lemma \ref{lem: ci_pr} is easy to extend to hypothesis testing. Say, e.g., that we want to test the null of no predictive information in the regressor, \(H_0:\gamma = 0\), versus the alternative \(H_A:\gamma \neq 0\). This is equivalent to checking whether \(0\in CR_\gamma(\alpha_1, \alpha_2)\). Alternatively, the test \(\phi_n:(\mathbb{R}^d)^n\rightarrow \{0, 1\}\) given by \(\phi_n = \mathbf{1}\left(\inf_{\tilde{\Gamma}\in CR(\alpha_1)}\hat{\sigma}^{-2}_Y \hat{t}^2_{0|\tilde{\Gamma}} \le q_{d-1, 1-\alpha_2}\right)\) has asymptotic uniform level and does not require the explicit computation of the confidence regions \(CR_{\gamma|\tilde{\Gamma}}(\alpha_2)\).
\end{remark}

\subsection{Lag agumentation} \label{sec: lag_aug}

Other approaches have been suggested to be robust against deviations from exact unit root assumptions. The first one is the lag-augmented VAR methodology proposed by \cite{dolado1996making, toda1995statistical}. In our setup of VAR(1) processes this approach regresses \(X_t\) on \(X_{t-1}\) and the additional augmented lag \(X_{t-2}\) upon which standard inference methodology is valid. Say, for example, that we are interested in testing the hypothesis \(H_0:A\text{ vec}(\Gamma) = b\) for \(A\in\mathbb{R}^{k\times d^2}\) of rank \(k\le d^2\) and \(b\in\mathbb{R}^k\). Let \(\hat{\Pi}_{LA} \in \mathbb{R}^{d\times 2d}\) denote the least squares estimator in the lag-augmented regression of \(X_{t}\) on \(\bar{X}_t = (X_{t-1}^T, X_{t-2}^T)^T\). Denote \(D = (I_d, 0)^T\) and define \(\hat{\Gamma}_{LA} = \hat{\Pi}_{LA}D\) along with the Wald-statistic
\begin{equation}\label{eq: t_la}
    \hat{t}^2_{LA, A, b} = \frac{n (A\text{vec}(\hat{\Gamma}_{LA}) - b)^2 }{\hat{\sigma}^2_{LA, A}},
\end{equation}
where \(\hat{\sigma}^2_{LA, A} = A\hat{\Sigma}_{LA}A^T\) and \(\hat{\Sigma}_{LA} = \hat{\Sigma}^{-1}\otimes \hat{\Sigma}\). Then, under the null, \(\hat{t}^2_{LA, A, b}\) converges in distribution to \(\chi^2_{k}\) uniformly over \(\Theta\) allowing for construction of uniformly valid confidence intervals and tests (see Appendix \ref{app: lag}). For example, a \((1-\alpha)100\%\) confidence interval for \(\Gamma_{i j}\) is given by
\begin{equation} \label{eq: ci_la}
    CI_{LA, ij}(\alpha) = \hat{\Gamma}_{LA, ij} \pm z_{1-\alpha/2}\frac{\hat{\sigma}_{LA, ij}}{\sqrt{n}},
\end{equation}
where \(z_{1-\alpha/2}\) is the \(1-\alpha/2\) standard normal quantile.

\begin{remark} \label{rem: la}
    The key ingredient that facilitates standard inference is the fact that \(\sqrt{n}(\hat{\Gamma}_{LA} - \Gamma)\) converges uniformly in distribution over \(\Theta\) to a family of \(d\)-dimensional Gaussians (see Lemma \ref{lem: la_inf}). In particular, there is no need for normalization, since all components converge at the same rate \(O(\sqrt{n})\). This is contrary to the limiting behaviour of \(\hat{\Gamma}\) which needs to be normalized by the matrix \(H^{-\frac{1}{2}}\) since the presence of roots close to unity make certain parts of \(\hat{\Gamma}\) super efficient. This also suggests some loss of efficiency when using lag augmentation, which does not come as a surprise since we are essentially overfitting the model.
\end{remark}

\subsection{IVX} \label{sec: ivx}

Another approach, known as IVX, that deals specifically with the potential presence of unit roots uses endogenously constructed instrumental variables to slow down the rate of convergence of the estimator enough to ensure mixed Gaussian limiting distributions. It was first suggested by \cite{phillips2009econometric} and later extended in \cite{kostakis2015robust, magdalinos2020econometric}. The most general framework considered so far was proposed in \cite{magdalinos2020econometric}. However, they make the crucial assumption that all roots converge to unity at the same speed. While this allows for easy construction of confidence intervals of general linear functions of \(\Gamma\) and simplifies the theory somewhat, this is a significant restriction. In particular, it does not yield uniform guarantees as the ones discussed in this paper. This excludes, for example, the mixed regime discussed in Section \ref{sec: mixed} covering cases where parts of the process are stationary and others exhibit random walk behaviour. In this section we detail how one may achieve truly uniform results. This comes at the cost of less general confidence regions, which is essentially because we need to employ different normalizations depending on how close the different roots are to unity. This is akin to using \(\hat{t}^2_{\Gamma}\) for inference.

The idea of IVX is simple. We achieve Gaussian asymptotics by constructing an endogenously generated instrument that lies in the stationary regime and then perform IV-regression. For some fixed \(\beta\in(1/2, 1)\), we define, for each \(n\in\mathbb{N}\), the instrument \((Z_t)_{t\in\mathbb{N}}\) by
\[
Z_t = (1-n^{-\beta})Z_{t-1} + \Delta X_t, \quad Z_0 = 0,
\]
where we have suppressed the dependence on \(n\) in the notation. For each \(n\in\mathbb{N}\), \(Z_t\) is a VAR(1) process with the error terms given by \(\Delta X_t\) and the sequence of coefficients, \((1-n^{-\beta})I\), fall inside the stationary regime. The IVX estimator is the IV estimator of regressing \(X_t\) on \(X_{t-1}\) and the instrumental variable \(Z_{t-1}\), i.e.,
\[
\hat{\Gamma}_{IV} = \sum_{t=1}^n X_tZ_{t-1}^T\left(\sum_{t=1}^n X_{t-1}Z_{t-1}^T\right)^{-1}.
\]
The corresponding \(t^2\)-statistic for testing the null \(H_0:\Gamma = \Gamma_0\) is then given by
\[
\hat{t}^2_{IV, \Gamma_0} = \tr\left(n\hat{\Sigma}^{-\frac{1}{2}}\left(\hat{\Gamma}_{IV} - \Gamma_0\right)S_{XZ}S_{ZZ}^{-1}S_{ZX}\left(\hat{\Gamma}_{IV} - \Gamma_0\right)^T\hat{\Sigma}^{-\frac{1}{2}}\right),
\]
where \(S_{XZ}\) and \(S_{ZZ}\) are defined analogously to \(S_{XX}\). It turns out that inference based on this statistic is standard. In particular, \(\hat{t}^2_{IV, \Gamma}\) has the standard asymptotic \(\chi^2_{d^2}\) distribution uniformly over the parameter space \(\Theta\) (see Appendix \ref{app: ivx} for more details) and therefore, for fixed \(\alpha\in(0, 1)\), a confidence region with asymptotic uniform level is given by
\[
    CR_{IV}(\alpha) = \left\{\Gamma : \hat{t}^2_{IV, \Gamma}\le q_{d^2, 1-\alpha}\right\}.
\]

\begin{remark}
    The use of the instrumental variable \(Z_t\) simplifies inference. Indeed, the asymptotic distribution is standard and there is no need for extensive simulations. There is, however, some loss of efficiency. The estimator \(\hat{\Gamma}_{IV}\) converges at rate of \(n^{\beta}\) or slower. If \(\beta\) is close to one or if all the roots converge to unity at rate that is slower than \(n^{\beta}\), this will not be a problem, but in general \(\hat{\Gamma}\) is a more efficient estimator of \(\Gamma\).
\end{remark}

\section{Simulations}

In this section we investigate the finite sample properties of the methods described in the preceding section. First, we consider the problem of constructing a confidence interval for \(\Gamma_{ij}\). Testing whether \(X_{j, t}\) Granger causes \(X_{i, t}\), which in this case amounts to testing the null \(H_0: \Gamma_{ij} = 0\), is then equivalent to checking if 0 is contained in the confidence interval. Since there is nothing special about the choice of \(i\) and \(j\) we choose to focus on \(\Gamma_{11}\) for simplicity. The second problem we consider is that of testing \(H_0:\gamma = 0\) in the predictive regression model.

\subsection{Confidence intervals}

Throughout we fix the significance level at \(\alpha = 0.05\). We compare three different ways of constructing confidence intervals for \(\Gamma_{11}\). The first is lag-augmentation yielding the confidence interval \(CI_{LA}(\alpha)=CI_{LA, 11}(\alpha)\), equation \eqref{eq: ci_la}. The other two methods first compute a confidence region for the entire matrix \(\Gamma\) and then find a confidence interval for \(\Gamma_{11}\) by projecting the confidence region onto the first coordinate. That is, for a given confidence region of \(\Gamma\) with \((1 - \alpha)100\%\) coverage, \(CR(\alpha)\), we obtain the projected confidence interval 
\[
CI(\alpha)= \left(\inf_{\Gamma\in CR(\alpha)}\Gamma_{11}, \sup_{\Gamma\in CR(\alpha)}\Gamma_{11}\right).
\]
We let \(CI_b(\alpha)\) (respectively \(CI_{IV}(\alpha)\)) be the confidence interval obtained by projecting \(CR_b(\alpha)\) (respectively \(CR_{IV}(\alpha)\)). Of these, \(CI_b\) is by far the most costly to compute as the dimension increases. The constraint in the optimization problem is costly to evaluate due to the need for simulations to compute the critical value \(\tilde{q}_{n, \Gamma}(1-\alpha)\) at each \(\Gamma\). This issue can, however, be partly resolved by the use of the EAM-algorithm \citep{kaido2019confidence}. See Appendix \ref{app: sim_algo} for details. \(CI_{IV}\) also involves an optimization problem, but the constraint is cheap to evaluate since the critical value in \(CR_{IV}\) is fixed and standard. \(CI_{IV}\) can be computed by standard solvers. We let \(\beta=0.9\) in the IVX regression.

To verify that the confidence intervals are truly uniform, we look at choices of \(\Gamma\) with eigenvalues in different regimes. In particular, for a fixed dimension \(d\) and sample size \(n\), we consider \(\Gamma\in\mathbb{R}^{d\times d}\) with eigenvalues \(\lambda_1(\Gamma) = 1\) and \(\lambda_i(\Gamma) = 1 - (1/n)^{1/(i-1)}\) for \(i=2,...,d\). For each simulation, we draw a new set of random eigenvectors and the errors are i.i.d. Gaussian with non-diagonal covariance matrix. For a detailed explanation of the setup, see Appendix \ref{app: sim_ci}. The results are recorded in Table \ref{tab: CI_tab}.

\begin{table}
\caption{Coverage and median length of \(C_b\), \(C_{IV}\), and \(CI_{LA}\).}
\centering
\noindent
\begin{tabular}{N N N N N N N }\toprule
\multicolumn{1}{N }{\textbf{}} & \multicolumn{3}{c }{Coverage} & \multicolumn{3}{c }{Median Length} \\
\cmidrule(lr){2-4}
\cmidrule(ll){5-7}
\multicolumn{1}{ N }{\(d\)} &  \(CI_b\) &  \(CI_{IV}\) &    \(CI_{LA}\) &  \(CI_b\) &  \(CI_{IV}\) &    \(CI_{LA}\) \\
\cmidrule(lr){1-1}
\cmidrule(lr){2-4}
\cmidrule(ll){5-7}
\multicolumn{1}{ c }{ } & \multicolumn{6}{ c }{\(n = 50\)} \\[1ex]
\multicolumn{1}{ N }{3} & 0.996 &     0.999 & 0.962 &     0.716 &     0.706 & 0.861 \\
\multicolumn{1}{ c }{4} & 0.999 &     1.000 & 0.971 &     1.090 &     1.075 & 0.948 \\
\multicolumn{1}{ c }{5} & 0.999 &     0.999 & 0.950 &     1.445 &     1.426 & 1.014 \\[1.5ex]
\multicolumn{1}{ c }{ } & \multicolumn{6}{ c }{\(n = 75\)} \\[1ex]
\multicolumn{1}{ c }{3} & 0.998 &     0.998 & 0.976 &     0.506 &     0.491 & 0.699 \\
\multicolumn{1}{ c }{4} & 0.999 &     0.999 & 0.973 &     0.786 &     0.756 & 0.758 \\
\multicolumn{1}{ c }{5} & 1.000 &     1.000 & 0.973 &     1.093 &     1.051 & 0.794 \\[1.5ex]
\multicolumn{1}{ c }{ } & \multicolumn{6}{ c }{\(n = 100\)} \\[1ex]
\multicolumn{1}{ c }{3} & 1.000 &     1.000 & 0.975 &     0.404 &     0.386 & 0.597 \\
\multicolumn{1}{ c }{4} & 0.999 &     0.999 & 0.971 &     0.653 &     0.620 & 0.648 \\
\multicolumn{1}{ c }{5} & 0.996 &     0.997 & 0.970 &     0.910 &     0.862 & 0.681 \\
\bottomrule
\end{tabular}
\label{tab: CI_tab}
\end{table}

All three confidence intervals have coverage greater than 0.95 in every case. As expected, both \(CI_b\) and \(CI_{IV}\) are conservative with practically a 100\% coverage. This is already apparent in 3 dimensions. Despite the loss of efficiency, however, both yield shorter intervals than \(CI_{LA}\) in 3 dimensions for all three sample sizes. This can most likely be attributed to \(\Gamma\) having multiple roots close to unity, implying that the lag-augmented estimator converges at a rate slower than the IVX and the LS estimators. This advantage more or less vanishes in 4 dimensions and in 5 dimensions \(CI_{LA}\) is the clear winner. Intuitively, the dimension of the confidence regions is quadratic in \(d\) and the loss suffered by projection methods therefore quickly sets in. Interestingly, this phenomenon seems less pronounced for higher sample sizes. Another key result in Table \ref{tab: CI_tab} is that \(CI_b\) is wider than \(CI_{IV}\) in every case. This is counter to the fact that the least squares estimator should be more efficient. One possible explanation is that \(\Gamma\) has roots that converge to 1 at a slower rate than \((1 / n)^{\beta}\) limiting the loss of efficiency of the IVX estimator. Another possible explanation is that finite sample behaviour of \(\hat{t}^2_{IV}\) is different from the asymptotic \(\chi^2\)-distribution for the sample sizes considered here, resulting in a confidence region for \(\Gamma\) with slightly lower coverage but still with conservative coverage when projected onto the first coordinate.

\subsection{Predictive regression testing}

Fix \(\alpha = 0.1\). We compare three methods to test \(H_0:\gamma = 0\) against the alternative \(H_A:\gamma\neq 0\) in the predictive regression model. The first one uses lag-augmentation and is based on the test-statistic in equation \eqref{eq: t_la}. We denote this test by \(\phi_{LA}\). The other two tests employ the Bonferroni strategy described in Section \ref{sec: pr}. In particular, for \(\alpha_1 = 0.05\), \(\alpha_2=0.05\), and \(CR(\alpha_1)\) a confidence region for \(\tilde{\Gamma}\) with uniform asymptotic level \(\alpha_1\), we define the test \(\phi_n = \mathbf{1}(\inf_{\tilde{\Gamma}\in CR(\alpha_1)}\hat{\sigma}^{-2}_Y\hat{t}^2_{0|\tilde{\Gamma}} \le q_{d-1, 1-\alpha_2})\). We consider \(CR(\alpha_1) = CR_b(\alpha_1)\) and \(CR(\alpha_1) = CR_{IV}(\alpha_1) \) denoting the corresponding tests by \(\phi_b\) and \(\phi_{IV}\).\footnote{Here, \(CR_b(\alpha_1)\) and \(CR_{IV}(\alpha_1)\) are confidence regions for \(\tilde{\Gamma}\) and not the entire matrix \(\Gamma\).} By Lemma \ref{lem: ci_pr}, the tests will have uniform asymptotic level. Computing the two latter test statistics involves an optimization problem. As in the case of the projection confidence intervals, it is much costlier to compute \(\phi_b\) and we employ the EAM-algorithm described in Appendix \ref{app: sim_algo} as a practically feasible solution. 

\begin{figure}
    \centering
    \includegraphics[width=400pt]{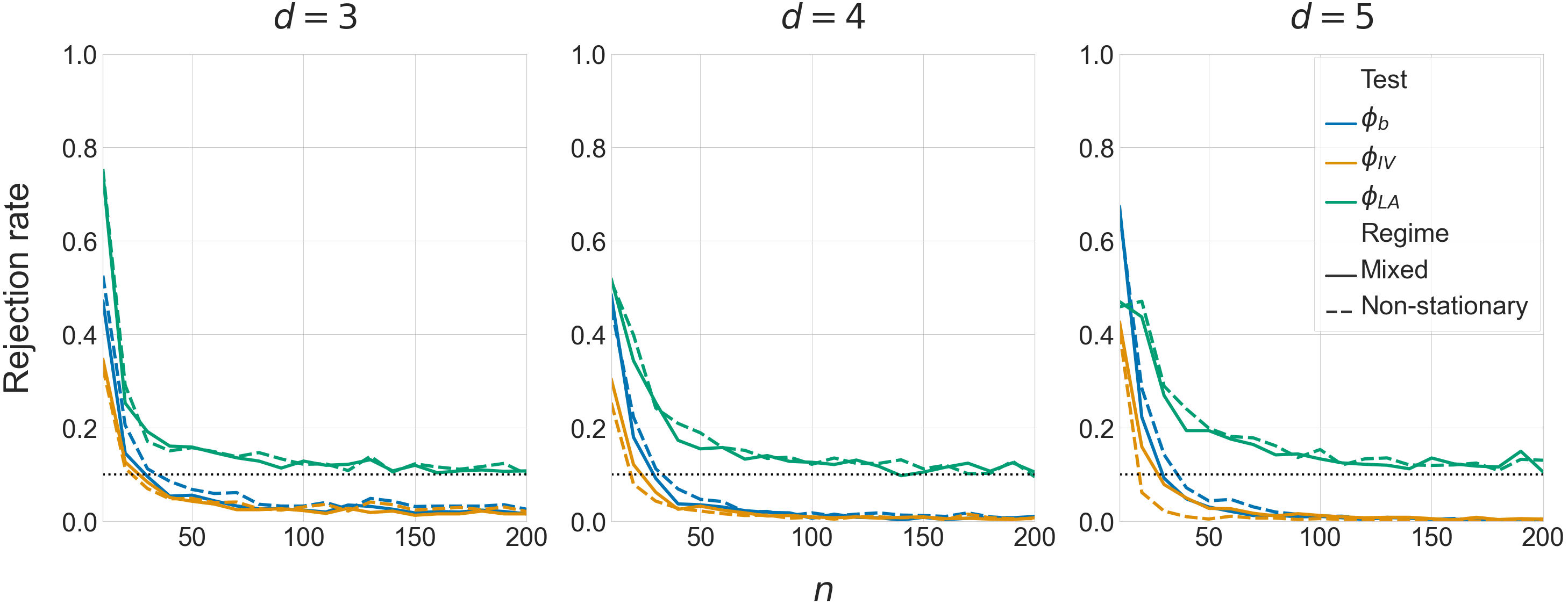}
    \caption{\small Rejection rates under the null \(H_0:\gamma = 0\) of the three tests at different sample sizes and dimensions and under two different regimes. The significance level is fixed at \(\alpha=0.1\) for all \(n\) and \(d\), given by the dotted line. The rejection rate is the proportion of times the null was rejected over 1000 simulations.}
    \label{fig:pr_level}
\end{figure}

\begin{figure}
    \centering
    \includegraphics[width=400pt]{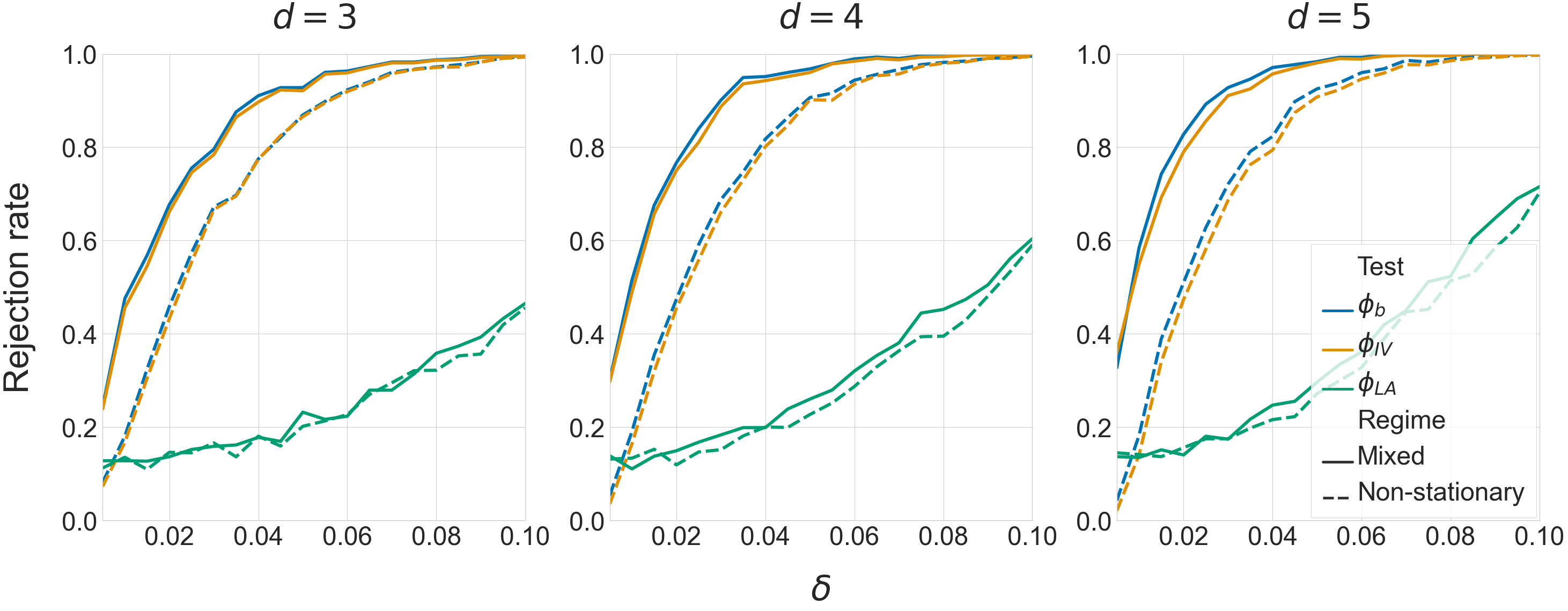}
    \caption{\small Rejection rates of the null \(H_0:\gamma = 0\) with \(\gamma = \delta\mathbbm{1}\) under different dimensions, regimes and for a sequence of alternatives. Sample size is fixed at \(n=100\), and the significance level is \(\alpha=0.1\). The rejection rate is the proportion of times the null was rejected over 1000 simulations. The dimension refers to the dimension of \(\tilde{X}_t\).}
    \label{fig:pr_power}
\end{figure}

We perform two sets of simulation experiments for the three tests. To verify that the uniform guarantees hold, we consider a sequence of \(\tilde{\Gamma}\) in the mixed regime. Throughout, we let \(d=4, 5, 6\) and \(\tilde{\Gamma}\in\mathbb{R}^{(d-1)\times (d-1)}\) is chosen as above. We also consider the case \(\tilde{\Gamma}=I\) so that \(\tilde{X}_t\) is a random walk, i.e., \(\tilde{\Gamma}\) is in the non-stationary regime. First we investigate the size properties of the three tests for different sample sizes. The results are depicted in Figure \ref{fig:pr_level}. Evidently, \(\phi_b\) and \(\phi_{IV}\) quickly achieve a rejection rate well below the significance level for all three dimensions and in both regimes. This is in line with the theory since the Bonferroni corrections inherent in these tests result in tests with conservative sizes. For \(\phi_{LA}\) the asymptotics take a little longer to set in. At around \(n=150\) it achieves the correct size for 3 and 4 dimensions, but convergence is slower in 5 dimensions. 

To compare the power of the three tests at finite sample sizes we consider a sequence of alternatives increasingly closer to \(0\). In particular, we let \(\gamma = \delta\mathbbm{1}\) for \(\delta\in\{0.005, 0.01,..., 0.1\}\) and fix the sample size at \(n=100\). In both regimes and for all three choices of \(d\), \(\phi_b\) and \(\phi_{IV}\) vastly outperform \(\phi_{LA}\) correctly rejecting the null around 90\% of the time in the mixed regime for \(\delta=0.04\) compared to a rejection rate of only around 18\% for \(\phi_{LA}\). It looks as though \(\phi_b\) is slightly better than \(\phi_{IV}\) especially as the dimension increases, but the two are close overall. Interestingly, both tests seem to fare better in the mixed regime than in the stationary regime although we would expect the LS and the IVX estimator to converge at a slower rate in the mixed regime. This might be related to the Bonferroni correction and the shape of the confidence regions in the different regimes. The power might be improved upon by using more efficient corrections than Bonferroni, see e.g. \cite{elliott2015,jansson2006optimal}. The performance of \(\phi_{LA}\) does not depend much on the regime, but it does seem to slightly improve with the size of the dimension. The latter observation also holds for the other two tests and is probably a reflection of the fact that the alternative is easier to detect in larger dimensions.

\section{Conclusion}

We proved two major uniform asymptotic results for the sample covariance matrices of VAR(1) processes with potential unit roots. First we showed that \(S_{XX}\) and \(S_{X\epsilon}\) can be uniformly approximated by their Gaussian counterparts \(S_{YY}\) and \(S_{Y\rho}\). This result was used to derive another uniform approximation involving integrals of Ornstein-Uhlenbeck processes. While uniform asymptotic results akin to those presented here have been given in the literature for specific sequences of \(\Gamma\), this is the first time anything has been proven that is truly uniform over the parameter space \(\Theta\).

As an application of the uniform approximation results, we showed how to construct confidence regions for \(\Gamma\) with uniform asymptotic level. Similarly, we proved that the IVX methodology and lag augmentation also lead to uniformly valid inference if done properly.

%%%%%%%%%%%%%%%%%%%%%%%%%%%%%%%%%%%%%%%%%%%%%%
%% Single Appendix:                         %%
%%%%%%%%%%%%%%%%%%%%%%%%%%%%%%%%%%%%%%%%%%%%%%
%\begin{appendix}
%\section*{???}%% if no title is needed, leave empty \section*{}.
%\end{appendix}
%%%%%%%%%%%%%%%%%%%%%%%%%%%%%%%%%%%%%%%%%%%%%%
%% Multiple Appendixes:                     %%
%%%%%%%%%%%%%%%%%%%%%%%%%%%%%%%%%%%%%%%%%%%%%%
\begin{appendix}

\section{Proofs} \label{app: proofs}

Before presenting the proofs we need three auxiliary results. We assume in this section that Assumptions \ref{ass: M} and \ref{ass: U} are true with the restriction \(F_\theta = I\). The first lemma collects some convergence results related to the normalizing matrix, \(H\), defined in \eqref{eq: HandG}.
\begin{lemma} \label{lem: norm_asym}
    We have
    \begin{enumerate}[label=(\alph*)]
        \item \label{lem: norm_asym_a} \(\liminf_{n\rightarrow\infty}\inf_{\theta\in\Theta} \sigma_{min}(H) > 0\).
        \item \label{lem: norm_asym_b} \(\sup_{R_{n, d}}\sigma_{min}(H)^{-1} = O(n^{-\eta})\).
        \item \label{lem: norm_asym_c} \(\sup_{R_{n, 0}}\sup_{1\le k\le j\le d}(1-|\lambda_{i_j}|)|H_{kj}| = O(1)\).
        \item \label{lem: norm_asym_d} \(\sup_{R_{n, 0}}||\Gamma^t|| \le C (1-\log(n)/n)^t\) for a constant \(C\ge 0\) not depending on \(\theta\) or \(n\).
        \item \label{lem: norm_asym_e} \(\sup_{R_{n, 0}}||\sum_{t=0}^{n-2} \Gamma^t \Sigma (\Gamma^t)^T|| = O(n/\log n)\)
        \item \label{lem: norm_asym_f} \(\sup_{R_{n, 0}}||\sum_{t=0}^{n-2} t\Gamma^t \Sigma (\Gamma^t)^T|| = O(n/\log n)\)
    \end{enumerate}
\end{lemma}

\begin{proof}
    We start with the proof of \ref{lem: norm_asym_a}. We have
    \begin{equation} \label{eq: sig_min}
        H = \frac{1}{n}\sum_{t=1}^n \sum_{s=1}^{t-1}\Gamma^{t-1-s}\Sigma\left(\Gamma^{t-1-s}\right)^T \geq \frac{1}{n}\sum_{t=1}^{n-1} \Sigma
    \end{equation}
    since every matrix in the summand is positive semidefinite. We get
    \[
    \inf_{\theta\in\Theta} \sigma_{min}(H) \ge \inf_{\theta\in\Theta}\frac{1}{n}\sum_{t=1}^{n-1}\sigma_{min}(\Sigma) = \frac{n-1}{n}\inf_{\theta\in \Theta}\sigma_{min}(\Sigma) > \frac{n-1}{n}c
    \]
    where \(c = \inf_{\theta\in\Theta}\sigma_{min}(\Sigma)>0\) by Assumption \ref{ass: sig}.
    
    For the proof of \ref{lem: norm_asym_b}, note that, for any \(n\) large enough and \(\theta\in R_{n, d}\), \(\Gamma\) is diagonal by Assumption \ref{ass: jord} and therefore we have from \eqref{eq: sig_min}
    \begin{align*}
        \sigma_{min}(H)
            &\ge \frac{\sigma_{min}(\Sigma)}{n}\sum_{t=1}^{n-1}\sum_{s=0}^{t-1}|\lambda_{min}(\Gamma)|^{2s} \\
            &\ge \frac{\sigma_{min}(\Sigma)}{n}\sum_{t=1}^{n-1}\frac{1 - (1-n^{-\eta})^{2t}}{1 - (1 - n^{-\eta})^2} \\
            &= \frac{\sigma_{min}(\Sigma)}{n(1-(1-n^{-\eta})^2)}\left(n - 1 -  \sum_{t=1}^{n-1}(1 - n^{-\eta})^{2t}\right) \\
            &= \frac{\sigma_{min}(\Sigma)}{n(1-(1-n^{-\eta})^2)}\left(n - \frac{1 - (1 - n^{-\eta})^{2n}}{1 - (1-n^{-\eta})^2}\right).
    \end{align*}
    For \(n_0\) large enough, we get
    \[
    \frac{1}{n}\left(n - \frac{1 - (1 - n^{-\eta})^{2n}}{1 - (1-n^{-\eta})^2}\right) \ge \frac{1}{2}, \quad \forall n\ge n_0
    \]
    and, thus, with \(C = 2\sup_{\theta\in\Theta}\sigma_{min}(\Sigma)^{-1} < \infty\), we have that, for all \(n\ge n_0\),
    \[
    \sup_{\theta\in R_{n, d}}\sigma_{min}(H)^{-1} \le C\left(1 - \left(1 - n^{-\eta}\right)^2\right).
    \]
    Since \(n^{\eta}(1 - (1 - n^{-\eta})^2)\rightarrow 2\) for \(n\rightarrow \infty\) and all \(\eta > 0\), this proves \ref{lem: norm_asym_b}.

    To prove \ref{lem: norm_asym_c}, fix some \(\theta\in R_{n, 0}\) and note that for all \(1\le k\le j\le d\)
    \[\left|\left(\Gamma^t\Sigma\left(\Gamma^t\right)^T\right)_{kj}\right|\le \left\|\Gamma^t\right\|_{\infty} \left|\sum_{i=1}^d\left(\Gamma^t\Sigma\right)_{ki}\right|\le d^2 \left\|\Gamma^t\right\|_{\infty}\|\Sigma\|_{\infty}\max_{k\le i\le d}\left|\left(\Gamma^t\right)_{ki}\right|.
    \]
    Now define \(c_n=d^2\sup_{\theta\in R_{n, 0}}\max_{t\le n}\|\Gamma^t\|_{\infty}\|\Sigma\|_{\infty}\) and note that \(\limsup_{n\rightarrow\infty}c_n<\infty\). If \(|\lambda_{i_j}|\le 1 -\alpha\), since \(\Gamma\) is Jordan-like and \(|\lambda_{i_j}|\le 1\), we find that 
    \begin{equation}\label{eq: gam_pow_t}
    \left(1-\left|\lambda_{i_j}\right|\right)\max_{k\le i\le d}\left|\left(\Gamma^t\right)_{ki}\right|\le {t \choose d-1}(1-\alpha)^{t - d + 1}
    \end{equation}
    for \(t\ge d\) in which case \ref{lem: norm_asym_c} holds. So assume that \(|\lambda_{i_j}|\ge 1 - \alpha\). Then it follows by assumption \ref{ass: jord} that
    \[
    |H_{kj}| \le c_n \frac{|\Sigma_{kj}|}{n}\sum_{t=1}^{n-1}\sum_{s=0}^{t-1}|\lambda_{i_j}|^s = c_n\frac{|\Sigma_{kj}|}{n(1-|\lambda_{i_j}|)}\left(n - \frac{1-|\lambda_{i_j}|^n}{1 - |\lambda_{i_j}|}\right)
    \]
    so that \ref{lem: norm_asym_c} follows from the fact that \(1-|\lambda_{i_j}|\ge \log(n)/n\) for any \(\theta\in R_{n, 0}\).

    We now prove \ref{lem: norm_asym_d}. By the equivalence of the Frobenius norm and the sup norm there exists some \(C\ge 0\) such that, for any \(\theta\in\Theta\), it holds that
    \[
    \|\Gamma^n\|\le C\|\Gamma^n\|_{\infty}\le C\sum_{k=1}^d\max_{k\le j\le d}\left|\left(\Gamma^n\right)_{kj}\right|.
    \]
    Now, by eq. \eqref{eq: gam_pow_t}, 
    \[
    \lim_{n\rightarrow\infty}\sup_{|\lambda_{i_k}|\le 1 - \alpha}\max_{k\le j\le d}\left|\left(\Gamma^n\right)_{kj}\right|=0
    \]
    and \(|\Gamma^n_{kj}|=|\lambda_{i_k}|^n\delta_{kj}\) for \(|\lambda_{i_k}| > 1 - \alpha\). Thus, there exists a \(C\ge 0\) not depending on \(\theta\) and \(n\) and such that
    \[
    ||\Gamma^t||\le C |\lambda_{max}(\Gamma)|^t
    \]
    for all \(\theta\in\Theta\). The result then follows from \(\sup_{\theta\in R_{n, 0}}|\lambda_{max}(\Gamma)|\le 1 - \log(n)/n\).
    
    For the proof of \ref{lem: norm_asym_e}, part \ref{lem: norm_asym_d} and the fact that \(\Sigma\) is uniformly bounded yield 
    \begin{align*}
      \sup_{\theta\in R_{n, 0}}\left|\left|\sum_{t=0}^{n-2} \Gamma^t \Sigma \left(\Gamma^t\right)^T\right|\right| 
        &\le C \sum_{t=0}^{n-2} \sup_{\theta\in R_{n, 0}} ||\Gamma^t||^2 \le C \sum_{t=0}^{n-2} \left(1 - \frac{\log  n}{n}\right)^{2t} \\
        & \le C \sum_{t=0}^{\infty}\left(1 - \frac{\log  n}{n}\right)^{2t}= \frac{C}{(1-(1-\log(n)/n)^2)} \\
        &= O\left(\frac{n}{\log n}\right).
    \end{align*}
    
    The proof of part \ref{lem: norm_asym_f} is almost the same. Indeed, by the same chain of inequalities, we find that 
    \[
    \sup_{\theta\in R_{n, 0}}\left|\left|\sum_{t=0}^{n-2} t\Gamma^t \Sigma \left(\Gamma^t\right)^T\right|\right| \le \frac{C(1-\log(n)/n)^2}{(1 - (1-\log(n)/n)^2)^2} = O\left(\frac{n}{\log n}\right).
    \]
\end{proof}

\begin{lemma} \label{lem: norm_exp}
    We have
    \[
    \lim_{n\rightarrow\infty}\sup_{\theta\in R_{n, 0}}\left|\left|H^{-\frac{1}{2}}\left(\sum_{t=0}^{t-2}\Gamma^t \Sigma \left(\Gamma^t\right)^T\right) H^{-\frac{1}{2}} - I\right|\right| = 0.
    \]
\end{lemma}

\begin{proof}
    Note that it suffices to prove that 
    \[
    \lim_{n\rightarrow\infty}\sup_{\theta\in R_{n, 0}}\left|\left|\left(\sum_{t=0}^{t-2}\Gamma^t \Sigma \left(\Gamma^t\right)^T\right)^{-\frac{1}{2}} H \left(\sum_{t=0}^{t-2}\Gamma^t \Sigma \left(\Gamma^t\right)^T\right)^{-\frac{1}{2}} - I\right|\right| = 0.
    \]
    Indeed, for any two positive definite matrices \(A, B\in\mathbb{R}^{d\times d}\), we have 
    \[
    \left\|A^{-\frac{1}{2}}BA^{-\frac{1}{2}} - I\right\| \le \left\|A^{-\frac{1}{2}}B^{\frac{1}{2}}\right\|^2\left\|B^{-\frac{1}{2}}A B^{-\frac{1}{2}} - I\right\|
    \]
    and so, if the second term on the right hand side goes to 0, by Lemma \ref{lem: sqr_id} below, the term on the left hand side will also go to 0.
    
    We have
    \begin{align*}
        H 
            &= \frac{1}{n}\sum_{t=1}^n \sum_{s=1}^{t-1}\Gamma^{t-1-s}\Sigma\left(\Gamma^{t-1-s}\right)^T    \\
            &= \frac{1}{n}\sum_{t=0}^{n-2}(n - 1 - t)\Gamma^t \Sigma \left(\Gamma^t\right)^T \\
            &= \frac{n - 1}{n}\sum_{t=0}^{n-2} \Gamma^t \Sigma \left(\Gamma^t\right)^T - \frac{1}{n}\sum_{t=1}^{n-2}t \Gamma^t \Sigma \left(\Gamma^t\right)^T.
    \end{align*}
    and, as a result, 
    \[
    \left(\sum_{t=0}^{t-2}\Gamma^t \Sigma \left(\Gamma^t\right)^T\right)^{-\frac{1}{2}} H \left(\sum_{t=0}^{t-2}\Gamma^t \Sigma \left(\Gamma^t\right)^T\right)^{-\frac{1}{2}} = \frac{n-1}{n}I - M
    \]
    where
    \[
    M = \left(\sum_{t=0}^{t-2}\Gamma^t \Sigma \left(\Gamma^t\right)^T\right)^{-\frac{1}{2}}\frac{1}{n}\sum_{t=1}^{n-2}t \Gamma^t \Sigma \left(\Gamma^t\right)^T \left(\sum_{t=0}^{t-2}\Gamma^t \Sigma \left(\Gamma^t\right)^T\right)^{-\frac{1}{2}}.
    \]
    All that is left to show is therefore that \(M\) goes to 0 uniformly over \(R_{n, 0}\) for \(n\) going to infinity. Since each term in the sum is positive definite, we have, by Assumption \ref{ass: sig},
    \[
    \liminf_{n\rightarrow\infty} \inf_{\theta\in\Theta} \sigma_{min}\left(\sum_{t=0}^{t-2}\Gamma^t \Sigma \left(\Gamma^t\right)^T\right) \ge \liminf_{n\rightarrow \infty}\inf_{\theta\in\Theta}\sigma_{min}(\Sigma) > 0
    \]
    so that, by equivalence of the spectral norm and the Frobenius norm,  
    \[
    \limsup_{n\rightarrow\infty}\sup_{\theta\in \Theta}\left|\left|\left(\sum_{t=0}^{t-2}\Gamma^t \Sigma \left(\Gamma^t\right)^T\right)^{-\frac{1}{2}}\right|\right| < \infty.
    \]
    By part \ref{lem: norm_asym_f} of Lemma \ref{lem: norm_asym}, we  have 
    \[
    \sup_{\theta\in R_{n, 0}}\left\|\frac{1}{n}\sum_{t=1}^{n-2}t \Gamma^t \Sigma \left(\Gamma^t\right)^T\right\| = o(1).
    \]
    Combining these results, we obtain
    \[
    \sup_{\theta\in R_{n, 0}}||M|| = o(1).
    \]
\end{proof}

The next lemma says that checking whether \(A_nB_n\) converges to the identity matrix is the same as checking whether \(A_n B_n^2 A_n\) converges to the identity matrix where \(A_n\) and \(B_n\) are positive semidefinite matrices of conforming dimension.

\begin{lemma} \label{lem: sqr_id}
    Let \(\mathcal{I}\) be some index set and consider two families of sequences of positive semidefinite \(d\times d\) matrices, \((A_{n, i})_{n\in\mathbb{N}, i\in \mathcal{I}}\) and \((B_{n, i})_{n\in\mathbb{N}, i\in \mathcal{I}}\). Then, if
    \[
    \lim_{n\rightarrow\infty}\sup_{i\in \mathcal{I}}\left|\left|A_{n, i}B_{n, i}^2 A_{n, i} - I\right|\right| = 0,
    \]
    it also holds that
    \[
    \lim_{n\rightarrow\infty}\sup_{i\in \mathcal{I}}\left|\left|A_{n, i}B_{n, i} - I \right|\right| = 0.
    \]
\end{lemma}

\begin{proof}
    First, note that 
    \[
    \liminf_{n\rightarrow \infty}\inf_{i\in \mathcal{I}}\lambda_{min}(A_{n, i}B_{n, i}) = c > 0.
    \]
    Since \(A_{n, i}B_{n, i}\) is similar to the positive definite matrix \(A_{n, i}^{\frac{1}{2}}B_{n, i}A_{n, i}^{\frac{1}{2}}\), the contrary would imply the existence of a sequence \((i_n)_{n\in\mathbb{N}}\subset \mathcal{I}\) such that \(\sigma_{min}(A_{n, i_n}B_{n, i_n})\rightarrow 0\) and, thus, \(\sigma_{min}(A_{n, i_n}B_{n, i_n}^2 A_{n, i_n})\rightarrow 0\) for \(n\rightarrow \infty\) which, of course, is a contradiction.
    
    Now, let \(A_{n, i}B_{n, i} = U_{n, i}P_{n, i}\) be a polar decomposition, i.e., \(U_{n, i}\) is orthogonal and \(P_{n, i}\) positive semidefinite. Then, since \(P_{n, i}=U_{n, i}^T A_{n, i}B_{n, i}\), we have \(P_{n, i}^2 = P_{n, i}^T P_{n, i}= A_{n, i}B_{n, i}^2 A_{n, i}\) which implies that
    \[
    \lim_{n\rightarrow\infty}\sup_{i\in \mathcal{I}}\left|\left| P_{n, i} - I\right|\right| = 0.
    \]
    It therefore suffices to show that \(U_{n, i}\) converges to the identity matrix uniformly over \(\mathcal{I}\). Since \(U_{n, i}\) is orthogonal, we get
    \[
    \left|\left|U_{n, i} - I\right|\right|^2 = 2d - 2\tr(U_{n, i}) \le 2d\sup_{1\le j\le d}\left|1 - \lambda_j\left(U_{n, i}\right)\right|.
    \]
    Let \(U_{n, i} = V_{n, i}D_{n, i}V_{n, i}^*\) be an eigendecomposition with \(D_{n, i}\) diagonal and \(V_{n, i}\) unitary. Define the Hermitian matrix \(H_{n, i}=V^*_{n, i}\left(P_{n, i} - I\right)V_{n, i}\). Denote by \(d_{n, i}^j\) the \(j\)'th diagonal of \(D_{n, i}\) and by \(h_{n, i}^{jk}\) the \(jk\)'th element of \(H_{n, i}\). Since \(H_{n, i}\rightarrow 0\) uniformly over \(\mathcal{I}\), for fixed \(\epsilon > 0\), we can pick \(n_0\in\mathbb{N}\) such that
    \[
    \sup_{i\in \mathcal{I}}\sup_{1\le j, k\le d}|h_{n, i}^{jk}| < \epsilon
    \]
    and
    \[
    \inf_{i\in \mathcal{I}} \lambda_{min}\left(A_{n, i}B_{n, i}\right) > \frac{c}{2}
    \]
    for all \(n\ge n_0\). We define the complex disk \(D_r(x)=\{y\in\mathbb{C}: |y-x|\le r\}\) for any \(x\in\mathbb{C}\) and \(r > 0\). The matrix \(D_{n, i} + D_{n, i}H_{n, i}\) is simlar to \(A_{n, i}B_{n, i}\) so, by the Gershgorin circle theorem, for \(1\le k \le d\), there exists \(1\le j \le d\) such that
    \[
    \lambda_k\left(A_{n, i}B_{n, i}\right)\in D_{R_j}\left(d_{n, i}^j + d_{n, i}^j h_{n, i}^{jj}\right)
    \]
    where \(R_j = \sum_{k\neq j}|h_{n, i}^{jk}|\). Recall that \(|d_{n, i}^j| =1\). Using the fact that \(\lambda_k(A_{n, i}B_{n, i})> c/2\) is real and that
    \[
    \sup_{i\in \mathcal{I}}\lambda_k\left(A_{n, i}B_{n, i}\right)^2 \le \sup_{i\in \mathcal{I}}\sigma_{max}\left(A_{n, i}B_{n, i}^2 A_{n, i}\right) \rightarrow 1
    \]
    for \(n\rightarrow \infty\), we may therefore assume that \(n_0\) is large enough so that
    \[
    \sup_{i\in \mathcal{I}}\left|\lambda_k\left(A_{n, i}B_{n, i}\right) - 1\right| \le \sup_{i\in\mathcal{I}}\left|d_{n, i}^j - \lambda_k\left(A_{n, i}B_{n, i}\right)\right| + \epsilon
    \]
    for all \(n \ge n_0\) which implies
    \[
    \sup_{i\in\mathcal{I}}\left|d_{n, i}^j - 1\right| \le 2\sup_{i\in \mathcal{I}}\left|d_{n, i}^j - \lambda_j\left(A_{n, i}B_{n, i}\right)\right| + \epsilon \le 2\sup_{i\in\mathcal{I}}R_j + \epsilon + \epsilon \le 2d\epsilon.
    \]
    Thus,
    \[
    \sup_{i\in\mathcal{I}}\left|\left| U_{n, i} - I\right|\right| \le 4d^2\epsilon
    \]
    for all \(n\ge n_0\). This completes the proof.
\end{proof}

\subsection{Nonstationary asymptotics}

\begin{lemma} \label{lem: l2_app}
    Let \(g(t, s, n, \theta) = G^{-\frac{1}{2}}e^{(t-s)C}\Sigma^{\frac{1}{2}}\). We have
    \begin{equation} \label{eq: l2_yr}
    \lim_{n\rightarrow\infty}\sup_{\theta\in R_{n, d}} \mathbb{E}\left|\left|\int_0^{1}\int_0^t f(t, s, n, \theta) - g(t, s, n, \theta)dW_s dW_t^T\right|\right|^2 = 0
    \end{equation}
    and
    \begin{multline} \label{eq: l2_yy}
    \lim_{n\rightarrow\infty}\sup_{\theta\in R_{n, d}} \mathbb{E}\text{ }\Bigg|\Bigg| \Bigg.\Bigg. \int_0^1\left(\int_0^t f(t, s, n, \theta) dW_s\right)\left(\int_0^t f(t, s, n, \theta) dW_s\right)^T \\
    - \left(\int_0^t g(t, s, n, \theta) dW_s\right)\left(\int_0^t g(t, s, n, \theta) dW_s\right)^T dt \Bigg.\Bigg.\Bigg|\Bigg|^2 = 0.
    \end{multline}
\end{lemma}

\begin{proof}[Proof of Lemma \ref{lem: l2_app}]
    We first prove \eqref{eq: l2_yr}. Define 
    \[
    h_1(t, s, n, \theta) = \sqrt{n} H^{-\frac{1}{2}}\Gamma^{\lfloor nt \rfloor - \lfloor ns \rfloor - 1} \Sigma^{\frac{1}{2}}, \quad h_2(t, s, n, \theta) = \sqrt{n} H^{-\frac{1}{2}}e^{(t-s)\tilde{C}}\Sigma^{\frac{1}{2}}
    \]
    \[
    h_3(t, s, n, \theta) = \sqrt{n} H^{-\frac{1}{2}}e^{(t-s)C}\Sigma^{\frac{1}{2}},
    \]
    where \(\tilde{C} = n\log \Gamma\) is well defined for all \(\theta\in R_{n, d}\). When it does not cause confusion, we shall omit the arguments of functions and simply write \(f\), \(g\), \(h_1\), \(h_2\), and \(h_3\). By applying the Itô isometry twice we find that the expectation in \eqref{eq: l2_yr} is equal to
    \begin{multline*}
        \int_0^{1}\int_0^t \left|\left|f - g\right|\right|^2 ds dt \\ 
        \le 4\int_0^{1}\int_0^t \left|\left|f - h_1\right|\right|^2 + \left|\left|h_1 - h_2\right|\right|^2 + \left|\left|h_2 - h_3\right|\right|^2 + \left|\left|h_3 - g\right|\right|^2 ds dt
    \end{multline*}
    where the inequality is Jensen's inequality. For the first term, for any \(\theta \in R_{n, d}\),
    \begin{align*}
    \int_0^{1}\int_0^t \left|\left|f - h_1\right|\right|^2 ds dt 
        &= \int_0^1\int_{\lfloor nt\rfloor / n}^t \left|\left| \sqrt{n} H^{-\frac{1}{2}}\Gamma^{\lfloor nt \rfloor - \lfloor ns \rfloor - 1} \Sigma^{\frac{1}{2}} \right|\right|^2 ds dt \\
        &= \left(n\int_0^1 t - \frac{\lfloor nt\rfloor}{n} dt\right)\left|\left| H^{-\frac{1}{2}}\Gamma^{-1}\Sigma^{\frac{1}{2}}\right|\right|^2\\
        &\le \left|\left| H^{-\frac{1}{2}}\Gamma^{-1}\Sigma^{\frac{1}{2}}\right|\right|^2 \le \left|\left|\Gamma^{-1}\Sigma^{\frac{1}{2}}\right|\right|^2\tr\left(H^{-1}\right).
    \end{align*}
    Here \(\Gamma^{-1}\) is well-defined since \(\theta\in R_{n, d}\). Thus,
    \[
    \lim_{n\rightarrow\infty} \sup_{\theta\in R_{n, d}} \int_0^{1}\int_0^t \left|\left|f - h_1\right|\right|^2 ds dt \le \lim_{n\rightarrow\infty} \sup_{\theta\in R_{n, d}} \left|\left|\Gamma^{-1}\Sigma^{\frac{1}{2}}\right|\right|^2 \tr\left(H^{-1}\right) = 0
    \]
    where we use that \(\Gamma^{-1}\Sigma^{\frac{1}{2}}\) is uniformly bounded on \(R_{n, d}\) in combination with Lemma \ref{lem: norm_asym}. For the second term, we first note that, due to Assumptions \ref{ass: eig} and \ref{ass: jord}, for \(n\) large enough, we may assume that \(||\Gamma - I||< 1\) for any \(\theta\in R_{n, d}\). We then get \(\Gamma^{\lfloor nt \rfloor - \lfloor ns \rfloor - 1} = e^{(\lfloor nt\rfloor - \lfloor ns\rfloor - 1)\log \Gamma}\) whence
    \begin{align*}
        \left|\left|h_1 - h_2\right|\right|^2 
            &= \left|\left|\sqrt{n} H^{-\frac{1}{2}}e^{(\lfloor nt\rfloor - \lfloor ns\rfloor - 1)\log \Gamma}\left(I - e^{((t-s)n - (\lfloor nt\rfloor - \lfloor ns\rfloor - 1))\log \Gamma}\right)\Sigma^{\frac{1}{2}}\right|\right|^2 \\
            &\le \left|\left|\sqrt{n} H^{-\frac{1}{2}}\Gamma^{\lfloor nt \rfloor - \lfloor ns \rfloor - 1}\right|\right|^2 c(t, s, n, \theta) ^2\left|\left|\log \Gamma\right|\right|^2
    \end{align*}
    where 
    \[
    c(t, s, n, \theta)=\left((t-s)n - (\lfloor nt\rfloor - \lfloor ns\rfloor -1)\right)\left|\left|\Sigma^{\frac{1}{2}}\right|\right|e^{||((t-s)n - (\lfloor nt\rfloor - \lfloor ns\rfloor - 1))\log \Gamma||}
    \]
    and the inequality follows from \(||e^A - e^B|| \le ||A - B|| e^{\max\{||A||, ||B||\}}\) for any \(A, B \in\mathbb{C}^{d\times d}\). By assumptions \ref{ass: eig} and \ref{ass: jord}, we have \(\lim_{n\rightarrow\infty}\sup_{\theta\in R_{n, d}}||\log \Gamma||^2 = 0\) and
    \[
    \limsup_{n\rightarrow \infty} \sup_{\theta\in R_{n, d}}\sup_{t\in[0,1]}\sup_{s \in[0, t]} c(t, s, n, \theta)^2 \le c_0 < \infty.
    \]
    We also have
    \begin{align*}
        \int_0^1\int_0^{\lfloor nt\rfloor /n}\left|\left|\sqrt{n} H^{-\frac{1}{2}}\Gamma^{\lfloor nt \rfloor - \lfloor ns \rfloor - 1}\right|\right|^2 ds 
            &= \frac{1}{n}\sum_{t=2}^n\sum_{s=1}^{t-1}\left|\left|H^{-\frac{1}{2}}\Gamma^{t-1-s}\right|\right|^2 \\
            &\le \left|\left|\Sigma^{-1}\right|\right|\tr\left(H^{-\frac{1}{2}}\mathbb{E}\left(S_{XX}\right) H^{-\frac{1}{2}}\right) \\
            &= d\left|\left|\Sigma^{-1}\right|\right|.
    \end{align*}
    As shown above
    \[
    \limsup_{n\rightarrow\infty}\sup_{\theta\in R_{n, d}}\int_0^1\int_{\lfloor nt\rfloor /n}^t\left|\left|\sqrt{n}H^{-\frac{1}{2}}\Gamma^{\lfloor nt \rfloor - \lfloor ns \rfloor - 1}\right|\right|^2 ds dt = 0
    \]
    so that
    \[
    \limsup_{n\rightarrow\infty}\sup_{\theta\in R_{n, d}}\int_0^1\int_0^t\left|\left|\sqrt{n} H^{-\frac{1}{2}}\Gamma^{\lfloor nt \rfloor - \lfloor ns \rfloor - 1}\right|\right|^2 ds dt \le c_1 < \infty.
    \]
    Combining these results then yields
    \[
    \lim_{n\rightarrow\infty}\sup_{\theta\in R_{n, d}}\int_0^1\int_0^t \left|\left|h_1 - h_2\right|\right|^2 ds dt \le c_0c_1 \lim_{n\rightarrow\infty}\sup_{\theta\in R_{n, d}} \left|\left|\log\Gamma\right|\right|^2 = 0.
    \]
    For the third term, note that \(\tilde{C}\) and \(C\) commute. A similar argument as that applied to the second term then yields
    \[
    \lim_{n\rightarrow\infty}\sup_{\theta\in R_{n, d}}\int_0^1\int_0^t ||h_2 - h_3|| = 0.
    \]
    Finally, for the fourth term, it suffices to show that \(||\sqrt{n}H^{-\frac{1}{2}}G^{\frac{1}{2}}-I||\) converges uniformly over \(R_{n, b}\) to 0 for \(n\) going to infinity. Indeed, since
    \begin{align*}
      \int_0^1\int_0^t\left|\left|h_2 - g\right|\right|^2 
        &\le \left|\left|\sqrt{n}H^{-\frac{1}{2}}G^{\frac{1}{2}}-I\right|\right|^2 \int_0^1\int_0^t \left|\left| G^{-\frac{1}{2}}e^{(t-s)C}\Sigma^{-\frac{1}{2}}\right|\right|^2 ds dt \\
        & = \left|\left|\sqrt{n} H^{-\frac{1}{2}}G^{\frac{1}{2}}-I\right|\right|^2 \tr\left(G^{-\frac{1}{2}}\mathbb{E}\left(\int_0^1J_{t, C}J_{t, C}^T dt\right) G^{-\frac{1}{2}}\right) \\
        &= d\left|\left|\sqrt{n} H^{-\frac{1}{2}}G^{\frac{1}{2}}-I\right|\right|^2,
    \end{align*}
    this would imply 
    \[
    \lim_{n\rightarrow\infty}\sup_{\theta\in R_{n, d}}\int_0^1\int_0^t\left|\left|h_3 - g\right|\right|^2 \le c_1 \lim_{n\rightarrow\infty}\sup_{\theta\in R_{n, d}}||\sqrt{n}H^{-\frac{1}{2}}G^{\frac{1}{2}}-I|| = 0.
    \]
    To prove the claim, we first consider \(nH^{-\frac{1}{2}}GH^{-\frac{1}{2}}\). By the Itô isometry we have
    \[
    nH^{-\frac{1}{2}}GH^{-\frac{1}{2}} = \int_0^1\int_0^t h_3 h_3^T ds dt.
    \]
    Also, similar to above,
    \[
    \int_{0}^1\int_0^t ff^T ds dt =  H^{-\frac{1}{2}}\frac{1}{n}\sum_{t=2}^n\sum_{s=1}^{t-1} \Gamma^{t-1-s}\Sigma\left(\Gamma^{t-1-s}\right)^T H^{-\frac{1}{2}} = I.
    \]
    Thus,
    \[
    \left|\left| nH^{-\frac{1}{2}}GH^{-\frac{1}{2}} - I\right|\right| \le \int_0^1\int_0^t \left|\left|\left(h_3 - f\right)h_3^T\right|\right| ds dt + \int_0^1\int_0^t \left|\left|f\left(h_3 - f\right)^T\right|\right| ds dt.
    \]
    Now, since \(\limsup_{n\rightarrow\infty} \sup_{\theta\in R_{n, d}} \int_0^t ||h_3||^2 + ||f||^2 ds dt < \infty\) and, as was shown above, \(\lim_{n\rightarrow\infty}\sup_{\theta\in R_{n, d}} \int_0^1\int_0^t ||h_3 - f||^2 ds dt = 0\), Hölder's inequality yields
    \[
    \lim_{n\rightarrow\infty}\sup_{\theta\in R_{n, d}} \left|\left| nH^{-\frac{1}{2}}GH^{-\frac{1}{2}} - I\right|\right| = 0.
    \]
    By Lemma \ref{lem: sqr_id} this implies that \(||\sqrt{n}H^{-\frac{1}{2}}G^{\frac{1}{2}}-I||\) converges uniformly over \(R_{n, b}\) to 0 for \(n\) going to infinity.

    For the proof of \eqref{eq: l2_yy} we start with the following chain of inequalities
    \begin{align*}
        &\left|\left|\int_0^1\int_0^t f dW_s\left(\int_0^t f dW_s\right)^T - \int_0^t g dW_s\left(\int_0^t g dW_s\right)^T dt \right|\right| \\
        \le &\left|\left|\int_0^1\int_0^t f dW_s\left(\int_0^t f - g dW_s\right)^Tdt\right|\right| + \left|\left|\int_0^1\int_0^t f - g dW_s\left(\int_0^t g dW_s\right)^Tdt\right|\right| \\
        \le & \int_0^1\left(\left|\left|\int_0^t f dW_s\right|\right| + \left|\left|\int_0^t g dW_s\right|\right|\right)\left|\left|\int_0^t f - g dW_s\right|\right| dt \\
        \le & \left(\int_0^1\left(\left|\left|\int_0^t f dW_s\right|\right| + \left|\left|\int_0^t g dW_s\right|\right|\right)^2 dt\right)^{\frac{1}{2}} \left(\int_0^1\left|\left|\int_0^t f - g dW_s\right|\right|^2 dt\right)^{\frac{1}{2}} \\
        \le & \left(2\int_0^1\left|\left|\int_0^t f dW_s\right|\right|^2 + \left|\left|\int_0^t g dW_s\right|\right|^2 dt\right)^{\frac{1}{2}} \left(\int_0^1\left|\left|\int_0^t f - g dW_s\right|\right|^2 dt\right)^{\frac{1}{2}}
    \end{align*}
    where the second to last inequality is Hölder's inequality. By the Itô isometry and Fubini's theorem we have, for any \(\theta\in R_{n, d}\),
    \[
    \mathbb{E}\left(\int_0^1\left|\left|\int_0^t f dW_s\right|\right|^2 + \left|\left|\int_0^t g dW_s\right|\right|^2 dt\right) = \int_0^1\int_0^t ||f||^2 + ||g||^2 ds dt = 2d
    \]
    and 
    \[
    \mathbb{E}\left(\int_0^1\left|\left|\int_0^t f - g dW_s\right|\right|^2 dt\right) = \int_0^1\int_0^t ||f - g||^2 ds dt
    \]
    so that equation \ref{eq: l2_yy} follows by the same argument as in the proof of equation \eqref{eq: l2_yr}.
\end{proof}

\subsection{Stationary asymptotics}

Before proving Theorem \ref{thm: stat_asym} we need two auxiliary results on the rate of convergence of \(X_{t-1, n}X_{t-1, n}^T\) and \(S_{X\epsilon}\) akin to Lemma 3.1 in \cite{phillips2007limit}.

\begin{lemma} \label{lem: stat_helper}
    For all \(s\in [0, 1]\), we have
    \begin{equation} \label{eq: stat_bound_xx}
        \sup_{\theta\in R_{n, 0}}\left|\left|X_{\lfloor ns\rfloor, \theta}X_{\lfloor ns \rfloor, \theta}^T\right|\right| = o_p\left(\frac{n}{\sqrt{\log n}}\right)
    \end{equation}
    and 
    \begin{equation}\label{eq: stat_bound_xe}
        \sup_{\theta\in R_{n, 0}}\left|\left|\frac{1}{n}\sum_{t=1}^{\lfloor ns\rfloor}X_{t-1, \theta}\epsilon_{t, \theta}^T\right|\right| = o_p\left(\frac{1}{\sqrt{\log n}}\right).
    \end{equation}
\end{lemma}

\begin{proof}
    We first prove \eqref{eq: stat_bound_xx}. Since \(X_{\lfloor ns\rfloor, \theta}X_{\lfloor ns \rfloor, \theta}^T\) is positive semidefinite, it suffices to show that \(\sup_{\theta\in R_{n, 0}} \tr\left(\mathbb{E}\left(X_{\lfloor ns\rfloor, \theta}X_{\lfloor ns \rfloor, \theta}^T\right)\right) = o\left(\frac{n}{\sqrt{\log n}}\right)\) for all \(s\in [0, 1]\). Now, fix some \(s\in[0, 1]\). We have, for all \(\theta\in \Theta\),
    \[
    \tr\left(\mathbb{E}\left(X_{\lfloor ns\rfloor, \theta}X_{\lfloor ns \rfloor, \theta}^T\right)\right) \le \tr\left(\mathbb{E}\left(X_{n, \theta}X_{n, \theta}^T\right)\right) = \tr\left(\sum_{t=0}^{n-1}\Gamma^t \Sigma \left(\Gamma^t\right)^T\right).
    \]
    The result then follows directly from part \ref{lem: norm_asym_e} of Lemma \ref{lem: norm_asym}.

    For the proof of \eqref{eq: stat_asym_xe}, fix some \(\theta\in \Theta\) and write
    \begin{align*}
        \mathbb{E}\left(\left(\sum_{t=1}^n X_{t-1, \theta}\epsilon_{t, \theta}\right)\left(\sum_{t=1}^n X_{t-1, \theta}\epsilon_{t, \theta}\right)^T\right) 
            &=\tr\left(\Sigma\right)\sum_{t=1}^n\sum_{s=1}^{t-1}\Gamma^{t-1-s}\Sigma\left(\Gamma^{t-1-s}\right)^T \\
            &\le\tr\left(\Sigma\right) n \sum_{t=1}^n \Gamma^t \Sigma \left(\Gamma^t\right)^T
    \end{align*}
    so that another application of part \ref{lem: norm_asym_e} of Lemma \ref{lem: norm_asym} shows that 
    \[\sup_{\theta\in R_{n, 0}}\mathbb{E}\left|\left|\frac{1}{n}\sum_{t=1}^n X_{t-1, \theta}\epsilon_{t, \theta}^T\right|\right|^2 = o_p\left(\frac{1}{\log n}\right).
    \]
    The result then follows since \(\mathbb{E}\left|\left|\frac{1}{n}\sum_{t=1}^{\lfloor ns \rfloor} X_{t-1, \theta}\epsilon_{t, \theta}^T\right|\right|^2 \le \mathbb{E}\left|\left|\frac{1}{n}\sum_{t=1}^n X_{t-1, \theta}\epsilon_{t, \theta}^T\right|\right|^2\) for all \(s\in [0, 1]\) and \(\theta \in \Theta\).
\end{proof}

\begin{proof}[Proof of Theorem \ref{thm: stat_asym}]
    We shall first tackle the proof of \eqref{eq: stat_asym_xx}. Fix some \(s\in[0, 1]\) and define \(\tilde{S}_{XX} = \frac{1}{n}\sum_{t=1}^{\lfloor ns\rfloor} X_{t-1, \theta}X_{t-1, \theta}^T\) for ease of notation. From the relation \(X_{t, \theta} = \Gamma X_{t-1, \theta} + \epsilon_{t, \theta}\), it follows that
    \begin{multline*}
        \Gamma X_{t-1, \theta}X_{t-1, \theta}^T\Gamma^T - X_{t-1, \theta}X_{t-1, \theta}^T - \epsilon_{t, \theta}\epsilon_{t, \theta}^T \\
        = X_{t, \theta}X_{t, \theta} - X_{t-1, \theta}X_{t-1, \theta}^T - \Gamma X_{t-1, \theta}\epsilon_{t, \theta}^T - \epsilon_{t, \theta}X_{t-1, \theta}^T\Gamma^T.
    \end{multline*}
    Summing over \(t\) and dividing by \(n\) then gives \(\tilde{S}_{XX} = \Gamma \tilde{S}_{XX} \Gamma^T  + s\Sigma - S_n\) where
    \[
    S_n = \frac{1}{n}\left(X_{n, \theta}X_{n, \theta}^T - \sum_{t=1}^{\lfloor ns\rfloor} \left(\epsilon_{t, \theta}\epsilon_{t, \theta}^T - \Sigma\right) - \sum_{t=1}^{\lfloor ns\rfloor} \Gamma X_{t-1, \theta}\epsilon_{t, \theta}^T - \sum_{t=1}^{\lfloor ns\rfloor}\epsilon_{t, \theta}X_{t-1, \theta}^T\Gamma^T\right).
    \]
    We can iterate this identity to get
    \[
    \tilde{S}_{XX} = \sum_{t=0}^{\lfloor ns\rfloor-2}\Gamma^t\Sigma \left(\Gamma^t\right)^T + \sum_{t=0}^{\lfloor ns\rfloor-2}\Gamma^t S_n \left(\Gamma^t\right)^T + \Gamma^{\lfloor ns\rfloor-1} \tilde{S}_{XX} \left(\Gamma^{\lfloor ns\rfloor-1}\right)^T.
    \]
    Now, define
    \[
    A_n = H^{-\frac{1}{2}}\sum_{t=0}^{\lfloor ns\rfloor-2}\Gamma^t S_n \left(\Gamma^t\right)^T H^{-\frac{1}{2}}, \quad B_n = \Gamma^{\lfloor ns\rfloor-1} \tilde{S}_{XX} \left(\Gamma^{\lfloor ns\rfloor-1}\right)^T. 
    \]
    By Lemma \ref{lem: norm_exp}, it suffices to show that \(\sup_{\theta\in R_{n, 0}} ||A_n|| + ||B_n|| = o_p(1)\). By Lemma \ref{lem: stat_helper} and Lemma \ref{lem: norm_exp} we see that \(\sqrt{\log n}S_n\) converges uniformly to 0 over \(R_{n, 0}\). But then, since \(\Sigma\) is uniformly bounded from below over \(\Theta\), for a fixed \(\epsilon>0\), we can find \(n_0\in\mathbb{N}\) large enough so that
    \[
    \sup_{\theta\in R_{n, 0}}\mathbb{P}\left(\left|\left|\sqrt{\log n}A_n\right|\right| \ge \left|\left|H^{-\frac{1}{2}}\left(\sum_{t=0}^{\lfloor ns\rfloor-2}\Gamma^t \Sigma\left(\Gamma^t\right)^T\right)H^{-\frac{1}{2}}\right|\right|\right) < \epsilon
    \]
    for all \(n\ge n_0\). It then follows from Lemma \ref{lem: norm_exp} that \(\sup_{\theta\in R_{n, 0}}||A_n|| = o_p(1)\). Next, we see that \(\mathbb{E}B_n = \Gamma^{\lfloor ns\rfloor-1}\tilde{H}(\Gamma^{\lfloor ns\rfloor-1})^T\) where \(\tilde{H} = \mathbb{E}(\tilde{S}_{XX})\) and therefore \(\sup_{\theta\in R_{n, 0}}\left|\left|\mathbb{E}B_n\right|\right| \le C \left(1 - \frac{\log n}{n}\right)^{2(\lfloor ns\rfloor-1)}||\tilde{H}||\) by part \ref{lem: norm_asym_d} of Lemma \ref{lem: norm_asym}. Finally, from the inequality \(1 - \frac{n}{\log n} \le \frac{1}{n^{1/n}}\) and part \ref{lem: norm_asym_c} of Lemma \ref{lem: norm_asym} we see that \(\sup_{\theta\in R_{n, 0}}||\mathbb{E}B_n|| = o(1)\). Since \(B_n\) is positive semidefinite, this implies that \(B_n\) converges in probability to 0 uniformly over \(R_{n, 0}\) and therefore concludes the proof of \eqref{eq: stat_asym_xx}.

    For the proof of \eqref{eq: stat_asym_xe}, let \((\theta_n)_{n\in\mathbb{N}}\subset \Theta\) be such that \(\theta_n\in R_{n, 0}\) and define the array \((e_{t, n})_{t\ge 1, n\in\mathbb{N}}\) by 
    \[
    e_{t,n} = \textnormal{vec}\left(n^{-\frac{1}{2}}H^{-\frac{1}{2}}X_{t-1, \theta_n}\epsilon_{t, \theta_n}^T\Sigma_n^{-\frac{1}{2}}\right) = n^{-\frac{1}{2}}\Sigma^{-\frac{1}{2}}\epsilon_{t, \theta_n}\otimes H^{-\frac{1}{2}}X_{t-1, \theta_n}.
    \]
    Proving \eqref{eq: stat_asym_xe} is equivalent to proving \(\sum_{t=1}^n e_{t, n}\rightarrow_w \mathcal{N}(0, I)\). Since \(X_{t-1, \theta_n}\) is measurable wrt. \(\mathcal{F}_{t-1}\), we see that \(e_{t, n}\) is a martingale difference array and, by \eqref{eq: stat_asym_xx}, 
    \[
    \sum_{t=1}^n\mathbb{E}(e_{t, n}e_{t, n}^T | \mathcal{F}_{t-1}) \rightarrow_p I
    \] 
    for \(n\rightarrow \infty\). Our aim is to apply the martingale difference array CLT given in Theorem \ref{thm: mda_clt} which amounts to checking that, for each \(\gamma > 0\), \(\sum_{t=1}^n \mathbb{E}\left(\left|\left|e_{t, n}\right|\right|^2 \mathbf{1}\left(\left|\left| e_{t, n}\right|\right| > \gamma\right)|\mathcal{F}_{t-1}\right) = o_p(1).\) Now, fix some \(\gamma >0\) and note that \(||e_{t, n}||^2 = ||H^{-\frac{1}{2}}X_{t-1, \theta_n}||^2 ||\epsilon_{t, \theta_n}||^2\) so that
    \begin{align*}
        &\sum_{t=1}^n \mathbb{E}\left(\left|\left|e_{t, n}\right|\right|^2 \mathbf{1}\left(\left|\left| e_{t, n}\right|\right| > \gamma\right)|\mathcal{F}_{t-1}\right) \\
            \le &\frac{1}{n}\sum_{t=1}^n \left|\left|H^{-\frac{1}{2}}X_{t-1, {\theta_n}}\right|\right|^2\mathbb{E}\left(\left|\left|\Sigma^{-\frac{1}{2}}\epsilon_{t, n}\right|\right|^2 \mathbf{1}\left(\left|\left| e_{t, n}\right|\right| > \gamma\right)|\mathcal{F}_{t-1}\right) \\
            \le & C_n\max_{1\le t\le n} \left\{\mathbb{E}\left(\left|\left|\epsilon_{t, n}\right|\right|^2 \mathbf{1}\left(\left|\left| e_{t, n}\right|\right| > \gamma\right)|\mathcal{F}_{t-1}\right)\right\}
    \end{align*}
    where \(C_n = \sup_{\theta\in R_{n, 0}}\tr\left(H^{-\frac{1}{2}}S_{XX}H^{-\frac{1}{2}}\right)\tr\left(\Sigma\right) = O_p(1)\) because of \eqref{eq: stat_asym_xx}. An application of Hölder's inequality and the Markov inequality gives us
    \begin{align*}
        &\mathbb{E}\left(\left|\left|\epsilon_{t, n}\right|\right|^2 \mathbf{1}\left(\left|\left| e_{t, n}\right|\right| > \gamma\right)|\mathcal{F}_{t-1}\right) \\
            \le &\mathbb{E}\left(\left|\left|\epsilon_{t, \theta_n}\right|\right|^{2 + \delta} | \mathcal{F}_{t-1, n}\right)^{\frac{2}{2+\delta}}\mathbb{P}\left(\left|\left|e_{t, n}\right|\right| > \gamma | \mathcal{F}_{t-1, n}\right)^{\frac{\delta}{2+ \delta}} \\
            \le & \mathbb{E}\left(\left|\left|\epsilon_{t, \theta_n}\right|\right|^{2 + \delta} | \mathcal{F}_{t-1, n}\right)^{\frac{2}{2+\delta}}\left(\frac{\mathbb{E}(||H^{-\frac{1}{2}}X_{t-1, \theta_n}||^2||\Sigma^{-\frac{1}{2}}\epsilon_{t, \theta_n}||^2|\mathcal{F}_{t-1})}{n\gamma^2}\right)^{\frac{\delta}{2+\delta}} \\
            \le &C \tr \left(\frac{1}{n}H^{-\frac{1}{2}}X_{t-1, \theta_n}X_{t-1, \theta_n}^T H^{-\frac{1}{2}}\right)^{\frac{\delta}{2+\delta}}
    \end{align*}
    where \(C = d^{\frac{\delta}{2+\delta}}\sup_{\theta\in\Theta}\mathbb{E}\left(\left|\left|\epsilon_{t, \theta}\right|\right|^{2 + \delta} | \mathcal{F}_{t-1, n}\right)^{\frac{2}{2+\delta}} < \infty\) because of Assumptions \ref{ass: cov} and \ref{ass: mom}. Thus, the proof is complete if we can show that
    \[
    \sup_{\theta\in R_{n, 0}}\max_{1\le t \le n}\tr \left(\frac{1}{n}H^{-\frac{1}{2}}X_{t-1, \theta_n}X_{t-1, \theta_n}^T H^{-\frac{1}{2}}\right)^{\frac{\delta}{2+\delta}} = o_p(1).
    \]
    This follows from the same argument as in the proof of equation (5) \citep{phillips2007limit}. (The multivariate case is essentially the same once \eqref{eq: stat_asym_xx} is established.)
\end{proof}

\begin{proof}[Proof of Lemma \ref{lem: jc_stat}]
    Define the sequence \((c_n)_{n\in\mathbb{N}}\subset \mathbb{R}^d\) given by \(c_{n, i}=e^{(C_n)_{ii}/n}\) for \(1 \le i \le d\). By assumption, we have \(c_n\rightarrow 0\) for \(n\rightarrow \infty\) so we can assume without loss of generality that \(\max_i c_{n, i} \le 1\) and \(\min_i c_{n, i} > 0\) and, by potentially passing to a sub sequence, that \(c_n\) is monotonically decreasing. 
    
    For each \(n\), we can then find \(k_n\in\mathbb{N}\) such that
    \[
    1 - k_n^{-\eta} \le \min_i e^{C_{n, ii}/k_n} \le \max_i e^{C_{n, ii}/k_n} \le 1 - \frac{\log k_n}{k_n}.
    \]
    Passing to another sub sequence if necessary, we may assume that \(k_n\) is strictly increasing. Now, define sequences \((\lambda_k)_{k\in\mathbb{N}}\subset \mathbb{C}^d\) and \((\Sigma_k)_{k\in\mathbb{N}}\subset \mathbb{R}^{d\times d}\) such that \(|\lambda_{k, i}| = 1-\log(k)/k\) and \(\Sigma_k = I\) for \(k < k_0\) and \(|\lambda_{k, i}| = e^{C_{n, ii}/k_n}\) and \(\Sigma_k = \Omega_n\) for \(k_n \le k < k_{n+1}\). We then have
    \[
    1-k^{-\eta}\le 1 - k_n^{-\eta} \le \min_i |\lambda_{n, i}|\le \max_i |\lambda_{n, i}| \le 1 - \frac{\log k_n}{k_n} \le 1 - \frac{\log k}{k}.
    \]
    Define \((\theta_k)_{k\in\mathbb{N}}\subset \Theta\) by \(\theta_k = (\Gamma_k, \Sigma_k, c)\) where \(\Gamma_k\) is the diagonal matrix whose diagonal entries are given by \(\lambda_k\) and \(c\in \mathbb{R}_+\). The above inequalities together with the assumptions on \(\Omega_n\) imply that \(\theta_k\in R_{k, 1}\cap R_{k, d}\) for all \(k\). Define
    \[
    S_k = \frac{1}{k}\sum_{t=1}^k X_{t-1, \theta_k}X_{t-1, \theta_k}^T, \quad T_k = \frac{1}{k}\sum_{t=1}^k X_{t-1, \theta_k}\epsilon_{t, \theta_k}, \quad \text{and} \quad H_k = \mathbb{E}\left(S_k\right).
    \]
    The result follows by the triangle inequality, equations \eqref{eq: l2_yr}, \eqref{eq: l2_yy} and Theorem \ref{thm: stat_asym}.
\end{proof}

\subsection{Mixed Asymptotics} \label{app: mix}

This section is devoted to the proof of Lemma \ref{lem: mixed_app}. Assume throughout Assumptions \ref{ass: M} and \ref{ass: U} and consider the special case where \(F_\theta = I\). For ease of notation we define
\[
A = \int_0^1 J_{C, t}J_{C, t}^T dt, \quad B = \int_0^1 J_{C, t}dW_t^T.
\]
We hold \(1\le k\le d-1\) fixed throughout and start by partitioning \(R_{n, k}\). Let \(r = d-k \ge 1\) and define \(w(j, l) = (1-|\lambda_{i_{k+j}}|)/(1-|\lambda_{i_{k+l}}|)\) for \(1\le j < l \le r\). We introduce the sets
\[
U_{n, 0} = \left\{\theta\in R_{n, k} : w(0, 1) \le n^{-\frac{\gamma}{r}}\right\}, \quad U_{n, r} = \left\{\theta\in R_{n, k}: w(0, r) \ge n^{-\gamma}\right\},
\]
\[
U_{n, j} = \left\{\theta\in R_{n, k}: w(0, j) \ge n^{-\frac{j\gamma}{r}}, w(j, j+1)\le n^{-\frac{\gamma}{r}}\right\}
\]
for \(j=1,..., r-1\). We have \(R_{n, k}=\bigcup_j U_{n, j}\) for all \(n\in \mathbb{N}\). Indeed, fix \(n\) and take some \(\theta\in R_{n, k}\) and define \(j_0 = \min\left(\inf\left\{0\le j\le r-1 : w(j, j+1) \le n^{-\frac{\gamma}{r}}\right\}, r\right)\), where we use the convention \(\inf\emptyset = \infty\). If \(j_0 = 0 \), then clearly \(\theta\in U_{n, 0} = U_{n, j_0}\). Otherwise, we find that
\[
w(0, j_0) = \prod_{j=0}^{j_0-1}w(j, j+1) \ge \prod_{j=0}^{j_0 - 1}n^{-\frac{\gamma}{r}} = n^{-\frac{j_0\gamma}{r}}
\]
so that, again, \(\theta\in U_{n, j_0}\). Fix some \(0\le j \le r\). It therefore suffices to show that \eqref{eq: mixed_app_xx} and \eqref{eq: mixed_app_xe} hold uniformly over \(U_{n, j}\). To do so, we need to split the covariance matrices and the normalizing matrix into four blocks. In particular, we write
\[
H = \begin{pmatrix} H_{11} & H_{12} \\ H_{21} & H_{22}\end{pmatrix}
\]
where \(H_{11}\) is \((k+j)\times (k+j)\) and the other blocks of conforming dimensions. Analogously, \(S_{XX}\), \(S_{X\epsilon}\), \(G\), \(A\) and \(B\) can be written as block matrices. Block coordinates are written in the subscript when possible and otherwise in the superscript. For example, \(S_{XX}^{12}\) and \(A_{12}\) are the top right \((k+j)\times(d-k-j)\) blocks of \(S_{XX}\) and \(A\).

\begin{lemma} \label{lem: mix_block_limits}
For fixed \(1\le k \le d-1\) and \(0\le j\le r\), let \(N\in\mathbb{R}^{(d-k-j)\times d}\) be a random matrix on \((\Omega, \mathcal{F}, \mathbb{P})\) such that \(\textnormal{vec}(N)\sim\mathcal{N}(0, I)\). We have the following block-wise limits
\begin{align}
    &\lim_{n\rightarrow \infty}\sup_{\theta\in U_{n, j}}d_{BL}\left(H_{11}^{-\frac{1}{2}}S_{XX}^{12}H_{22}^{-\frac{1}{2}}, 0\right) = 0 \label{eq: mix_off_diag}\\
    &\lim_{n\rightarrow \infty}\sup_{\theta\in U_{n, j}}d_{BL}\left(H_{11}^{-\frac{1}{2}}S_{XX}^{11}H_{11}^{-\frac{1}{2}}, G_{11}^{-\frac{1}{2}}A_{11}G_{11}^{-\frac{1}{2}}\right) = 0 \label{eq: mix_diag_1}\\
    &\lim_{n\rightarrow \infty}\sup_{\theta\in U_{n, j}}d_{BL}\left(H_{22}^{-\frac{1}{2}}S_{XX}^{22}H_{22}^{-\frac{1}{2}}, I\right) = 0. \label{eq: mix_diag_2} \\
    &\lim_{n\rightarrow \infty} \sup_{\theta\in U_{n, j}} d_{BL}\left(\sqrt{n}H^{-\frac{1}{2}}_{11}\left(S_{X\epsilon}^{11}, S_{X\epsilon}^{12}\right), G_{11}^{-\frac{1}{2}}\left(B_{11}, B_{12}\right)\right) = 0 \label{eq: mix_xe_1}\\
    &\lim_{n\rightarrow \infty} \sup_{\theta\in U_{n, j}} d_{BL}\left(\sqrt{n}H^{-\frac{1}{2}}_{22}\left(S_{X\epsilon}^{21}, S_{X\epsilon}^{22}\right), N\right) = 0 \label{eq: mix_xe_2}
\end{align}
\end{lemma}

\begin{proof}
    Fix some \(\theta \in U_{n, j}\). For any \(i \le i_{k+j}\), we have \(|\lambda_i| \ge |\lambda_{i_{k+j}}| \ge 1 - n^{-\eta - \gamma}n^{\frac{j\gamma}{r}}\ge 1 - n^{-\eta}\) and, for any \(i \ge i_{k+j}\), \(|\lambda_i| \le |\lambda_{i_k}| \le 1 - n^{-\eta - \gamma} \le 1 - \frac{\log n}{n}.\) Equations \eqref{eq: mix_diag_1}, \eqref{eq: mix_diag_2}, \eqref{eq: mix_xe_1} and \eqref{eq: mix_xe_2} then follow from the proofs in Sections \ref{sec: ltu} and \ref{sec: stat}.

    For the proof of \eqref{eq: mix_off_diag}, note that \(S_{XX}^{12} = \Gamma_{11}^n S_{XX}^{12}\Gamma_{22}^n + \sum_{t=0}^{n-1}\Gamma_{11}^t S_n \left(\Gamma_{22}^t\right)^T\), where 
    \[
    S_n = S^{12}_{\epsilon\epsilon} + \frac{1}{n}\left(X_{n, \theta}X_{n, \theta}^T\right)_{12} - \Gamma_{11}S_{X\epsilon}^{12} - \left(S_{X\epsilon}^{21}\right)^T\Gamma_{22}^T
    \]
    and \(S_{\epsilon\epsilon} = (\sum_{t=1}^n \epsilon_{t, \theta}\epsilon_{t, \theta}^T)/n\). An application of Hölder's inequality yields
    \begin{multline*}
        \left|\left|H_{11}^{-\frac{1}{2}}\Gamma_{11}^nS_{XX}^{12}\left(\Gamma_{22}^n\right)^T H_{22}^{-1}\right|\right| \le \tr\left(H_{11}^{-\frac{1}{2}}\Gamma_{11}^n S_{XX}^{11}\left(\Gamma_{11}^n\right)^TH_{11}^{-\frac{1}{2}}\right)^{\frac{1}{2}} \times \\
        \tr\left(H_{22}^{-\frac{1}{2}}\Gamma_{22}^n S_{XX}^{22}\left(\Gamma_{22}^n\right)^TH_{22}^{-\frac{1}{2}}\right)^{\frac{1}{2}}
    \end{multline*}
    and it follows from \eqref{eq: mix_diag_1} and \eqref{eq: mix_diag_2} along with the fact that \(\sup_{\theta\in U_{n, j}}||\Gamma_{22}^n||=o(1)\) that
    \[
    \sup_{\theta\in U_{n, j}} \tr\left(H_{11}^{-\frac{1}{2}}\Gamma_{11}^n S_{XX}^{11}\left(\Gamma_{11}^n\right)^TH_{11}^{-\frac{1}{2}}\right) = O_p(1),
    \]
    \[
    \sup_{\theta\in U_{n, j}} \tr\left(H_{22}^{-\frac{1}{2}}\Gamma_{22}^n S_{XX}^{22}\left(\Gamma_{22}^n\right)^TH_{22}^{-\frac{1}{2}}\right) = o_p(1)
    \]
    so that \(\sup_{\theta\in U_{n, j}}||H_{11}^{-\frac{1}{2}} \Gamma_{11}^n S_{XX}^{12} \left(\Gamma_{22}^n\right)^T H_{22}^{-1}|| = o_p(1)\). For the second term, we have, for all \(\theta \in U_{n, j}\),
    \[
    \left|\left|H_{11}^{-\frac{1}{2}}\sum_{t=1}^{n-1}\Gamma_{11}^t S_n \left(\Gamma_{22}^t\right)^T H_{22}^{-\frac{1}{2}}\right|\right| \le C_n \left|\left|H_{22}^{-\frac{1}{2}}\right|\right|\sum_{t=0}^{n-1}\left|\left|H_{11}^{-\frac{1}{2}}\Gamma_{11}^t\Sigma_{11}^{-\frac{1}{2}}\right|\right| 
    \]
    where \(C_n = \sup_{\theta\in U_{n, j}}\sup_{t\ge 1}||S_n||||\Sigma_{11}^{\frac{1}{2}}||||\Gamma_{22}^t||\). Hölder's inequality yields
    \[
    \sum_{t=0}^{n-1}\left|\left|H_{11}^{-\frac{1}{2}}\Gamma_{11}^t\Sigma_{11}^{-\frac{1}{2}}\right|\right| \le \tr\left(H_{11}^{-\frac{1}{2}}\sum_{t=0}^{n-1}\Gamma_{11}^t\Sigma_{11}\left(\Gamma_{11}^t\right)^T H_{11}^{-\frac{1}{2}}\right)
    \]
    so that, by part \ref{lem: norm_asym_b} of Lemma \ref{lem: norm_asym} and Lemma \ref{lem: norm_exp},
    \[
    \sup_{\theta\in U_{n, j}} \left|\left|H_{22}^{-\frac{1}{2}}\right|\right|\sum_{t=0}^{n-1}\left|\left|H_{11}^{-\frac{1}{2}}\Gamma_{11}^t\Sigma_{11}^{-\frac{1}{2}}\right|\right| = o_p(1).
    \]
    Since \(||\Sigma_{11}^{\frac{1}{2}}||\) and \(\sup_{t\ge 1}||\Gamma_{22}^t||\) are uniformly bouded over \(U_{n, j}\), it therefore suffices to show that \(\sup_{\theta\in U_{n, j}}||S_n|| = O_p(1)\). From Lemma \ref{lem: cov_wlln} and part (b) of Lemma \ref{lem: gaus_rep} in the Appendix we have \(\sup_{\theta\in U_{n, j}}\left|\left| S_{\epsilon\epsilon} + \frac{1}{n}X_{n, \theta}X_{n, \theta}\right|\right| = O_p(1)\) and it follows from \eqref{eq: mix_xe_1} and \eqref{eq: mix_xe_2} along with the fact that \(H_{11}, H_{22} = O(n)\) uniformly over \(\Theta\) that \(\sup_{\theta\in \Theta} \left|\left|S_{X\epsilon}\right|\right| = O_p(1)\) which completes the proof.
\end{proof}

We now define \(\tilde{H}\) as the block diagonal matrix obtained by deleting the off-diagonal blocks of \(H\). Lemma \ref{lem: mix_block_limits} determines the limiting behaviour of \(\tilde{H}^{-\frac{1}{2}}S_{XX}\tilde{H}^{-\frac{1}{2}}\). The next lemma explains why this is sufficient.

\begin{lemma} \label{lem: mixed_norm_app}
    For fixed \(1\le k \le d-1\) and \(0\le j\le r\), let \(\tilde{H}\) and \(\tilde{G}\) be the block-diagonal matrices obtained by deleting the off-diagonal blocks of \(H\) and \(G\), respectively. We then have
    \[
    \sup_{\theta\in U_{n, j}}\left\{\left|\left|H^{-\frac{1}{2}}\tilde{H}^{\frac{1}{2}} - I\right|\right| + \left|\left|G^{-\frac{1}{2}}\tilde{G}^{\frac{1}{2}} - I\right|\right|\right\} = o(1).
    \]
\end{lemma}
    
\begin{proof}
    It suffices to show that \(\sup_{\theta\in U_{n, j}}||H_{11}^{-\frac{1}{2}}H_{12}H_{22}^{-\frac{1}{2}}|| = o(1)\). To do so, we first note that, arguing as in the proof of part \ref{lem: norm_asym_b} of Lemma \ref{lem: norm_asym}, \(\sup_{\theta\in U_{n, j}}\sigma_{min}(H_{11}^{-1}) = O(1 - |\lambda_{i_{k+j}}|)\) and, consequently,
    \[
    \sup_{\theta \in U_{n, j}}\left|\left|H_{11}^{-\frac{1}{2}}\left(1 - |\lambda_{i_{k+j}}|\right)^{-\frac{1}{2}}\right|\right| = O(1).
    \]
    Let \(\Lambda\in \mathbb{R}^{(d - k - j)\times(d - k - j)}\) be the diagonal matrix satisfying \(\Lambda_{ll} = 1 - |\lambda_{i_{k+j+l}}|\). Then, by part \ref{lem: norm_asym_c} of Lemma \ref{lem: norm_asym}, we have
    \[
    \sup_{\theta\in U_{n, j}}\left|\left|H_{12}\Lambda^{\frac{1}{2}}\right|\right| = O\left(\left(1 - |\lambda_{i_{k+j+1}}|\right)^{-\frac{1}{2}}\right).
    \]
    Because of the Jordan-like nature of $\Gamma$ 
    and, for any \(\theta\in U_{n, j}\),
    \begin{align*}
        \sigma_{min}\left(\Lambda^{\frac{1}{2}}H_{22}\Lambda^{\frac{1}{2}}\right) 
            &\ge \frac{\sigma_{min}(\Sigma_{22})}{n}\sum_{t=1}^{n-1}\sum_{s=0}^{t-1}\sigma_{min}\left(\Lambda^{\frac{1}{2}}\Gamma^s\right)^2 \\
            &\ge \sigma_{min}\left(\Sigma_{22}\right) \min_{1\le l \le d-k-j}\frac{\Lambda_{ll}}{n}\sum_{t=1}^{n-1}\sum_{s=0}^{t-1}\min\left\{\left|\lambda_{i_{k+j+l}}\right|^{2s}, 1\right\}
    \end{align*}
    where the second inequality follows from Assumption \ref{ass: jord} and the fact that \(\Gamma^0=I\). For any \(1\le l \le d-k-j\), it holds that
    \begin{align*}
        \frac{\Lambda_{ll}}{n}\sum_{t=1}^{n-1}\sum_{s=0}^{t-1}\left|\lambda_{i_{k+j+l}}\right|^{2s}
            &= \frac{1-|\lambda_{i_{k+j+l}}|}{n}\sum_{t=1}^{n-1}\sum_{s=0}^{t-1}\left|\lambda_{i_{k+j+l}}\right|^{2s} \\
            &= \frac{1}{1+|\lambda_{i_{k+j+l}}|}\left(1 - \frac{1 - |\lambda_{i_{k+j+l}}|^{2n}}{n(1-|\lambda_{i_{k+j+l}}|^2)}\right).
    \end{align*}
    Now, since \(\sup_{\theta\in U_{n, j}}|\lambda_{i_{k+j+l}}|^{2n} \rightarrow 0\) and \(\inf_{\theta\in U_{n, j}}n(1-|\lambda_{i_{k+j+l}}|^2)\rightarrow \infty\) for \(n\rightarrow \infty\) and \(1\le l \le d-k-j\), we get that
    \[
    \sup_{\theta\in U_{n, j}} \sigma_{max}\left(\Lambda^{-\frac{1}{2}}H_{22}^{-1}\Lambda^{-\frac{1}{2}}\right) = \sup_{\theta\in U_{n, j}}\sigma_{min}\left(\Lambda^{\frac{1}{2}}H_{22}\Lambda^{\frac{1}{2}}\right)^{-1} = O(1)
    \]
    from which it follows that \(\sup_{\theta\in U_{n, j}}||\Lambda^{-\frac{1}{2}}H_{22}^{-\frac{1}{2}}|| = O(1)\). Combining all these rates yields
    \[
    \sup_{\theta\in U_{n, j}}\left|\left|H_{11}^{-\frac{1}{2}}H_{12}H_{22}^{-\frac{1}{2}}\right|\right| = O\left(\frac{1 - |\lambda_{i_{k+j}}|}{1 - |\lambda_{i_{k+j+1}}|}\right)^{\frac{1}{2}} = O\left(n^{-\frac{\gamma}{2r}}\right).
    \]
\end{proof}

With these two lemmas we can complete the proof of \eqref{eq: mixed_app_xx} and \eqref{eq: mixed_app_xe}. First, for any \(\theta\in U_{n, j}\), we have
\[
\left|\left|H^{-\frac{1}{2}}S_{XX}H^{-\frac{1}{2}} - \tilde{H}^{-\frac{1}{2}}S_{XX}\tilde{H}^{-\frac{1}{2}}\right|\right| \le C_n \left|\left|\tilde{H}^{-\frac{1}{2}}S_{XX}\tilde{H}^{-\frac{1}{2}}\right|\right|
\]
where, by Lemma \ref{lem: mixed_norm_app}, \(C_n = \sup_{\theta\in U_{n, j}} \left|\left|\tilde{H}^{\frac{1}{2}}H^{-\frac{1}{2}} - I\right|\right|\left(\left|\left|\tilde{H}^{\frac{1}{2}}H^{-\frac{1}{2}}\right|\right| + \sqrt{d}\right) = o(1).\) It then follows from Lemma \ref{lem: mix_block_limits} that
\[
\sup_{\theta\in U_{n, j}}\left|\left|H^{-\frac{1}{2}}S_{XX}H^{-\frac{1}{2}} - \tilde{H}^{-\frac{1}{2}}S_{XX}\tilde{H}^{-\frac{1}{2}}\right|\right| = o_p(1)
\]
and, similarly,
\begin{align*}
    &\sup_{\theta\in U_{n, j}}\left|\left|G^{-\frac{1}{2}}AG^{-\frac{1}{2}} - \tilde{G}^{-\frac{1}{2}}A\tilde{G}^{-\frac{1}{2}}\right|\right| = o_p(1), \\
    &\sup_{\theta\in U_{n, j}}\left|\left|\sqrt{n}H^{-\frac{1}{2}}S_{X\epsilon} - \sqrt{n}\tilde{H}^{-\frac{1}{2}}S_{X\epsilon}\right|\right| = o_p(1), \\
    &\sup_{\theta\in U_{n, j}}\left|\left|G^{-\frac{1}{2}}B - \tilde{G}^{-\frac{1}{2}}B\right|\right| = o_p(1). 
\end{align*}
Finally, arguing as in the proof of Lemma \ref{lem: jc_stat}, we find 
\begin{align*}
    &\lim_{n\rightarrow \infty}\sup_{\theta\in U_{n, j}}d_{BL}\left(G_{11}^{-\frac{1}{2}}A_{12}G_{22}^{-\frac{1}{2}}, 0\right) = 0 \\
    &\lim_{n\rightarrow \infty}\sup_{\theta\in U_{n, j}}d_{BL}\left(G_{22}^{-\frac{1}{2}}A_{22}G_{22}^{-\frac{1}{2}}, I\right) = 0 \\
    &\lim_{n\rightarrow \infty}\sup_{\theta\in U_{n, j}}d_{BL}\left(G_{11}^{-\frac{1}{2}}\left(B_{11}, B_{12}\right), N\right) = 0
\end{align*}
so that \eqref{eq: mixed_app_xx} and \eqref{eq: mixed_app_xe} follow from Lemma \ref{lem: mix_block_limits}.

\section{Confidence Regions} \label{app: cr}

This section captures some of the more technical details omitted from Section \ref{sec: inf}. Validity of \(CR_a(\alpha)\) and \(CR_b(\alpha)\) is a fairly straightforward consequence of the fact that \(\hat{t}^2_\Gamma\) can be uniformly approximated by \(t^2_\Gamma\) and \(\tilde{t}^2_\Gamma\) both of which have continuous distributions.

\begin{proof}[Proof of Theorem \ref{thm: valid_cr}]
    The result follows from Proposition 13 in the supplementary material for \cite{lundborg2021conditional} since \(\hat{t}^2_\Gamma\) and \(\tilde{t}^2_\Gamma\) both converge in distriubtion to \(t^2_\Gamma\) uniformly over \(\Theta\) and the latter is uniformly absolutely continuous wrt. Lebesgue measure.
\end{proof}

\subsection{Predictive regression}

\begin{proof} [Proof of Lemma \ref{lem: ci_pr}]
    For each \(\theta\in \Theta_P\) with \(\gamma\) and \(\tilde{\Gamma}\) the corresponding autoregressive coefficients, define the events
    \[
    A_\theta = \left\{\omega \in \Omega : \tilde{\Gamma}\notin CR(\alpha_1, \omega)\right\}, \quad B_\theta = \left\{\omega\in \Omega : \gamma \notin C_{\gamma|\tilde{\Gamma}}(\alpha_2, \omega)\right\}
    \]
    where the dependence of the confidence regions on \(\omega\) is made explicit in the notation. If \(\omega\in \Omega\) is such that \(\gamma\notin CI_\gamma(\alpha_1, \alpha_2, \omega)\), then we must either have \(\omega\in A_\theta\) or \(\omega\in B_\theta\) implying, by Bonferroni's inequality, \(\mathbb{P}\left(\gamma\notin CI_\gamma(\alpha_1, \alpha_2)\right) \le \mathbb{P}(A_\theta) + \mathbb{P}(B_\theta).\) It follows by assumption that \(\limsup_{n\rightarrow\infty}\sup_{\theta\in \Theta_P}\mathbb{P}(A_\theta) \le \alpha_1\) so that
    \[
    \liminf_{n\rightarrow\infty}\inf_{\theta\in \Theta_P}\mathbb{P}\left(\gamma\in CI_\gamma(\alpha_1, \alpha_2)\right) \ge 1 - \alpha_1 - \limsup_{n\rightarrow\infty}\sup_{\theta\in\Theta_P} \mathbb{P}(B_\theta).
    \]
    The proof is complete if we can show that \(\hat{\sigma}^{-2}_Y \hat{t}^2_{\gamma|\tilde{\Gamma}}\rightarrow_w \chi^2_{d-1}\) uniformly over \(\Theta_P\) since this would imply that \(\limsup_{n\rightarrow\infty}\sup_{\theta\in\Theta_P} \mathbb{P}(B_\theta) \le \alpha_2\) by the same argument as in the proof of Theorem \ref{thm: valid_cr}. Defining \(\tilde{\rho}_t = \rho_t - \Sigma_{YX}\Sigma_{X}^{-1}\tilde{\epsilon_t}\), we have in obvious notation \(\hat{t}^2_{\gamma|\tilde{\Gamma}} = S_{\tilde{\rho}\tilde{X}}S_{\tilde{X}\tilde{X}}^{-1}S_{\tilde{X}\tilde{\rho}} + R_n\), where
    \[
    R_n = n\left(r_nS_{\tilde{\epsilon}\tilde{X}}S_{\tilde{X}\tilde{X}}S_{\tilde{X}\tilde{\epsilon}}r_n^T + r_nS_{\tilde{\epsilon}\tilde{X}}S_{\tilde{X}\tilde{X}}S_{\tilde{X}\tilde{\rho}} + S_{\tilde{\rho}\tilde{X}}S_{\tilde{X}\tilde{X}}S_{\tilde{X}\tilde{\epsilon}}r_n^T\right)
    \]
    with \(r_n = \hat{\Sigma}_{YX}\hat{\Sigma}_X^{-1} - \Sigma_{YX}\Sigma_X^{-1}\). Since \(\hat{\Sigma}\) is uniformly consistent and \(\Sigma\) is uniformly invertible over \(\Theta_P\) we see that \(r_n\rightarrow_p 0\) uniformly over \(\Theta_P\). This follows from the uniform versions of the continuous mapping theorem and Slutsky's Lemma (see, e.g., Proposition 9 and Proposition 15 in the supplementary material for \cite{lundborg2021conditional}). Furthermore, all the matrix products in the expression above converge uniformly in distribution over \(\Theta_P\) by Theorem \ref{thm: unif_app}. Thus, \(R_n\rightarrow_p 0\) uniformly over \(\Theta_P\). Now, let \(\tilde{\Sigma}\in \mathbb{R}^{d\times d}\) be given by \(\tilde{\Sigma}_{11} = \Sigma_Y - \delta^T\Sigma_X\delta\), \((\tilde{\Sigma}_{1i})_{2\le i\le d}=(\tilde{\Sigma}_{i1})_{2\le i\le d} = 0\), and \((\tilde{\Sigma}_{ij})_{2\le i, j\le d} = \Sigma_X\). Similar to \eqref{eq: pred_asym}, we then find
    \[
    S_{\tilde{\rho}\tilde{X}} S_{\tilde{X}\tilde{X}}^{-1} S_{\tilde{X}\tilde{\rho}} \rightarrow_w\left|\left|\left(\tilde{\Sigma}_Y - \tilde{\delta}^T\tilde{\Sigma}_X\tilde{\delta}\right)^{\frac{1}{2}}Z\right|\right|^2
    \]
    uniformly over \(\Theta_P\) with \(Z\) a \((d-1)\)-dimensional standard normal random variable. But then, since \(\tilde{\Sigma}_Y - \tilde{\delta}^T\tilde{\Sigma}_X\tilde{\delta} = \Sigma_Y - \delta^T\Sigma_X\delta\) which, in turn, is uniformly estimated by \(\hat{\sigma}^2_Y\), this completes the proof.
\end{proof}

\section{Martingale Limit Theorems}

We start by stating a strong invariance principle for stationary martingale difference arrays due to \cite{cuny2021rates}. A martingale difference array is a doubly infinite array, \((e_{t, n})_{t,n \in\mathbb{N}}\), along with an array of filtrations, \((\mathcal{F}_{t, n})_{t,n \in\mathbb{N}}\), such that, for each \(n\), \(e_{t, n}\) is a martingale difference sequence wrt. \(\mathcal{F}_{t, n}\).

\begin{theorem} \label{thm: asip}
    Let \((e_{t, n})_{t, n\in\mathbb{N}}\) be a stationary \(\mathbb{R}^d\)-valued martingale difference array wrt. \((\mathcal{F}_{t, n})_{t, n\in\mathbb{N}}\) such that \(\mathbb{E}(e_{t, n}e_{t, n}^T|\mathcal{F}_{t-1, n})=\mathbb{E}e_{0, n}e_{0, n}^T = I\) a.s. for all \(t\ge 1, n \in\mathbb{N}\) and assume there exists some small \(\delta > 0 \) such that \(\sup_{t,n\in\mathbb{N}} \mathbb{E}||e_{t,n}||^{2+\delta} < \infty\).
    Then, after possibly enlarging \((\Omega, \mathcal{F}, \mathbb{P})\), there exist a triangular array of random variables, \((\rho_{t, n})_{t\ge 1, n\in\mathbb{N}}\), where each row, \((\rho_{t, n})_{t\in\mathbb{N}}\), is i.i.d. standard normal and such that
    \[
    \mathbb{E}\left(\left|\left|\sup_{1 \le k \le n}||\sum_{t = 1}^k e_{t, n} - \sum_{t = 1}^k \rho_{t, n}||\right|\right|\right)=O\left(n^{\frac{1}{2+\delta}}\left(\log n\right)^{\frac{1+\delta}{2(2+\delta)}}\right).
    \]
\end{theorem}

\begin{proof}
    The proof is essentially the same as the proof of Theorem 2.1 in \cite{cuny2021rates} and we will not go into much detail here, but only mention the key steps and how they generalize to the martingale difference array setting.

    First, we note that under above assumptions, Lemma 4.1 in \cite{cuny2021rates} holds for triangular arrays as well. Indeed, the constant \(C\) depends only on \(\delta\), \(d\) and \(\mathbb{E}||e_{1, n}||^{2+\delta}\) and the latter is uniformly bounded over \(n\).

    After possibly enlarging the initial probability space, we can assume that it is large enough to contain a doubly infinite array \((u_{t, n})_{t, n\in\mathbb{N}}\) and a sequence \((\rho_{1,n})_{n\in\mathbb{N}}\) such that, for each fixed \(n\), \(u_{t, n}\) is i.i.d. uniform on \([0, 1]\) and independent of \(e_{t, n}\) and \(\rho_{1, n}\) is \(d\)-dimensional standard normal independent of \(e_{t, n}\) and \(u_{t, n}\). Now for each \(n\), we can follow the steps in the proof of Theorem 2.1 in \cite{cuny2021rates} and construct a sequence, \((\rho_{t, n})_{t\ge 1}\), of i.i.d. \(d\)-dimensional standard normal random variables satisfying certain inequalities. Now define, for \(L\in\mathbb{N}\),
    \[
    D_{L, n} = \sup_{l\le 2^L}\left|\left|\sum_{i=2^L + 1}^{2^L + l} e_{n, t} - \rho_{n, t}\right|\right|.
    \]
    By the same arguments as in \cite{cuny2021rates}, we have, for any \(N\in\mathbb{N}\),
    \[
    \sup_{1\le k\le 2^{N+1}}\left|\left|\sum_{t=1}^k e_{t, n} - \sum_{t=1}^k \rho_{t, n}\right|\right| \le \left|\left|e_{1, n} - \rho_{1, n}\right|\right| + \sum_{L=0}^{N-1} D_{L, n} + D_{N, n}.
    \]
    Furthermore, they show that the array, \(\rho_{t, n}\), was constructed in such a way that there exists a constant \(c_0\) depending only on \(\delta\), \(d\) and \(\mathbb{E}||e_{1, n}||^{2+\delta}\) such that 
    \[
    \left|\left|D_{L, n}\right|\right|_1 \le C 2^{\frac{L}{2+\delta}}L^{\frac{1 + \delta}{2(2+\delta)}}
    \]
    for all \(L\in\mathbb{N}\). Since \(\mathbb{E}||e_{1, n}||^{2+\delta}\) is uniformly bounded in \(n\), we may assume that \(c_0\) depends only on \(\delta\) and \(d\). For \(2^N \le n < 2^{N+1}\), we have
    \begin{align*}
        \sum_{L=1}^{N-1} 2^{\frac{L}{2+\delta}}L^{\frac{1 + \delta}{2(2+\delta)}}
            & \le \left(\frac{\log n}{\log 2}\right)^{\frac{1 + \delta}{2(2+\delta)}}\sum_{L=0}^{N-1}2^{\frac{L}{2+\delta}} \\
            & = \left(\frac{\log n}{\log 2}\right)^{\frac{1 + \delta}{2(2+\delta)}}\frac{1 - 2^{\frac{N}{2+\delta}}}{1 - 2^{\frac{1}{2+\delta}}} \\
            &\le \frac{1}{\left(1-2^{\frac{1}{2+\delta}}\right)\left(\log 2\right)^{\frac{1+\delta}{2(2+\delta)}}}\left(\log n\right)^{\frac{1+\delta}{2(2+\delta)}}\left(n^{\frac{1}{2+\delta}} + 1\right) \\
            & \le c_1 n^{\frac{1}{2+\delta}} \left(\log n\right)^{\frac{1+\delta}{2(2+\delta)}}
    \end{align*}
    where \(c_1\) does not depend on \(n\). Finally, under the assumptions of the theorem, there exists a constant \(c_2\) not depending on n and such that
    \[
    \mathbb{E}\left|\left|e_{1, n} -  \rho_{1, n}\right|\right| \le c_2.
    \]
    Putting all the pieces together, we find that
    \[
    \mathbb{E}\left(\left|\left|\sup_{1\le k\le n}||\sum_{t=1}^k e_{t, n} - \sum_{t=1}^k \rho_{t, n}||\right|\right|\right) \le c_2 + c_0(c_1 + 1)n^{\frac{1}{2+\delta}} \left(\log n\right)^{\frac{1+\delta}{2(2+\delta)}}
    \]
    which is the result we wanted.
\end{proof}

Next we adapt the weak law of large numbers from Theorem 6 in \cite{de1998weak} to the multidimensional setting.

\begin{theorem} \label{thm: mda_wlln}
    Let \((e_{t, n})_{t, n\in\mathbb{N}}\) be an \(\mathbb{R}^d\)-valued martingale difference array wrt. \((\mathcal{F}_{t,n})_{t, n\in \mathbb{N}}\). Assume there exists \(\delta >0\) such that \(\sup_{t, n}\mathbb{E}||e_{t, n}||^{1 + \delta} < \infty\). Then, for any \(\epsilon>0\), it holds that
    \[
    \left|\left|\frac{1}{n}\sum_{t=1}^n e_{t, n}\right|\right| = o_p\left(n^{ \frac{1}{1+\delta} - 1 + \epsilon}\right).
    \]
\end{theorem}

\begin{proof}
    Fix \(\epsilon>0\) and let \(a\in\mathbb{R^d}\) and \(\tilde{e}_{t, n} = a^T e_{t, n}\). By Cauchy-Schwartz, we have 
    \[
    \sup_{t, n \in\mathbb{N}}|\tilde{e}_{t, n}|^{1 + \delta} \le ||a||^{1+\delta} \sup_{t, n \in \mathbb{N}}||e_{t, n}||^{1+\delta}= C < \infty.
    \]
    Now, define \(k_n = n^{\frac{1}{1+\delta}+\epsilon}\). Then, with \(p = 1 + \delta\),
    \[
    k_n^{-p} \sum_{t=1}^n \left(\mathbb{E}|\tilde{e}_{t, n}|\right)^p \le C k_n^{-p}n = o(1)
    \]
    and, by Theorem 6 in \cite{de1998weak},
    \[
    \left|\frac{1}{n}\sum_{t=1}^n \tilde{e}_{t, n}\right| = o_p\left(n^{\frac{1}{1 + \delta} - 1 + \epsilon}\right).
    \]
    The result then follows since \(a\) was arbitrary.
\end{proof}

Finally, we need the following well-known martingale difference array central limit theorem (see, e.g., Theorem 1 of Chapter VIII in \cite{pollard1984convergence}).

\begin{theorem} \label{thm: mda_clt}
    Let \((e_{t, n})_{t, n\in\mathbb{N}}\) be an \(\mathbb{R}^d\)-valued martingale difference array wrt. \((\mathcal{F}_{t,n})_{t, n\in \mathbb{N}}\). Assume that
    \[
    \sum_{t=1}^n\mathbb{E}\left(e_{t, n}e_{t, n}^T| \mathcal{F}_{t-1, n}\right) \rightarrow_p I
    \]
    and, for each \(\gamma > 0\),
    \[
    \sum_{t=1}^n \mathbb{E}\left(\left|\left|e_{t, n}\right|\right|^2 \mathbf{1}\left(\left|\left| e_{t, n}\right|\right| > \gamma\right)|\mathcal{F}_{t-1, n}\right) \rightarrow_p 0
    \]
    for \(n\rightarrow \infty\). Then, \(\sum_{t=1}^n e_{t, n} \rightarrow_w \mathcal{N}(0, I)\) for \(n\rightarrow \infty\).
\end{theorem}

\section{Gaussian Approximation} \label{app: approx}

In this section we detail how the Gaussian approximation described in Section 3.1 is achieved. Throughout we assume that \(X_{t, \theta}\) and \(\epsilon_{t, \theta}\) satisfy Assumptions \ref{ass: U} and \ref{ass: M} with \(F_\theta = I\). The key result is the strong invariance principle of Theorem \ref{thm: asip}. Although it is stated in terms of martingale difference arrays, the version we need (Lemma \ref{lem: gaus_app} below) follows easily from Proposition 8 in \cite{lundborg2021conditional} and Assumptions \ref{ass: mom}, \ref{ass: c}, and \ref{ass: sig}.

\begin{lemma} \label{lem: gaus_app}
    We can enlarge the initial probability space such that there exists a family of stochastic processes \((\rho_{t, \theta})_{t\ge 1, \theta\in\Theta}\) where, for each \(\theta\), the sequence \(\rho_{t, \theta}\) is i.i.d. \(d\)-dimensional gaussian with mean 0 and covariance matrix \(\Sigma\) and such that
    \[
    \sup_{\theta\in\Theta}\sup_{1 \le k \le n}\left|\left|\sum_{t = 1}^k \epsilon_{t, \theta} - \sum_{t = 1}^k \rho_{t, \theta}\right|\right| = o_p\left(n^{\frac{1}{2}-\beta}\right).
    \]
    for some \(\beta>0\).
\end{lemma}

\begin{lemma} \label{lem: cov_wlln}
    For any \(\epsilon > 0\), we have 
    \[
    n^{-\frac{1}{1 + \delta}-\epsilon} \sum_{t=1}^n \left(\epsilon_{t, \theta}\epsilon_{t, \theta}^T - \Sigma\right)\rightarrow_p 0
    \]
    uniformly over \(\Theta\).
\end{lemma}

\begin{proof}
    Fix \(a, b\in\mathbb{R}^d\) and define \(\xi_{t, \theta}=a^T(\epsilon_{t, \theta}\epsilon_{t, \theta}^T - \Sigma)b\) so that \(\xi_{t, \theta}\) is a one-dimensional martingale difference sequence for each \(\theta\in\Theta\). By Cauchy-Schwartz, we have, for all \(t\in\mathbb{N}\) and \(\theta\in\Theta\), \(|\xi_{t, \theta}| \le ||a\Sigma^{\frac{1}{2}}||||b\Sigma^{\frac{1}{2}}||\left(||\Sigma^{-\frac{1}{2}}\epsilon_{t, \theta}||^2 + 1\right)\) so that, by assumption,
    \[
    \sup_{\theta\in\Theta}\mathbb{E}|\xi_{t, \theta}|^{1+\frac{\delta}{2}} < \infty.
    \]
    The result then follows from Theorem \ref{thm: mda_wlln} in combination with Proposition 8 in \cite{lundborg2021conditional}.
\end{proof}

The following is similar to Lemma 4 in \cite{mikusheva2007uniform}. It shows that many of the important statistics can be replaced by their Gaussian counterpart.

\begin{lemma}\label{lem: gaus_rep}
    There exists \(\beta > 0\) such that
    \begin{enumerate}[label=(\alph*)]
        \item \(\sup_{\theta\in\Theta}\sup_{1\le t\le n}||X_{t, \theta}/\sqrt{n} - Y_{t, \theta}/\sqrt{n}|| = o_p(n^{-\beta})\),
        \item \label{lem: gaus_rep_b}\(\sup_{\theta\in\Theta}\sup_{1\le t\le n}\{ ||X_{t, \theta}/\sqrt{n}|| + ||Y_{t, \theta}/\sqrt{n}||\} = O_p(1)\),
        \item \(\sup_{\theta\in\Theta}||\sum_{t=1}^n\sum_{s=1}^t \epsilon_{s, \theta}\epsilon_{t, \theta}^T/n - \rho_{s, \theta}\rho_{t, \theta}^T/n|| = o_p(n^{-\beta})\).
        \item \(\sup_{\theta\in R_{n, d}}||H^{-\frac{1}{2}}(S_{XX}- S_{YY}) H^{-\frac{1}{2}}|| = o_p(n^{1 - \eta - \beta})\).
        \item \(\sup_{\theta\in R_{n, d}}||\sqrt{n}H^{-\frac{1}{2}}(S_{X\epsilon}- S_{Y\rho})|| = o_p(n^{\frac{3}{2}(1-\eta) - \beta})\).
    \end{enumerate}
\end{lemma}

\begin{proof}
    For the proof of part \textit{(a)} we use summation by parts to write
    \begin{align*}
        X_{t, \theta} 
            &= \sum_{s=1}^t\Gamma^{t-s}\epsilon_{s, \theta} 
            = \sum_{s=1}^t \epsilon_{s, \theta} - \sum_{s=1}^{t-1}\left(\Gamma^{t - s + 1} - \Gamma^{t - s}\right)\sum_{k=1}^{s+1}\epsilon_{k, \theta} \\
            &= \sum_{s=1}^t \epsilon_{s, \theta} - \left(\Gamma - I\right)\sum_{s=1}^{t-1}\Gamma^{t - s}\sum_{k=1}^{s+1}\epsilon_{k, \theta}
    \end{align*}
    and a similar expression holds for \(Y_{t, \theta}\). Thus,
    \begin{multline*}
       \sup_{\theta\in\Theta}\sup_{1\le t\le n}\frac{1}{\sqrt{n}}\left|\left|X_{t, \theta} - Y_{t, \theta}\right|\right|\le \sup_{\theta\in\Theta}\sup_{1\le t \le n}\left|\left|\left(\Gamma - I\right)\sum_{s=1}^{t-1}\Gamma^{t-s} + I\right|\right| \\
       \times \frac{1}{\sqrt{n}}\sup_{\theta\in\Theta}\sup_{1\le t \le n}\left|\left|\sum_{s=1}^t \epsilon_{s, \theta} - \sum_{s=1}^t \rho_{s, \theta}\right|\right|.
    \end{multline*}
    Since the first term on the right hand side is bounded by Assumptions \ref{ass: eig} and \ref{ass: jord}, Lemma \ref{lem: gaus_app} yields \textit{(a)}.

    To prove \textit{(b)}, we start with the same expression for \(Y_{t, \theta}\) as above. This gives us
    \[
    \sup_{\theta\in\Theta}\sup_{1\le t\le n}\frac{Y_{t, \theta}}{\sqrt{n}} \le \sup_{\theta\in\Theta}\sup_{1\le t \le n}\left|\left|\left(\Gamma - I\right)\sum_{s=1}^{t-1}\Gamma^{t-s} + I\right|\right|\frac{1}{\sqrt{n}}\sup_{\theta\in\Theta}\sup_{1\le t \le n}\left|\left|\sum_{s=1}^t \rho_{s, \theta}\right|\right|.
    \]
    Again, the first term on the right hand side is bounded uniformly over \(n\). For the second term, we have
    \begin{align*}
    \frac{1}{\sqrt{n}}\sup_{\theta\in\Theta}\sup_{1\le t \le n}\left|\left|\sum_{s=1}^t \rho_{s, \theta}\right|\right| 
        &\le \sup_{\theta\in\Theta}\left|\left|\Sigma^{\frac{1}{2}}\right|\right|\sup_{1\le t \le n}\left|\left|\frac{1}{\sqrt{n}}\sum_{s=1}^t \Sigma^{-\frac{1}{2}}\rho_{s, \theta}\right|\right| \\
        & \le C \sup_{\theta\in\Theta}\sup_{1\le t \le n}\left|\left|\frac{1}{\sqrt{n}}\sum_{s=1}^t \Sigma^{-\frac{1}{2}}\rho_{s, \theta}\right|\right| = O_p(1)
    \end{align*}
    since \(\Sigma^{-\frac{1}{2}}\rho_{s, \theta}\) is i.i.d. standard normal for all \(\theta\in\Theta\). The result then follows from \textit{(a)}.

    For \textit{(c)} we start with
    \[
    \frac{1}{n}\sum_{t=1}^n\sum_{s=1}^t \epsilon_{s, \theta}\epsilon_{t, \theta}^T = \frac{1}{2n}\left(\sum_{t=1}^n \epsilon_{t, \theta}\right)\left(\sum_{t=1}^n \epsilon_{t, \theta}\right)^T + \frac{1}{2n}\sum_{t=1}^{n}\left(\epsilon_{t, \theta}\epsilon_{t, \theta}^T - \Sigma\right) + \frac{1}{2}\Sigma.
    \]
    Lemma \ref{lem: cov_wlln} then yields
    \[
    \sup_{\theta\in\Theta}\left|\left|\frac{1}{n}\sum_{t=1}^n\sum_{s=1}^t \epsilon_{s, \theta}\epsilon_{t, \theta}^T - \frac{1}{2n}\left(\sum_{t=1}^n \epsilon_{t, \theta}\right)\left(\sum_{t=1}^n \epsilon_{t, \theta}\right)^T - \frac{1}{2}\Sigma\right|\right| =  o_p(n^{-\beta})
    \]
    and a similar argument holds for \(\frac{1}{n}\sum_{t=1}^n\sum_{s=1}^t \rho_{s, \theta}\rho_{t, \theta}^T\). Thus,
    \begin{align*}
        \sup_{\theta\in\Theta}\left|\left|\frac{1}{n}\sum_{t=1}^n\sum_{s=1}^t \epsilon_{s, \theta}\epsilon_{t, \theta}^T - \rho_{s, \theta}\rho_{t, \theta}^T\right|\right|
            &= \sup_{\theta\in\Theta}\left|\left|\frac{1}{n}\sum_{t=1}^n\sum_{s=1}^n \epsilon_{s, \theta}\epsilon_{t, \theta}^T - \rho_{s, \theta}\rho_{t, \theta}^T\right|\right| + o_p(n^{-\beta}) \\ 
            &\le C_n\sup_{\theta\in\Theta}\left|\left|\frac{1}{\sqrt{n}}\sum_{t=1}^n\epsilon_{t, \theta} - \rho_{t, \theta}\right|\right| + o_p(n^{-\beta})
    \end{align*}
    where \(C_n = \sup_{\theta\in\Theta}\left\{\left|\left|\frac{1}{\sqrt{n}}\sum_{t=1}^n \epsilon_{t, \theta}\right|\right| + \left|\left|\frac{1}{\sqrt{n}}\sum_{t=1}^n \rho_{t, \theta}\right|\right|\right\}.\) Since the law of \(\frac{1}{\sqrt{n}}\sum_{t=1}^n \rho_{t, \theta}\) is equal to a \(d\)-dimensional Gaussian with mean 0 and covariance matrix \(\Sigma\), Assumption \ref{ass: sig} yields \(\sup_{\theta\in\Theta}\left|\left|\frac{1}{\sqrt{n}}\sum_{t=1}^n \rho_{t, \theta}\right|\right| = O_p(1)\) and, by Lemma \ref{lem: gaus_app}, we get \(C_n = O_p(1)\) and therefore also the result in \textit{(c)}.

    To prove \textit{(d)}, we first note that 
    \[
    \sup_{\theta\in\Theta}\frac{1}{n}\left|\left|S_{XX} - S_{YY}\right|\right| \le C_n \sup_{\theta\in\Theta} \sup_{1\le t\le n} \frac{1}{\sqrt{n}}\left|\left|X_{t, \theta} - Y_{t, \theta}\right|\right|
    \]
    where \(C_n = \sup_{\theta\in\Theta}\frac{1}{\sqrt{n}}\left(\sup_{1\le t\le n}\left|\left|X_{t, \theta}\right|\right| + \sup_{1\le t\le n}\left|\left| Y_{t, \theta}\right|\right|\right).\) From Lemma \ref{lem: norm_asym}, we find that 
    \begin{equation} \label{eq: rate_sqrtnh}
       \sup_{\theta\in R_{n, d}}\sigma_{max}\left(H^{-\frac{1}{2}}\right) = \sup_{\theta\in R_{n, d}}\left(\sigma_{min}(H)\right)^{-\frac{1}{2}} = O(n^{-\frac{\eta}{2}}) 
    \end{equation}
    so, by part \textit{(a)} and \textit{(b)}, we get the result in \textit{(d)}.
    
    For the proof of \textit{(e)} we again use summation by parts to write
    \begin{align*}
    S_{X\epsilon} 
        &= \frac{1}{n} \left(X_{n-1, \theta}\sum_{t=1}^n \epsilon_{t, \theta}^T - \sum_{t=1}^{n-1}\left(X_{t, \theta} - X_{t-1, \theta}\right)\sum_{s=1}^t \epsilon_{s, \theta}^T\right) \\
        & = \frac{1}{n}\left(X_{n-1, \theta}\sum_{t=1}^n \epsilon_{t, \theta}^T - \left(\Gamma - I\right)\sum_{t=1}^{n-1}\sum_{s=1}^t X_{t-1, \theta} \epsilon_{s, \theta}^T - \sum_{t=1}^{n-1}\sum_{s=1}^t \epsilon_{t, \theta}\epsilon_{s, \theta}^T\right),
    \end{align*}
    and similarly for \(S_{Y\rho}\). We have, for all \(1\le t \le n\),
    \begin{multline*}
    \sup_{\theta\in\Theta}\left|\left|\frac{1}{n}\sum_{s=1}^t X_{t-1, \theta}\epsilon_{s, \theta}^T - Y_{t-1, \theta} \rho_{s, \theta}^T\right|\right| \\
        \le C_n \sup_{\theta\in\Theta}\left\{\sup_{1\le k\le n}\frac{1}{\sqrt{n}}\left(\left|\left|X_{k, \theta} - Y_{k, \theta}\right|\right| + \left|\left|\sum_{s=1}^k \epsilon_{s, \theta} - \rho_{s, \theta}\right|\right|\right)\right\},    
    \end{multline*}
    where \(C_n = \sup_{\theta\in\Theta}\left\{\sup_{1\le k \le n}||\frac{1}{\sqrt{n}}X_{k, \theta}|| + \sup_{1\le k\le n}\left|\left|\frac{1}{\sqrt{n}}\sum_{s=1}^k \rho_{s, \theta}\right|\right|\right\} = O_p(1)\) by part \textit{(b)}. From part \textit{(a)} and Lemma \ref{lem: gaus_app} it then follows that
    \[
    \sup_{\theta\in\Theta}\left|\left|\frac{1}{n}\sum_{s=1}^t X_{t-1, \theta}\epsilon_{s, \theta}^T - Y_{t-1, \theta} \rho_{s, \theta}^T\right|\right| = o_p\left(n^{-\beta}\right).
    \]
    We also have \(\sup_{\theta\in R_{n, d}}n||(\Gamma - I)|| = o_p(n^{1-\eta})\) so that
    \[
    \sup_{\theta\in R_{n, d}}\left|\left|\frac{1}{n}\left(\Gamma - I\right)\sum_{t=1}^{n-1}\sum_{s=1}^t X_{t-1, \theta} \epsilon_{s, \theta}^T - Y_{t-1, \theta}\rho_{s, \theta}\right|\right| = o_p\left(n^{1-\eta - \beta}\right)
    \]
    and, by part \textit{(c)}, \(\sup_{\theta\in R_{n, d}}\left|\left| S_{X\epsilon} - S_{Y\rho}\right|\right| = o_p\left(n^{1-\eta -\beta}\right).\) The result then follows from \eqref{eq: rate_sqrtnh}.
\end{proof}

Lemma \ref{lem: gaus_app} allows us to assume that \(\epsilon_{t, \theta}\) is an i.i.d. Gaussian sequence with mean zero and covariance matrix \(\Sigma\) for each \(\theta\in R_{n, d}\). Indeed, let \(\rho_{t, \theta}\) be as given in the Lemma and define the family \((Y_{t, \theta})_{t\in\mathbb{N}, \theta\in\Theta}\) by
\[
Y_{t, \theta} = \Gamma Y_{t-1, \theta} + \rho_{t, \theta}, \quad Y_{0, \theta} = 0
\]
and define the corresponding sample covariances
\[
S_{YY} = \frac{1}{n}\sum_{t=1}^n Y_{t-1, \theta}Y_{t-1, \theta}^T, \quad S_{Y\rho} = \frac{1}{n}\sum_{t=1}^n Y_{t-1, \theta}\rho_{t, \theta}^T.
\]
Then, by Lemma \ref{lem: gaus_rep}, we can pick \(\eta\) close enough to 1 such that
\begin{equation} \label{eq: gaus_rep}
    \sup_{\theta\in R_{n, d}} \left\{\left|\left|H^{-\frac{1}{2}}\left(S_{XX}- S_{YY}\right) H^{-\frac{1}{2}}\right|\right| + \left|\left|\sqrt{n}H^{-\frac{1}{2}}\left(S_{X\epsilon} - S_{Y\rho}\right)\right|\right|\right\} = o_p(1).
\end{equation}

\section{Simulations}

\subsection{Confidence intervals} \label{app: sim_ci}

For each \(n\in\{50, 75, 100\}\) and \(d\in\{3, 4, 5\}\) the simulation experiment is repeated 1000 times. In each repetition, for \(i,j=1,..., d\), we draw \(U_{ij}\sim \text{Unif}([0, 1])\) and set \(\Gamma= U^{-1}\Lambda_n U\) where \(\Lambda_n\in\mathbb{R}^{d\times d}\) is diagonal with \(\Lambda_{n, 11} = 1\) and \(\Lambda_{n, ii} = 1 - (1/n)^{1/(i-1)}\) for \(i=2,..., d\). We then sample \(\epsilon_t\sim \mathcal{N}(0, \Sigma)\) i.i.d. for \(t=1,...,n\) with
\[
\Sigma = \frac{1}{2}\left(I + \mathbbm{1}\mathbbm{1}^T\right)
\]
where \(\mathbbm{1} = (1,...,1)^T\in\mathbb{R}^d\) and let
\[
X_t = \Gamma X_{t-1} + \epsilon_t \text{ for }t=1,..., n, \quad X_0 = 0.
\]
\(X_t\) is a sample from a VAR(1) process with \(\theta = (\Gamma, \Sigma, \cdot)\). We then compute \(CI_b\), \(CI_{IV}\), and \(CI_{LA}\) for this sample and record the length of each confidence interval and whether it contains \(\Gamma_{11}\).

\subsection{Predictive regression testing}

For both simulation experiments we fix \(d=4\) and \(\alpha=0.1\). This implies that \(\tilde{\Gamma}\in\mathbb{R}^{3\times 3}\). The two regimes correspond to two different choices of \(\tilde{\Gamma}\):
\begin{itemize}
    \item \emph{Mixed Regime}: In this setting \(\tilde{\Gamma}\) is chosen as above, that is, with roots of differing proximity to unity and with random eigenvectors sampled anew for every simulation run.
    \item \emph{Non-stationary Regime}: In this setting we set \(\tilde{\Gamma}=I\) so that \(\tilde{X}_{t}\) is a random walk.
\end{itemize}
In both regimes the errors are i.i.d. Gaussian with covariance matrix \(\Sigma\) as given above. 

To obtain Figure 1, we do the following: For each \(n\in\{10, 20,..., 200\}\), we draw two samples \(X_{t}\) from the VAR(1) processes given by the two choices of \(\tilde{\Gamma}\) and under the null \(H_0:\gamma=0\). We then compute the three tests on both samples recording whether the null was rejected or not. This is repeated 1000 times and the rejection rate is the proportion of times the null was rejected across all simulations.

For Figure 2, we do essentially the same thing except that we now fix \(n=100\) and perform the experiment across different choices of \(\gamma\neq 0\). In particular, we run the experiment for \(\gamma = \delta\mathbbm{1}\), \(\delta\in\{0.005, 0.01,..., 0.1\}\) and record the proportion of times the null was rejected across all 1000 simulations.

\subsection{EAM} \label{app: sim_algo}

We present a slightly generalized version of the EAM-algorithm. EAM stands for Evaluation-Approximation-Maximization and the algorithm can more or less be split into three steps. It is an algorithm for solving problems of the following form
\begin{align*}
    &\sup f(x) \\
    &\text{s.t. } g(x) \le c(x) \\
\end{align*}
over \(x\in\mathcal{X}\subset\mathbb{R}^p\) where \(f\), \(g\), and \(c\) are fixed scalar functions sufficiently smooth and satisfying certain requirements. In \cite{kaido2019confidence} it is required that \(f(x)=v^Tx\) for some \(v\in\mathbb{R}^p\). Here we only require that \(f(x)\) is convex and twice continuously differentiable on \(\mathcal{X}\). We assume that \(c\) is costly to evaluate. Without going into too much detail, the algorithm proceeds as follows:
\begin{enumerate}
    \item \emph{Initialization:} Randomly sample initial points \(x^{(1)},..., x^{(k)}\) from \(\mathcal{X}\) and evaluate \(c(x^{(i)})\) for \(i=1,..., k\). Set \(L=k\).
    \item Iterate the following three steps until convergence:
    \begin{enumerate}
        \item \emph{E-step}: Evaluate \(c(x^{(L)})\) and pick current optimum
        \[
        y^{*, L} = \max\left\{f\left(x^{(i)}\right) : g\left(x^{(i)}\right) \le c\left(x^{(i)}\right), i=1,...,L \right\}
        \]
        \item \emph{A-step}: Approximate \(x\mapsto c(x)\) by a Gaussian process regression model, with mean \(\mu\), constant variance \(\sigma^2\), and covariance kernel \(K_\beta(x - x') = \exp(-\sum_{i=1}^p |x_i - x'_i|^2/\beta_i)\), fitted on \((c(x^{(i)}), x^{(i)})\), \(i=1,..., L\). Fitting the model yields the mean function \(c_L(x)\) and the variance function \(s_L(x)\) as well as the fitted parameters \(\hat{\mu}_L, \hat{\sigma}_L, \hat{\beta}_L\).
        \item \emph{M-step}: With probability \(1-\epsilon\), let
        \[
        x^{(L+1)} = \argmax_{x\in\mathcal{X}} EI_L(x)
        \]
        and with probability \(\epsilon\) draw \(x^{(L+1)}\) randomly from \(\mathcal{X}\). Set \(L = L + 1\). 
    \end{enumerate}
\end{enumerate}
\(EI_L(x)\) is the \emph{expected improvement function} and it is given by
\[
EI_L(x)=\left(f(x) - y^{*, L}\right)_+\left(1 - \Phi\left(\frac{g(x) - c_L(x)}{\hat{\sigma}_L s_L(x)}\right)\right)
\]
where \(\cdot_+ = \max\{\cdot, 0\}\) and \(\Phi\) is the standard normal CDF. Note that the optimization problem in the M-step can be reformulated as a constrained optimization problem with smooth objective function and smooth constraints for which all derivatives are known and it can therefore be solved with standard solvers. A key observation is that we only evaluate \(c\) once per iteration. In practice this results in far fewer evaluations of \(c\) when compared to, say, grid methods.

This algorithm was intended to compute confidence intervals that arise as projections of confidence regions exactly as is the case for \(CI_b\). In this case we would simply take \(x = \Gamma\) and let \(f(\Gamma)=\Gamma_{11}\), \(g(\Gamma)=\hat{t}^2_\Gamma\), and \(c(\Gamma)=\tilde{q}_{n, \Gamma}(\alpha)\). It does not really matter that \(\Gamma\) is a matrix since we can just vectorize it and redefine all the functions correspondingly. This would give us the upper bound of \(CI_b\) and the lower bound can be found by taking \(f(\Gamma)=-\Gamma_{11}\).

Similarly, the EAM-algorithm can be used to compute \(\phi_b\). This is done by letting \(x=\tilde{\Gamma}\) and then \(f(\tilde{\Gamma})=\hat{t}^2_{0|\tilde{\Gamma}}\) with \(g\) and \(c\) as before, but for \(\tilde{\Gamma}\) instead of \(\Gamma\). Note that in both cases \(f\) and \(g\) are polynomials of \(\text{vec}(x)\) and are therefore smooth with known derivatives of all orders.\footnote{All code used for the simulations can be found at \url{https://github.com/cholberg/unif_inf_var}. Our implementation of the EAM-algorithm is based on \cite{kaido2017calibrated}.}

\section{Lag augmentation} \label{app: lag}

We shall prove that the lag augmented estimator converges in distribution to a normal distribution uniformly over \(\Theta\) upon which it follows that standard inference is uniformly valid by the same arguments as applied in the proof of Theorem \ref{thm: valid_cr}.

\begin{lemma} \label{lem: la_inf}
    \(\Gamma_{LA}\) be defined as in Section \ref{sec: lag_aug}. Assume \ref{ass: U} and \ref{ass: M}. Then, as \(n\rightarrow \infty\),
    \[
    \sqrt{n}\textnormal{vec}\left(\hat{\Gamma}_{LA} - \Gamma\right)\rightarrow_w \mathcal{N}\left(0, \Sigma^{-1}\otimes \Sigma\right)
    \]
    uniformly over \(\Theta\).
\end{lemma}

\begin{proof}
    We write \(S_{\epsilon X'}= \sum_{t=2}^n \epsilon_{t}X_{t-2}^T/n\), \(S_{X'X'}=\sum_{t=2}X_{t-2}X_{t-2}^T/n\), \(S_{\epsilon'X'} = \sum_{t=2}\epsilon_{t-1}X_{t-2}^T/n\), \(S_{\epsilon'\epsilon'}=\sum_{t=2}^n\epsilon_{t-1}\epsilon_{t-1}^T\), and \(S_{\epsilon\epsilon'}=\sum_{t=2}^n\epsilon_t\epsilon_{t-1}^T\). Then,
    \begin{align*}
        \hat{\Gamma}_{LA} - \Gamma 
            &= S_{\epsilon\bar{X}} S_{\bar{X}\bar{X}}^{-1} D \\
            &= \left(S_{\epsilon X} - S_{\epsilon X'}S_{X'X'}^{-1}\left(S_{X'X'}\Gamma^T + S_{X'\epsilon'}\right)\right)\left(S_{\epsilon'\epsilon'} - S_{\epsilon'X'}S_{X'X'}^{-1}S_{X'\epsilon'}\right)^{-1} \\
            &= \left(S_{\epsilon\epsilon'} - S_{\epsilon X'}S_{X'X'}^{-1}S_{X'\epsilon'}\right)\left(S_{\epsilon'\epsilon'} - S_{\epsilon'X'}S_{X'X'}^{-1}S_{X'\epsilon'}\right)^{-1}
    \end{align*}
    where we used the relation \(X_{t-1} = \Gamma X_{t-2} + \epsilon_{t-1}\) multiple times. By Theorem \ref{thm: unif_app}, \(S_{\epsilon X'}S_{X'X'}^{-1}S_{X'\epsilon'}\rightarrow_p 0\) and \(S_{\epsilon' X'}S_{X'X'}^{-1}S_{X'\epsilon'}\rightarrow_p 0\) for \(n\rightarrow\infty\) uniformly over \(\Theta\). Furthermore, Theorem \ref{thm: mda_wlln} and \ref{thm: mda_clt} in the Appendix yield \(S_{\epsilon'\epsilon'} \rightarrow_p \Sigma\) and \(\sqrt{n}\textnormal{vec}(S_{\epsilon\epsilon'})\rightarrow_w \mathcal{N}(0, \Sigma\otimes\Sigma)\) for \(n\rightarrow \infty\) uniformly over \(\Theta\). Finally, since \(\Sigma\) is uniformly invertible and bounded on \(\Theta\), the uniform versions of the continuous mapping theorem and Slutsky's Lemma (Proposition 9 and Proposition 15 in \cite{lundborg2021conditional}), yield the desired result.
\end{proof}

\section{IVX} \label{app: ivx}

We shall prove that the IVX \(t^2\)-statistic converges in distribtuion to \(\chi^2_{d^2}\) uniformly over \(\Theta\) after which the rest follows by the same arguments as those applied in the proof of Theorem \ref{thm: valid_cr}.

\begin{theorem} \label{thm: ivx}
    Assume that Assumptions \ref{ass: M} and \ref{ass: U} are true. Let \(\hat{\Gamma}_{IV}\) be the IVX estimator and \(\hat{t}^2_{IV, \Gamma}\) the corresponding \(t^2\)-statistic as defined in Section \ref{sec: ivx}. Then, for \(n\rightarrow\infty\),
    \[
    \hat{t}^2_{IV, \Gamma}\rightarrow \chi^2_{d^2}
    \]
    uniformly over \(\Theta\).
\end{theorem}

The proof of Theorem \ref{thm: ivx} that we present here is conceptually different from the proofs presented so far. We rely on the theory developed in \cite{magdalinos2020econometric, phillips2009econometric}, but since we do not require that all roots approach unity at the same rate, there are some extra difficulties that need to be dealt with. In particular, we need to employ a different normalization in obtaining the asymptotics of \(S_{ZZ}\) and \(S_{\epsilon Z}\). Furthermore, Theorem \ref{thm: ivx} shows that the suggested IVX approach is truly uniformly valid (at least over the suggested parameter space, \(\Theta\)). The first lemma is of a technical nature.

\begin{lemma} \label{lem: ivx_seq}
    Let \((\theta_n)_{n\in\mathbb{N}}\subset\Theta\) satisfy Assumptions \ref{ass: M} and \ref{ass: U} with \(F_{\theta_n}=I\) and \(\beta\in(0, 1)\). Then, there exist \(0\le r\le d\), \((k_n)_{n\in\mathbb{N}}\subset \mathbb{N}\) strictly increasing, and \((\tilde{\theta})_{n\in\mathbb{N}}\subset\Theta\) such that
    \begin{enumerate}[label=(\roman*)]
        \item \(\theta_{k_n} = \tilde{\theta}_{k_n}, \; \forall n\in\mathbb{N}\), \label{lem: ivx_seq_i}
        \item \(n^\beta (1-\tilde{\Gamma}_{n, ii})\rightarrow \kappa_i\in\mathbb{C}, \; |\kappa_i|\in[0,1], \; \text{ for } 1\le i \le r\), \label{lem: ivx_seq_ii}
        \item \(n^{-\beta} (1-\tilde{\Gamma}_{n, ii})^{-1}\rightarrow \kappa_i\in\mathbb{C}, \; |\kappa_i|\in[0,1], \; \text{ for } r+1\le i \le d\), \label{lem: ivx_seq_iii}
        \item \(\tilde{\theta}_n\rightarrow\theta\in\Theta\), \label{lem: ivx_seq_iv}
    \end{enumerate}
    where \(\tilde{\theta}_n = (\tilde{\Gamma}_n, \tilde{\Sigma}_n, \cdot)\) and all limits are taken as \(n\rightarrow \infty\).
\end{lemma}

\begin{proof}
    Fix \((\theta_n)_{n\in\mathbb{N}}\subset\Theta\) and \(\beta\in(0, 1)\). For each \(n\in\mathbb{N}\), let \(0\le r_n\le d\) be such that \(|n^\beta(1-\Gamma_{n, ii})|\le 1\) for \(1\le i\le r_n\) and \(|n^{-\beta}(1-\Gamma_{n, ii})^{-1}|\le 1\) otherwise. Then, \((r_n)_{n\in\mathbb{N}}\) is a sequence in \(\{0, 1,..., d\}\) so that, by compactness, it has a convergent sub-sequence. In other words, there exists a sub-sequence, \((\theta_{n_k})_{k\in\mathbb{N}}\), and \(0\le r\le d\) such that \(|(n_k)^\beta (1 - \Gamma_{n_k, ii})| \le 1\) for \(1\le i\le r\) and \(|(n_k)^{-\beta} (1 - \Gamma_{n_k, ii})^{-1}| \le 1\) otherwise. By Bolzano-Weierstrass, we may assume without loss of generality (passing to another sub-sequence if necessary) that there exists \(\kappa\in \mathbb{C}^d\) with \(|\kappa_i|\le 1\) and such that
    \begin{align*}
        \left(n_k\right)^\beta\left(1 - \Gamma_{n_k, ii}\right)\rightarrow\kappa_i, \; &\text{ for } 1\le i \le r \\
        \left(n_k\right)^{-\beta}\left(1 - \Gamma_{n_k, ii}\right)^{-1}\rightarrow\kappa_i, \; &\text{ for } r+1\le i \le d
    \end{align*}
    for \(k\rightarrow\infty\). By another compactness argument, we can furthermore choose the sub-sequence such that \(\theta_{n_k}\rightarrow \theta = (\Gamma, \Sigma, c)\in\Theta\). Now, take some \(\delta\in(0, \beta)\), let \(0\le r_1\le r_2\le d\) be such that \(|\kappa_i|>0\) for \(r_1< i \le r_2\) and \(\kappa_i = 0\) otherwise, and define the diagonal matrix \(C_n\in\mathbb{C}^{d\times d}\) by
    \[
    C_{n, ii} = \begin{cases}
        n^{-\delta},  & \text{ if } i\le r_1, \\
        \kappa_i,       & \text{ if } r_1< i \le r, \\
        \kappa_i^{-1},   & \text{ if } r < i \le r_2, \\
        n^{\delta},    & \text{ otherwise}.
    \end{cases}
    \]
    By Assumptions \ref{ass: eig} and \ref{ass: jord}, we must have \((\Gamma_{i,j})_{1\le i, j\le r_2} = I_{r_2}\) so we find that \(\Gamma_n' = \Gamma - n^{-\beta}C_n\) satisfies \ref{lem: ivx_seq_ii} and \ref{lem: ivx_seq_iii} with \(\Gamma_n'\rightarrow \Gamma\) for \(n\rightarrow \infty\). Finally, let \(\tilde{\Gamma}_n = \Gamma_{n_k}\) if \(n=n_k\) for some \(k\in\mathbb{N}\) and \(\tilde{\Gamma}_n = \Gamma_n'\) otherwise and \(\tilde{\Sigma}_n = \Sigma_{n_k}\) and \(\tilde{c}_n = c_{n_k}\) for \(n_k\le n < n_{k+1}\). Then, \(\tilde{\theta}_n=(\tilde{\Gamma}_n, \tilde{\Sigma}_n, \tilde{c}_n)\) satisfies all the conditions.
\end{proof}

Sequences of parameters like \(\tilde{\theta}_n\) in the above Lemma fit nicely into the framework of \cite{magdalinos2020econometric, phillips2009econometric}. We can adapt their results to this more general setup. Fix some \(\beta\in(0, 1)\) and consider a sequence \((\theta_n)_{n\in\mathbb{N}}\subset\Theta\) such that conditions \ref{lem: ivx_seq_ii} and \ref{lem: ivx_seq_iii} are satisfied for some \(0\le r\le d\) and \(\kappa\in\mathbb{C}^d\) with \(|\kappa_i|\le 1\). For such a sequence, we can define the integers \(0\le r_1\le r_2\le d\) as in the proof above along with the diagonal matrices \(D_n\in \mathbb{C}^{d\times d}\) given by \(D_{n, ii} = n^{-\beta}\) for \(1\le i\le r_2\) and \(D_{n, ii} = (1 - |\Gamma_{n, ii}|)\) otherwise. This normalization is sufficiently flexible to ensure convergence of the relevant sample covariance matrices.

\begin{lemma} \label{lem: ivx_asym}
    Let \(\beta\in(\frac{1}{2}, 1)\) and \((\theta_n)_{n\in\mathbb{N}}\subset\Theta\) be a sequence of parameters satisfying Assumptions \ref{ass: M} and \ref{ass: U} with \(F_{\theta_n} = I\) as well as \ref{lem: ivx_seq_ii}, \ref{lem: ivx_seq_iii}, and \ref{lem: ivx_seq_iv} of Lemma \ref{lem: ivx_seq} for some \(0\le r \le d\), \(\kappa\in\mathbb{C}^d\) with \(|\kappa_i|\le 1\), and \(\theta\in\Theta\). For \((D_n)_{n\in\mathbb{N}}\) as defined above and \(\textnormal{vec}(V)\sim \mathcal{N}(0, I)\), there exists a sequence of positive definite matrices \((\Sigma_{Z, n})_{n\in\mathbb{N}}\) such that
    \[
    \limsup_{n\rightarrow\infty} \left\{\sigma_{min}\left(\Sigma_{Z, n}\right)^{-1} + \sigma_{max}\left(\Sigma_{Z, n}\right)\right\} < \infty
    \]
    and the following holds for any \(\epsilon > 0\)
    \begin{equation} \label{eq: ivx_zz}
        \lim_{n\rightarrow\infty}\mathbb{P}\left(\left|\left|D_n^{\frac{1}{2}}S_{ZZ}D_n^{\frac{1}{2}} - \Sigma_{Z, n}\right|\right| > \epsilon\right) = 0,
    \end{equation}
    \begin{equation} \label{eq: ivx_ze}
        \lim_{n\rightarrow\infty}d_{BL}\left(\sqrt{n}\Sigma_{Z,n}^{-\frac{1}{2}}D_n^{\frac{1}{2}}S_{Z\epsilon}\Sigma^{-\frac{1}{2}}, V\right) = 0.
    \end{equation}
\end{lemma}

The following result is useful for the proof of Lemma \ref{lem: ivx_asym}. With a slight abuse of notation, for any \(\theta\in\Theta\), let \((I-\Gamma)^{\frac{1}{2}}\) be the diagonal matrix given by the principal square root of the diagonal of \(I-\Gamma\).

\begin{lemma} \label{lem: h_lam}
    Assume that Assumptions \ref{ass: M} and \ref{ass: U} hold with \(F_\theta = I\). Then,
    \begin{equation} \label{eq: h_lam_bound}
        \sup_{\theta\in\Theta}\sup_{t\ge 1}\left|\left|\mathbb{E}\left((I-\Gamma)^{\frac{1}{2}}X_{t, \theta}X_{t, \theta}^T(I-\Gamma)^{\frac{1}{2}}\right)\right|\right| < \infty
    \end{equation}
    and, furthermore,
    \begin{equation} \label{eq: h_lam_lim}
        \lim_{n\rightarrow\infty}\sup_{\theta\in R_{n, 0}}\left|\left|(I-\Gamma)^{\frac{1}{2}}H(I - \Gamma)^{\frac{1}{2}} - \Sigma_X\right|\right| = 0
    \end{equation}
    where \(\textnormal{vec}(\Sigma_X)=(I-\Gamma)^{\frac{1}{2}}\otimes(I-\Gamma)^{\frac{T}{2}}(I - \Gamma\otimes\Gamma^T)^{-1}\textnormal{vec}(\Sigma)\) with
    \[
    \limsup_{n\rightarrow\infty}\sup_{\theta\in R_{n, 0}}\left\{\sigma_{min}\left(\Sigma_X\right)^{-1} + \sigma_{max}\left(\Sigma_X\right)\right\} < \infty.
    \]
\end{lemma}

\begin{proof}
    For any \(\theta\in \Theta\) we have
    \begin{align*}
        \left|\left|\mathbb{E}\left((I-\Gamma)^{\frac{1}{2}}X_{t, \theta}X_{t, \theta}^T(I-\Gamma)^{\frac{T}{2}}\right)\right|\right| 
            &= \left|\left|(I-\Gamma)^{\frac{1}{2}}\sum_{s=0}^{t-1}\Gamma^s\Sigma\left(\Gamma^s\right)^T(I-\Gamma)^{\frac{T}{2}}\right|\right| \\
            &\le ||\Sigma||\sum_{s=0}^{t-1}\left|\left|(I-\Gamma)^{\frac{1}{2}}\Gamma^{s}\right|\right|^2
    \end{align*}
    Due to the block diagonal structure of \(\Gamma\) (Assumption \ref{ass: jord}) and Assumption \ref{ass: sig}, there exists some generic constant \(c_0>0\) such the last term in above inequality is bounded by
    \[
    c_0\left(\sup_{\theta\in\Theta\;:\;|\lambda_{N}|>1-\alpha}\sum_{s=0}^{t-1}\left\|(I-\Gamma)^{\frac{1}{2}}\Gamma^{s}\right\|^2 + \sup_{\theta\in\Theta\;:\;|\lambda_{1}|\le1-\alpha}\sum_{s=0}^{t-1}\left\|(I-\Gamma)^{\frac{1}{2}}\Gamma^{s}\right\|^2\right).
    \]
    By equation \eqref{eq: t_la}, the second term is converges for \(t\rightarrow\infty\). For the first term, we use the fact that, for any \(\theta\in\Theta\), the condition \(|\lambda_N|>1-\alpha\) implies that \(\Gamma\) is diagonal and, thus, 
    \[
    \sum_{s=0}^{t-1}\left\|(I-\Gamma)^{\frac{1}{2}}\Gamma^{s}\right\|^2 \le \sum_{i=1}^N\sum_{s=0}^{t-1}m_i|1-\lambda_i||\lambda_i|^s\le d^2r_\alpha \frac{1-|\lambda_i|^t}{|\lambda_i|}\le \frac{d^2 r_\alpha}{1-\alpha}
    \]
    where in the first inequality we used the fact that \(|\lambda_i|^{2s}\le |\lambda_i|^s\) since \(|\lambda_i|\le 1\) and in the second inequality we used Assumption \ref{ass: eig}, the fact that \(N, m_i\le d\) and that \(|\lambda_i|=1\) implies that \(\lambda_i = 1\) and therefore \(|1 - \lambda_i||\lambda_i|^s=0\) for all \(s=0, \dots, t-1\) in this case. This proves \eqref{eq: h_lam_bound}.

    For the proof of \eqref{eq: h_lam_lim}, simply note that Lemma \ref{lem: norm_exp} and \eqref{eq: h_lam_bound} imply that 
    \[
    \lim_{n\rightarrow\infty}\sup_{\theta\in R_{n, 0}}\left|\left|(I-\Gamma)^{\frac{1}{2}}\left(H - \mathbb{E}(X_{n-1,\theta}X_{n-1, \theta}^T)\right)(I - \Gamma)^{\frac{T}{2}}\right|\right| = 0
    \]
    and
    \begin{align*}
        &\sup_{\theta\in R_{n, 0}}\left|\left|(I-\Gamma)^{\frac{1}{2}}\mathbb{E}(X_{n-1,\theta}X_{n-1, \theta}^T)(I - \Gamma)^{\frac{T}{2}} - \Sigma_X\right|\right| \\
        =& \sup_{\theta\in R_{n, 0}}\left|\left|(I-\Gamma)^{\frac{1}{2}}\sum_{s=n-1}^{\infty}\Gamma^s\Sigma\left(\Gamma^s\right)^T(I-\Gamma)^{\frac{T}{2}}\right|\right| \rightarrow 0
    \end{align*}
    for \(n\rightarrow \infty\). It remains to check that \(\Sigma_X\) is uniformly bounded and invertible in the limit. Since \(\limsup_n \sup_{\theta\in R_{n, 0}}\sigma_{max}(\Sigma_X) < \infty\) follows immediately from \eqref{eq: h_lam_bound}, we only need to show the latter. For any \(\theta\in R_{n, 0}\) diagonal, we have
    \begin{align*}
        \sigma_{min}(\Sigma_X) &\ge \sigma_{min}(\Sigma)\sum_{t=0}^{\infty}\sigma_{min}\left((I-\Gamma)^{\frac{1}{2}}\Gamma^{t}\right)^2 \\
            &\ge \sigma_{min}(\Sigma)\sum_{t=0}^{\infty}\min_{1\le k\le d}(1 - |\lambda_{i_k}|)|\lambda_{i_k}|^{2t} \\
            &\ge \sigma_{min}(\Sigma)\sum_{t=0}^{\infty}\frac{\log n}{n}\left(1 - \frac{\log n}{n}\right)^{2t} \\
            &= \sigma_{min}(\Sigma)\frac{\log n}{n}\left(1 - \left(1 - \frac{\log n}{n}\right)^2\right)^{-1} \\
            &\ge \frac{\sigma_{min}(\Sigma)}{2}.
    \end{align*}
    If \(\Gamma\) is non-diagonal, the same bound holds since adding ones on the super-diagonal does not decrease the minimum singular value. Because \(\Sigma\) is uniformly invertible over \(\Theta\), the proof is then complete.
\end{proof}

\begin{proof}[Proof of Lemma \ref{lem: ivx_asym}]
    Throughout the proof we write matrices as \(3\times 3\) block matrices such that the top-left block is \(r_1\times r_1\), the middle block is \((r_2-r_1)\times (r_2-r_1)\), and the bottom-left block is \((d-r_2)\times(d-r_2)\). We use a superscript to denote the block index, e.g., \(S_{ZZ}^{13}\) denotes the top-right block of \(S_{ZZ}\). Furthermore, for a diagonal matrix \(A\) with complex values, we let \(A^{\frac{1}{2}}\) denote the diagonal matrix obtained by taking the principal square root of the diagonal of \(A\). Let \(\tilde{Z}_t = \sum_{s=1}^t (1-n^{-\beta})^{t-s}\epsilon_s\) and \(\psi_t = \sum_{s=1}^t (1-n^{-\beta})^{t-s}X_{s-1}\) so that
    \begin{equation} \label{eq: z_decomp}
        Z_t = \tilde{Z}_t + \left(\Gamma_n - I\right)\psi_t.
    \end{equation}
    Let \(\Lambda_n = (I-\Gamma_n)^{\frac{1}{2}}\) be the diagonal matrix as defined in Lemma \ref{lem: h_lam} above. Then, since \(\Gamma_n\) and \(\Lambda_n\) commute, with \(c\) denoting some generic constant not depending on \(t\) or \(n\),
    \begin{equation} \label{eq: psi_ineq}
    \begin{split}
        \mathbb{E}\left|\left|\Lambda_n\psi_t\right|\right|^2
            &=\sum_{i,j=1}^t(1-n^{-\beta})^{2t-i-j}\tr\left(\Lambda_n\mathbb{E}\left(X_{i-1}X_{j-1}^T\right)\Lambda_n\right)\\
            &\le 2\sum_{1\le j\le i \le t}(1-n^{-\beta})^{2t-i-j}\left|\tr\left(\Gamma_n^{i-j}\Lambda_n\mathbb{E}\left(X_{i-1}X_{j-1}^T\right)\Lambda_n\right)\right| \\
            &\le 2c\sum_{i, j=1}^t (1-n^{-\beta})^{2t-i-j}\left\|\Gamma^{i - j}\Lambda_n\right\| \\
            &\le 2c\sum_{i=0}^{t-1}(1-n^{-\beta})^i\sum_{j=0}^{t-i-1}\left\|\left((1-n^{-\beta})\Gamma\right)^j\Lambda_n\right\|
    \end{split}
    \end{equation}
    where the second inequality follows from Lemma \ref{lem: h_lam} in the Appendix and the Cauchy-Schwartz inequality. This inequality yields a result equivalent to equation (40) in \cite{phillips2009econometric}. In particular, we deduce that \(\sup_{1\le t \le n}\mathbb{E}\left|\left|(\Gamma_n - I)\psi_t\right|\right|^2 = o(n)\) from which it follows that
    \begin{equation} \label{eq: s_pe}
        S_{Z\epsilon} = S_{\tilde{Z}\epsilon} + o_p(1).
    \end{equation}
    
    We first prove \eqref{eq: ivx_zz}. For ease of notation, we write \(S_n = D_n^{\frac{1}{2}}S_{ZZ}D_n^{\frac{1}{2}}.\) Since \(D_n^{11}=n^{-\beta} I_{r_1}\), essentially the same proof as that of Lemma 3.1.(iii) in \cite{phillips2009econometric} using \eqref{eq: psi_ineq} shows that
    \[
    S_n^{11} = n^{-\beta} S^{11}_{\tilde{Z}\tilde{Z}} + o_p(1) = \frac{1}{2}\Sigma^{11} + o_p(1),
    \]
    where the latter equality follows from Theorem \ref{thm: stat_asym} and the fact that \(n^{-\beta}\mathbb{E}(S_{\tilde{Z}\tilde{Z}}) \rightarrow \Sigma/2\) for \(n\rightarrow \infty\). Similarly, the proof of Lemma 3.5.(ii) in \cite{phillips2009econometric} can be adapted to show that
    \[
    S_n^{33} = \left(D_{n}^{33}\right)^{\frac{1}{2}} S^{33}_{XX} \left(D_{n}^{33}\right)^{\frac{1}{2}} + o_p(1) =  \Sigma^{33}_{X, n} + o_p(1),
    \]
    where the latter equality follows from Theorem \ref{thm: stat_asym} and Lemma \ref{lem: h_lam} in the Appendix and \(\Sigma^{33}_{X, n}\) is defined as in Lemma \ref{lem: h_lam} in the Appendix but emphasizing the dependence on \(n\). For the middle block, using the recursive relations \(Z_t = (1-n^{-\beta})Z_{t-1} + \Delta X_t\) and \(\Delta X_t = (\Gamma_n - I)X_{t-1} + \epsilon_t\), we can write
    \[
    \left(1 - \left(1-n^{-\beta}\right)^2\right)S^{22}_{ZZ} = S^{22}_{\Delta X Z} + S^{22}_{Z \Delta X} + S^{22}_{\Delta X \Delta X} + o_p(1).
    \]
    It follows from Theorem \ref{thm: stat_asym}, that \(S^{22}_{\Delta X \Delta X} = \Sigma^{22} + o_p(1)\). For the other two terms, we use \eqref{eq: s_pe} and write \(S^{22}_{Z\Delta X}= S^{22}_{\tilde{Z}\epsilon} + S^{22}_{ZX}(\Gamma_n^{22} - I)^T + o_p(1)\). Using the recursive relations and \eqref{eq: s_pe} once more yields
    \[
    \left(I - \left(1-n^{-\beta}\right)\Gamma^{22}_n\right)S^{22}_{XZ} = S^{22}_{X\epsilon} + S^{22}_{\epsilon\tilde{Z}} + S^{22}_{\epsilon\epsilon} + S^{22}_{XX}(\Gamma^{22}_n - I) + o_p(1)
    \]
    It follows from Theorem \ref{thm: stat_asym} that the first two terms tend to 0 in probability for \(n\rightarrow\infty\) and \(S^{22}_{\epsilon\epsilon} = \Sigma^{22} + o_p(1)\). If we define \(K\in\mathbb{C}^{(r_2 - r_1)\times(r_2 - r_1)}\) diagonal with \(K_{ii} = \kappa_i\) for \(i\le r_2 - r\) and \(K_{ii} = \kappa_i^{-1}\) otherwise, we get \(n^\beta\left(I - \Gamma^{22}_n\right) \rightarrow K\), \(n^\beta\Lambda_{n}^{22} \rightarrow K^{\frac{1}{2}}\), and \(n^\beta\left(I - (1-n^{-\beta})\Gamma_n^{22}\right) \rightarrow K + I\) for \(n\rightarrow\infty\). Lemma \ref{lem: h_lam} in the Appendix then yields
    \[
    \left(\Gamma_n^{22}-I\right)S_{XZ}^{22}=\left(K+I\right)^{-1}\left(K\Sigma^{22} + K^{\frac{1}{2}}\Sigma_{X, n}^{22}K^{\frac{1}{2}}\right) + o_p(1).
    \]
    and \((\Lambda^{22}_n)^{\frac{1}{2}}\otimes (\Lambda^{22}_n)^{\frac{T}{2}}(I-\Gamma_n^{22}\otimes (\Gamma_n^{22})^T)^{-1} \rightarrow (K\otimes K^T)^{\frac{1}{2}}(I\otimes K^T + K\otimes I)^{-1}\) for \(n\rightarrow \infty\) so that (noting that \(K_{ii}\) has strictly positive real part for all \(i\))
    \[
    \Sigma^{22}_{X, n}\rightarrow K^{\frac{1}{2}}\int_0^{\infty}e^{-sK}\Sigma^{22}e^{-sK^T} ds K^{\frac{T}{2}} = K^{\frac{1}{2}}\Omega^{22}K^{\frac{T}{2}}.
    \]
    Then, using the relation \(\Sigma^{22} - \Omega^{22}K =  K\Omega^{22}\), the limiting expression simplifies to
    \[
    (\Gamma_n^{22} - I)S_{XZ}^{22} = (K+I)^{-1}K^2\Sigma^{22}.
    \]
    Finally, since \(n^\beta(1-(1-n^{-\beta})^2) = 2 + o(1)\) and \(D_{n}^{22} = n^{-\beta}I_{r_2-r_1}\), we find 
    \begin{align*}
        S_n^{22} 
            &= \frac{1}{2}\left(\Sigma^{22} + \left(K+I\right)^{-1}K^2\Omega^{22} + \Omega^{22}(K^T)^2\left(K+I\right)^{-T}\right) + o_p(1) \\
            &= \frac{1}{2}\left((K+I)^{-1}K\Omega^{22} + \Omega^{22}K^T(K+I)^{-T}\right) + o_p(1) \\
            &= \frac{1}{2}(I+K)^{-1}\left(2K\Omega^{22}K^T + \Sigma\right)(I+K)^{-T} + o_p(1).
    \end{align*}
    We have yet to characterize the asymptotic behaviour of the off-diagonal blocks. First, note that by (23) in \cite{phillips2009econometric} and \eqref{eq: z_decomp}, we get
    \[
    S_{ZZ}^{32} - S_{XZ}^{32} = -n^{-\beta}\left(S_{\psi\psi}^{32}(\Gamma_n^{22} - I)^T - S_{\psi\tilde{Z}}^{32}\right)
    \]
    so that \eqref{eq: psi_ineq} and Theorem \ref{thm: stat_asym} yield \(S_n^{32} - \left(D_n^{33}\right)^{\frac{1}{2}} S_{XZ}^{32} \left(D_n^{22}\right)^{\frac{1}{2}} =  o_p(1).\) Similar to above, we have
    \[
    \left(I - (1-n^{-\beta})\Gamma_n^{33}\right)S_{XZ}^{32} = S^{32}_{XX}\left(\Gamma_n^{22} - I\right)^T + o_p(1).
    \]
    But then, because \(n^{-\beta}(I - \Gamma_n^{33})=o(1)\), we find that
    \[
    n^{-\frac{\beta}{2}}\left(\Lambda_n^{33}\right)^{\frac{1}{2}}\left(I - \left(1 - n^{-\beta}\right)\Gamma_n^{33}\right)^{-1} = o(1)
    \]
    and, by Lemma \ref{lem: h_lam} in the Appendix and Theorem \ref{thm: stat_asym}, \(S_{XX}^{32}(\Gamma_n^{22} - I)^T = o_p(1)\). Thus, \(S_n^{32} = o_p(1)\). A similar argument show that \(S_n^{31} = o_p(1)\) so that the only block left is \(S_n^{21}\). As a consequence of \eqref{eq: z_decomp}, we have
    \[
    n^{-\beta}\left|\left|S_{ZZ}^{12} - S_{\tilde{Z}Z}^{12}\right|\right| = n^{-\beta}\left|\left|(\Gamma_n^{11}-I)S_{\psi\psi}^{12}(\Gamma_n^{22}-I)^T + (\Gamma_n^{11}-I)S_{\psi\tilde{Z}}^{12}\right|\right|
    \]
    and arguments like the one employed in the proof of Lemma 3.1 in \cite{phillips2009econometric} in combination with \eqref{eq: psi_ineq} shows that the right hand side is \(o_p(1)\). Using the recursive relations for \(Z_t\), \(\tilde{Z}_t\), and \(\Delta X_t\) in combination with \eqref{eq: s_pe} and Theorem \ref{thm: stat_asym}, we have
    \begin{align*}
        \left(1-(1-n^{-\beta})^2\right)S_{\tilde{Z}Z}^{12}
            &=S^{12}_{\epsilon Z} + S_{\tilde{Z}\Delta X}^{12} + S_{\epsilon\Delta X}^{12} + o_p(1) \\
            &= S_{\tilde{Z}X}^{12}(\Gamma_n^{22} - I)^T + S^{12}_{\epsilon\epsilon} + o_p(1).
    \end{align*}
    An application of Lemma \ref{lem: h_lam} in the Appendix and Theorem \ref{thm: stat_asym} yields \(S_{\tilde{Z}X}^{12} = \Omega^{12}_n K^{\frac{T}{2}} + o_p(1)\) where
    \[
    \Omega^{12}_n = \Sigma^{12}\left(I - (1-n^{-\beta})\Gamma_n^{22}\right)^T n^{-\frac{\beta}{2}}\left(I - \Lambda_n^{22}\right)^{-\frac{T}{2}} \rightarrow \Sigma^{12}K^{\frac{T}{2}}\left(I + K\right)^{-T} 
    \]
    for \(n\rightarrow\infty\). In conclusion, since \(D_n^{22} = n^{-\beta}I_{r_2-r_1}\) and \(D_n^{33} = n^{-\beta}I_{r_1}\), we get
    \[
    S_n^{12} = \frac{1}{2}\left(\Sigma^{12} + \Sigma^{12}K^T(I+K)^{-T}\right) + o_p(1) = \frac{1}{2}\Sigma^{12} + o_p(1).
    \]
    Collecting all the limiting expressions, we define
    \[
    \bar{K} = \begin{pmatrix}
        I_{r_1} & 0 \\ 0 & K
    \end{pmatrix}, \quad \Omega = \int_0^\infty e^{-s\bar{K}} \begin{pmatrix}
        \Sigma^{11} & \Sigma^{12} \\
        \Sigma^{21} & \Sigma^{22}
    \end{pmatrix} e^{-s\bar{K}^T}ds
    \]
    and observe that
    \begin{align*}
        \Sigma^{11} &= \left((I+\bar{K})^{-1}\left(2\bar{K}\Omega \bar{K}^T + \Sigma\right)(I+\bar{K})^{-T}\right)^{11} \\
        \Sigma^{12} &= \left((I+\bar{K})^{-1}\left(2\bar{K}\Omega \bar{K}^T + \Sigma\right)(I+\bar{K})^{-T}\right)^{12}.
    \end{align*}
    so that \(S_n = \Sigma_{Z, n} + o_p(1)\) with
    \[
    \Sigma_{Z, n} =  \frac{1}{2}\begin{pmatrix}
        (I+\bar{K})^{-1}\left(2\bar{K}\Omega \bar{K}^T + \Sigma\right)(I+\bar{K})^{-T} & 0 \\
        0 & \Sigma^{33}_{X, n}
    \end{pmatrix}.
    \]
    To see that \(\Sigma_{Z, n}\) is asymptotically invertible and bounded simply note that the real part of \(\bar{K}_{ii}\) is in \([0, 1]\) for all \(1\le i \le r_2\) and \(\Sigma\) is positive definite. Therefore, the top left block of \(\Sigma_{Z, n}\) is some fixed positive definite matrix for all \(n\in \mathbb{N}\). The result then follows from Lemma \ref{lem: h_lam} in the Appendix.

    Once \eqref{eq: ivx_zz} has been established, the proof of \eqref{eq: ivx_ze} is completely analogous to the proof of equation \eqref{eq: non_stat_app_xx}.
\end{proof}

\begin{proof}[Proof of Theorem \ref{thm: ivx}]
    It follows from Lemma \ref{lem: ivx_seq} and \ref{lem: ivx_asym} in combination with Proposition 8 in the supplementary material for \cite{lundborg2021conditional} that \(\hat{t}^2_{IV, \Gamma}\rightarrow_w \chi^2_{d^2}\) uniformly over \(\Theta\).
\end{proof}

\end{appendix}

%%%%%%%%%%%%%%%%%%%%%%%%%%%%%%%%%%%%%%%%%%%%%%
%% Support information, if any,             %%
%% should be provided in the                %%
%% Acknowledgements section.                %%
%%%%%%%%%%%%%%%%%%%%%%%%%%%%%%%%%%%%%%%%%%%%%%
%\begin{acks}[Acknowledgments]
% The authors would like to thank ...
%\end{acks}
%%%%%%%%%%%%%%%%%%%%%%%%%%%%%%%%%%%%%%%%%%%%%%
%% Funding information, if any,             %%
%% should be provided in the                %%
%% funding section.                         %%
%%%%%%%%%%%%%%%%%%%%%%%%%%%%%%%%%%%%%%%%%%%%%%

\textbf{Funding}. \smallskip The authors gratefully acknowledge financial support from Novo Nordisk Foundation through Grant NNF20OC0062958 and from
Independent Research Fund Denmark | Natural Sciences through Grant 9040-00215B.

%%%%%%%%%%%%%%%%%%%%%%%%%%%%%%%%%%%%%%%%%%%%%%%%%%%%%%%%%%%%%
%%                  The Bibliography                       %%
%%                                                         %%
%%  imsart-???.bst  will be used to                        %%
%%  create a .BBL file for submission.                     %%
%%                                                         %%
%%  Note that the displayed Bibliography will not          %%
%%  necessarily be rendered by Latex exactly as specified  %%
%%  in the online Instructions for Authors.                %%
%%                                                         %%
%%  MR numbers will be added by VTeX.                      %%
%%                                                         %%
%%  Use \cite{...} to cite references in text.             %%
%%                                                         %%
%%%%%%%%%%%%%%%%%%%%%%%%%%%%%%%%%%%%%%%%%%%%%%%%%%%%%%%%%%%%%

%% if your bibliography is in bibtex format, uncomment commands:
\bibliographystyle{plainnat} % Style BST file (imsart-number.bst or imsart-nameyear.bst)
\bibliography{bibliography}       % Bibliography file (usually '*.bib')

%% or include bibliography directly:
% \begin{thebibliography}{}
% \bibitem{b1}
% \end{thebibliography}

\end{document}